\documentclass[10pt]{article}

\usepackage{bm}
\usepackage{amsmath}
\usepackage{graphicx}
\usepackage{epsfig, setspace}
\usepackage{url}
\usepackage{graphicx, transparent, color}

\DeclareSymbolFont{rsfs}{U}{rsfs}{m}{n}
\DeclareSymbolFontAlphabet{\mathcal}{rsfs}

\usepackage[height=21cm]{geometry}

\usepackage{epsfig}
\usepackage{amsmath}
\usepackage{amssymb}
\usepackage{amsfonts}
\usepackage{amsthm}
\usepackage[ruled]{algorithm2e}
\usepackage{algpseudocode}

\numberwithin{equation}{section}

\newtheorem{remark}{Remark}

\newtheorem{example}{Example}[section]

\newtheorem{thm}{Theorem}
\newtheorem{lem}{Lemma}

\newcommand{\R}{\mathbb{R}}

\newcommand{\fb}{{\bf f}}
\newcommand{\gb}{{\bf g}}

\newcommand{\ub}{{\bf u}}

\newcommand{\Vb}{{\bf w}}
\newcommand{\Vvb}{{\bf V}}

\newcommand{\Lb}{{\bf L}}
\newcommand{\Rb}{{\bf R}}
\newcommand{\Bb}{{\bf B}}
\newcommand{\nb}{{\bf n}}

\newcommand{\Cb}{{\bf C}}

\newcommand{\wb}{{\bf w}}

\newcommand{\vb}{{\bf v}}
\newcommand{\Eb}{{\bf E}}
\newcommand{\pb}{{\bf p}}
\newcommand{\qb}{{\bf q}}
\newcommand{\pa} {\partial}
\newcommand{\na} {\nabla}

\newcommand{\eps} {\epsilon}

\newcommand{\hf}{\frac{1}{2}}

\newcommand{\iph}{{i+\hf}}

\newcommand{\jph}{{j+\hf}}

\newcommand{\uad}{\mathbb{U}_{\textrm{ad}}}

\title{Positivity-preserving finite difference WENO scheme for Ten-Moment equations with source term}
 
\date{}
\author{Asha Kumari Meena\footnote{Email: asha@tifrbng.res.in}, \ Rakesh Kumar\footnote{Email: rakesh@tifrbng.res.in}, \ Praveen Chandrashekar\footnote{Email: praveen@tifrbng.res.in, Corresponding author}\\
TIFR Centre for Applicable Mathematics, Bangalore-560065, India}

\begin{document}
\maketitle

\begin{abstract}
We develop a positivity-preserving finite difference WENO scheme for the Ten-Moment equations with body forces acting as a source in the momentum and energy equations. A positive forward Euler scheme under a CFL condition is first constructed which is combined with an operator splitting approach together with an integrating factor, strong stability preserving Runge-Kutta scheme. The positivity of the forward Euler scheme is obtained under a CFL condition by using a scaling type limiter, while the solution of the source operator is performed exactly and is positive without any restriction on the time step. The proposed method can be used with any WENO reconstruction scheme and we demonstrate it with fifth order accurate WENO-JS, WENO-Z and WENO-AO schemes. An adaptive CFL strategy is developed which can be more efficient than the use of reduced CFL for positivity preservation. Numerical results show that high order accuracy and positivity preservation are achieved on a range of test problems.
\end{abstract}	

\section{Introduction}
The Euler equations are used to model compressible flows but there are many situations where classical Euler equations cannot be used~\cite{berthon2015,ber_06a,sangam,mee-kum_17a,mee-kum_17b,meenaKumarFVM}. In case of Euler equations, pressure is a scalar quantity, which comes from the assumption of local thermodynamical equilibrium. But in many situations, local thermodynamical equilibrium is not valid and for such situations, several alternative models are proposed~\cite{levermore1,levermore2}. Ten-Moment equations is one such model which allows anisotropic effects in the flow and produces the symmetric tensorial description of the pressure field.

The Ten-Moment equations form a system of hyperbolic conservation laws \cite{levermore1,levermore2}. The appearance of shock, contact and rarefaction waves~\cite{godlewskiRaviart,toro,godlewski,bouchet} in the solution makes it difficult to develop non-oscillatory higher order numerical schemes for this model. The oscillations can lead to loss of positivity of density and/or the loss of positive definiteness of the pressure tensor, which leads to break down of the simulations. Many numerical schemes have been proposed in the literature based on  different approaches. TVD type limiters are commonly used to control spurious numerical oscillations may not be enough to prevent loss of positivity. Berthon~\cite{ber_06a,berthon2015} designed first order positivity preserving scheme and first order entropy stable schemes. Meena et al.~\cite{mee-kum_17a} proposed a positivity preserving discontinuous Galerkin method and results are presented for degree upto three. Further a second order positivity preserving finite volume scheme is developed in Meena \& Kumar~\cite{meenaKumarFVM} where positivity is achieved by applying appropriate conditions on slopes of reconstruction. An entropy stable finite volume scheme is developed by Sen \& Kumar~\cite{sen-kum_18a}.

The weighted essentially non-oscillatory (WENO) schemes are higher order accurate spatial reconstruction techniques that can be used for numerical solution of hyperbolic conservation laws. They are very often used to simulate compressible flows due to their capability to capture small scale structures with a high resolution and resolving shocks in a non-oscillatory manner.  WENO scheme was developed by Liu et al. \cite{liu-etal_94a} in a finite volume framework, as a convex combination of stencils used in ENO scheme \cite{shu-osh_88a,shu-osh_89a} and have an advantage of one extra order of accuracy over the ENO scheme. Over the years, WENO schemes have witnessed  many developments and improvements \cite{jia-shu_96a,hen-etal_05a,bal-shu_00a,shi-etal_02,bor-etal_08d,ger-etal_09a, bal-etal_16a}. In particular, we refer to some WENO schemes which are relevant to the current work, namely, WENO-JS \cite{jia-shu_96a}, WENO-Z \cite{bor-etal_08d}, WENO-ZQ \cite{zhu-qiu_16a}, and WENO-AO \cite{bal-etal_16a,hua-che_18a, kum-pra_18a, kum-pra_18b}.

High order positivity preserving schemes are built by first constructing a first order scheme which is provably positive under some CFL condition. Usually, upwind type schemes or those based on exact or approximate Riemann solvers like HLL/HLLC~\cite{Batten1997} or relaxation solvers~\cite{Waagan2011},~\cite{Thomann2019} can be proved to be positivity preserving. Another important class of positive schemes are central schemes like Lax-Friedrich or local Lax-Friedrich's which can be proven to be positive for many systems of conservation laws and is also used in the present work within a finite difference WENO setting. High order positivity preserving schemes have been developed in \cite{zha-shu_10d} using a scaling type limiter applied to the reconstruction polynomial. This type of scaling limiter has been very successful in case of other numerical schemes also and has been used with WENO and DG schemes. Positivity preserving finite volume WENO and DG schemes were developed in~\cite{Zha-shu_10a} and finite difference WENO schemes in~\cite{zha-shu_12a, hu-etal_13a}. The other positivity preserving schemes for Euler equations with and without source term can be found in~\cite{Zha-shu_11a,guo-etal_14a}.

In the present work, we have developed a positivity preserving finite difference WENO scheme for the Ten-Moment equations with source term originating from body forces. The scheme consists of two main ingredients due to the operator splitting approach employed to deal with the source terms: (a) a positive finite difference WENO scheme for the hyperbolic part using a scaling type limiter, and, (2) an integrating factor strong stability preserving Runge-Kutta scheme to include the effect of the source term. The positivity property of the present algorithm is independent of the type of WENO reconstruction and order of WENO scheme. To demonstrate the idea of positivity preservation, we have considered three fifth order WENO reconstruction schemes: WENO-JS \cite{jia-shu_96a}, WENO-Z \cite{bor-etal_08d} and WENO-AO \cite{bal-etal_16a}. Our analysis to achieve the positivity preservation of the scheme is similar to the approach discussed in~\cite{zha-shu_12a} where the flux reconstruction polynomial is modified with a scaling limiter based on the ideas of Zhang and Shu. We use local Lax-Friedrich splitting of the flux and prove the positivity of the splitting for the Ten-Moment equations. The forward Euler scheme together with WENO reconstruction is written as a sum of two terms originating from the flux splitting. A scaling type limiter is used to modify the numerical fluxes so that the two parts are each positive under a CFL condition.

The source term in the Ten-Moment equations originating from the body forces is linear in the conserved variables which allows us to solve the source ODE exactly in case of time independent sources and approximately if the body force potential is time dependent and cannot be integrated exactly. In either case, we can show that the solution of the source ODE is positivity preserving without any restriction on the time step, and this holds both forward and backward in time. The complete update is achieved by using the integrating factor, strong stability preserving Runge-Kutta schemes developed in~\cite{lea-etal_18a} and we demonstrate this approach using third order time integration.

The resulting algorithm is provably positivity preserving under a time step restriction which is $\alpha^x \Delta t/\Delta x \le \frac{1}{12}$ in 1-D and $(\alpha^x/\Delta x + \alpha^y/\Delta y)\Delta t \le \frac{1}{12}$ in 2-D, where $\alpha^x, \alpha^y$ are the maximum wave speeds occuring in the whole mesh along the $x,y$ directions, respectively. The CFL of $1/12$ is a consequence of the fifth order WENO reconstruction scheme that we employ for the hyperbolic part of the model. The use of such a small CFL can lead to small time steps and consequently increased computational cost. In order to improve this situation, we propose and test an adaptive strategy where the finite difference scheme with CFL $\approx 1$ is first used to update the solution and if positivity is lost, then we repeat the time update with the reduced CFL of $1/12$ together with the positivity limiter. We show in some numerical experiments that this strategy is faster than using the reduced CFL in all the time steps. Numerical tests show that the proposed scheme give stable computations which would break down if positivity limiter is not used.

The rest of the article is organized as follows. In Section~(\ref{sec:tmf}), we introduce the Ten-Moment Gaussian closure equations with source term and discuss its hyperbolicity and positivity of the source operator. Section~(\ref{sec:tms}) gives the main time stepping algorithm based on the integrating factor approach. The WENO scheme in 1-D is introduced in Section~(\ref{sec:weno1d}) and its positivity property under a scaling limiter is eastablished. The corresponding scheme in 2-D is introduced in Section~(\ref{sec:weno2d}). We then explain the algorithm including the adaptive CFL approach in Section~(\ref{sec:impl}). A comprehensive set of numerical tests is performed in Section~(\ref{sec:nm}) to show order of accuracy, non-oscillatory property and positivity preservation. Finally, the paper ends with a summary and conclusions.

\section{Ten-moment equations with source term}
\label{sec:tmf}
We consider the Ten-Moment equations with source term \cite{lev_96a,lev-wil_98a,ber-etal_15a,ber_06a} which is a system of balance laws given as,
\begin{subequations}
	\label{eq:10momeq}
	\begin{eqnarray}
	\label{eq:masscons}
	\pa_{t}\rho + \nabla\cdot(\rho \vb) &=& 0\\ 
	\label{eq:momcons}
	\pa_{t}(\rho \vb) + \nabla\cdot(\rho \vb\otimes \vb + \pb) &=& -\frac{1}{2}\rho\na W\\
	\label{eq:energycons}
	\pa_{t}\Eb + \nabla\cdot((\Eb+\pb)\otimes \vb)^s &=& -\frac{1}{4}\rho (\na W \otimes \vb + \vb\otimes\na W)
	\end{eqnarray} 
\end{subequations}
where $\rho$, $\vb$, and $\Eb$ represent the density, velocity and symmetric energy tensor and the symbol $\otimes$ denotes the tensor outer product and the superscript $(\cdot)^s$ denotes the symmetric tensor outer product~\cite{sangam}. The variable $\pb$ is the symmetric pressure tensor and is related to the energy tensor by the equation of state,
\begin{equation*}
\label{eq:eqnstate}
\Eb=\frac{1}{2}\left(\pb +\rho\vb\otimes\vb \right)
\end{equation*}
and $W(x,y,t)$ is a given potential function whose gradient gives an external body force term. In two-dimensions, the system \eqref{eq:10momeq} can be written as,
\begin{equation}
\label{eq:balance_law}
\pa_t \ub +\pa_x \fb(\ub) + \pa_{y} \gb(\ub) = {\bf s}(\ub)
\end{equation}
with the {\em conservative variable} $\ub=\{\rho,\rho v_{1},\rho v_2,E_{11},E_{12},E_{22}\}^{\top}$, the fluxes

\begin{equation*}
\label{eq:flux}
\fb(\ub) = \left\{\begin{array}{c}
\rho v_1 \\ 
\rho v_1^2 + p_{11} \\
\rho v_{1}v_2 + p_{12} \\
(E_{11} + p_{11})v_1 \\
E_{12}v_{1} + \frac{1}{2}(p_{11}v_2 + p_{12}v_1) \\
E_{22}v_1 + p_{12}v_2
\end{array} \right\},  \qquad
\gb(\ub) = \left\{ \begin{array}{c}
\rho v_2 \\
\rho v_1 v_2 + p_{12} \\ 
\rho v_2^{2} + p_{22} \\
E_{11}v_{2} + p_{12}v_{1} \\ 
E_{12}v_{2} +\frac{1}{2}( p_{12}v_2 + p_{22}v_{1})\\
(E_{22} + p_{22})v_{2}
\end{array}
\right\}
\end{equation*}
and the source,
\begin{equation*}
\label{eq:source}
{\bf s}(\ub) = \left\{\begin{array}{c}
0 \\ 
-\frac{1}{2}\rho \partial_{x} W \\
-\frac{1}{2}\rho \partial_{y} W \\
-\frac{1}{2}\rho v_{1} \partial_{x} W \\
-\frac{1}{4}\rho v_{2} \partial_{x} W -\frac{1}{4}\rho v_{1} \partial_{y} W \\
-\frac{1}{2}\rho v_{2} \partial_{y} W
\end{array} \right\}
\end{equation*}
Here $v_{1}$ and $v_{2}$ are components of two-dimensional velocity $\vb$; $E_{11}$, $E_{12}$ and $E_{22}$ are components of two-dimensional symmetric energy tensor $\Eb$ and $p_{11}$, $p_{12}$ and $p_{22}$ are components of two-dimensional symmetric pressure tensor $\pb$. We aim to find the solutions in the following convex set of physically admissible states,
\begin{equation}
\uad = \left\{ \ub\in\R^6 \ | \ \rho(\ub) >0 \; \mbox{ and } \; {\bf x}^\top \pb(\ub) \: {\bf x} > 0, \; \forall \; {\bf x} \in \R^2 \setminus \lbrace \bm{0} \rbrace \right\}
\label{eq:sol_set}
\end{equation}
which contains the states $\ub$ with positive density and positive definite pressure tensor. The positive definiteness of $\pb$ implies that $p_{11}+p_{22}>0$ and $\det\pb=p_{11}p_{22}-p_{12}^{2}>0$. As a consequence, we also have the condition that $p_{11}>0$ and $p_{22}>0$. From \cite{mee-kum_17a}, we have the following result for the hyperbolicity of the system.

\begin{lem} For $\ub \in \uad$, the system \eqref{eq:10momeq} without source term is hyperbolic. The eigenvalues for the system are given as,
	$$
	\vb \cdot \nb, \quad \vb \cdot \nb \pm \sqrt{\frac{3(\pb \cdot \nb) \cdot \nb}{\rho}}, \quad \vb \cdot \nb \pm \sqrt{\frac{(\pb \cdot \nb) \cdot \nb}{\rho}}
	$$
	along the unit vector $\nb$.
	The first eigenvalue has two order of multiplicity and it is associated to a linearly degenerate field. The second eigenvalue has one order of multiplicity and is associated to a genuinely nonlinear field. The third eigenvalue has one order of multiplicity and is associated to a linearly degenerate field.
\end{lem}
\begin{remark}
For later use in the numerical schemes, we define the maximum wave speeds for the Ten-Moment model along the two coordinate directions by
\[
\alpha^{x}(\ub)= |v_{1}| + \sqrt{\frac{3 p_{11}}{\rho}}, \qquad
\alpha^{y}(\ub)=|v_{2}| + \sqrt{\frac{3 p_{22}}{\rho}}
\]
\end{remark}
\begin{remark}
We will refer to the set of variables $\rho, \vb, \pb$ as {\em primitive variables}. We can of course uniquely convert between primitive and conserved variables.
\end{remark}
Our numerical strategy for including the source term in the numerical scheme is based on an operator splitting scheme. The Lemma below shows that the source term alone leads to a positive solution. We assume that the spatial location is fixed and hence the space coordinate is not specified in the ODE system.
\begin{lem}
\label{lem:pos_ExactSource}
The solution of the source ODE 
\[
\frac{d \ub}{d t} = {\bf s}(\ub)
\]
is positivity preserving, i.e.  $\ub(t) \in \uad$ at any time $t$ implies $\ub(t+\tau) \in \uad$ $\forall$ $\tau \in \mathbb{R}$.
	\begin{proof}{\rm 
		We write the source ODE in the following form,
		\[
		\frac{d \ub}{d t} = {\bf s}(\ub) = {\bf B}(t) \ub
		\]	
		where ${\bf B}(t)$ is the following matrix,
		\begin{equation*}
		\label{eq:source_matrix}
		{\bf B}(t) = \left[ \begin{array}{cccccc}
		0 & 0 & 0 & 0 & 0 & 0\\
		a(t) & 0 & 0 & 0 & 0 & 0 \\
		b(t) & 0 & 0 & 0 & 0 & 0 \\
		0 & a(t) & 0 & 0 & 0 & 0 \\
		0 & \frac{b(t)}{2} & \frac{a(t)}{2} & 0 & 0 & 0 \\
		0 & 0 & b(t) & 0 & 0 & 0
		\end{array} \right]
		\end{equation*}
		with $a(t) = -\frac{1}{2} \pa_{x}W$ and $b(t) = -\frac{1}{2} \pa_{y}W$.
		For given $\ub(t)\in \uad$, solution of the ODE at time $t+\tau$ is the given by,
		\begin{equation}
		\label{eq:exact_sol}
		\ub(t+\tau) = e^{\Cb(t, \tau)} \ub(t)
		\end{equation}
		with 
		\[
		\Cb(t,\tau) = \int_{t}^{t+\tau} \Bb(s) ds
		\]
		and
		\begin{equation*}
		\label{eq:expC}
		e^{\Cb(t,\tau)} = \left[ \begin{array}{c c c c c c}
		1 & 0 & 0 & 0 & 0 & 0 \\
		\hat{a}(t,\tau) & 1 & 0 & 0 & 0 & 0 \\
		\hat{b}(t,\tau) & 0 & 1 & 0 & 0 & 0 \\
		\frac{(\hat{a}(t,\tau))^{2}}{2} & \hat{a}(t,\tau) & 0 & 1 & 0 & 0 \\
		\frac{\hat{a}(t,\tau) \hat{b}(t,\tau)}{2} & \frac{\hat{b}(t,\tau)}{2} & \frac{\hat{a}(t,\tau)}{2} & 0 & 1 & 0 \\
		\frac{(\hat{b}(t,\tau))^{2}}{2} & 0 & \hat{b}(t,\tau) & 0 & 0 & 1
		\end{array} \right]
		\end{equation*}
		where $\hat{a}(t,\tau) = \int_{t}^{t+\tau}a(s)ds$ and $\hat{b}(t,\tau) = \int_{t}^{t+\tau}b(s)ds$. The equation \eqref{eq:exact_sol} implies,
		\begin{eqnarray*}
			\rho(t+\tau) & = & \rho(t),\\
			(\rho v_{1})(t+\tau) & = & \rho(t) \left(\hat{a}(t,\tau) + v_{1}(t) \right),\\
			(\rho v_{2})(t+\tau) & = & \rho(t) \left(\hat{b}(t,\tau) + v_{2}(t) \right),\\
			E_{11}(t+\tau) & = & \rho(t) \left( \frac{(\hat{a}(t,\tau))^{2}}{2} + v_{1}(t) \hat{a}(t,\tau) \right) + E_{11}(t),\\
			E_{12}(t+\tau) & = & \rho(t) \left( \frac{\hat{a}(t,\tau) \hat{b}(t,\tau)}{2} + v_{1}(t) \frac{\hat{b}(t,\tau)}{2} +  v_{2}(t) \frac{\hat{a}(t,\tau)}{2} \right) + E_{12}(t),\\
			E_{22}(t+\tau) & = & \rho(t) \left( \frac{(\hat{b}(t,\tau))^{2}}{2} + v_{2}(t) \hat{b}(t,\tau) \right) + E_{22}(t)
		\end{eqnarray*}
		Then the velocity components are given by,
		\begin{eqnarray}
		\nonumber
		 v_{1}(t+\tau) = \hat{a}(t,\tau) + v_{1}(t),~~~
		 v_{2}(t+\tau) = \hat{b}(t,\tau) + v_{2}(t),
		\end{eqnarray}
		and the pressure components are given as,
		\begin{align}
		\nonumber
		p_{11}(t+\tau) & =  2 E_{11}(t+\tau) - \rho(t+\tau) [ v_{1}(t+\tau) ]^{2} \\
		\nonumber
		& =\rho(t) [\hat{a}(t,\tau)]^{2} + 2 \rho(t) v_{1}(t) \hat{a}(t,\tau) + 2 E_{11}(t) - \rho(t) [\hat{a}(t,\tau) + v_{1}(t)]^{2} \\
		\nonumber
		& =  2 E_{11}(t) - \rho(t) [v_{1}(t)]^{2} \\
		\nonumber
		& = p_{11}(t)
		\end{align}
		Similarly, it can be shown that $p_{12}(t+\tau)=p_{12}(t)$ and $p_{22}(t+\tau)=p_{22}(t)$.
		We see that the density and pressure solution at time $t+\tau$ coincide with the corresponding solutions at time $t$, unconditionally (without any condition on $W$). Hence, we can conclude that for any given potential function $W(x,y,t)$, the solution of the source ODE is positivity preserving.} 
	\end{proof}
\end{lem}
\begin{remark}
The source operator does not change the density and pressure tensor which means that an initial positive solution remains positive at all times under the action of the source term alone. The positivity property of the source term holds both forward and backward in time. This allows us to use standard SSPRK schemes to achieve the positivity property of the entire scheme including the hyperblic part of the model. The proof also holds if we replace the integrals in the definition of $\hat{a}$, $\hat{b}$ with some quadrature rule in case the integrals cannot be analytically evaluated.
\end{remark}
\section{Discretization with source term} \label{sec:tms}
Let us assume that we have discretized the flux divergence term in the balance law by a finite difference scheme leading to the system of ordinary differential equations which can be written in the form
\begin{equation*}
\frac{d \ub}{d t} = \Rb(\ub)+ \Bb(t) \ub,
\end{equation*}
Since the source term is linear in the unknown solution $\ub$ we can use an integrating factor approach~\cite{lea-etal_18a}. Using an exponential integrating factor, we can transform this ODE to
\[
\frac{d}{dt} e^{-\int_{t_n}^t \Bb(s)ds} \ub =  e^{-\int_{t_n}^t \Bb(s)ds} \Rb(\ub)
\]
We can now apply the standard third order accurate strong stability preserving scheme to the above ODE leading to the following set of update equations
\begin{subequations}
	\label{eq:expssprk3a}
	\begin{align}
	\ub^{(1)} & = e^{\Cb(t_{n},\Delta t)} [\ub^{n} + \Delta t \; \Rb(\ub^{n})] \\
	\ub^{(2)} & = \frac{3}{4} e^{\Cb(t_{n},\frac{\Delta t}{2})} \ub^{n} + \frac{1}{4} e^{-\Cb(t_{n}+\frac{\Delta t}{2},\frac{\Delta t}{2})}[\ub^{(1)} + \Delta t \; \Rb(\ub^{(1)})] \\
	\ub^{n+1} & = \frac{1}{3} e^{\Cb(t_{n},\Delta t)} \ub^{n} + \frac{2}{3} e^{\Cb(t_{n}+\frac{\Delta t}{2},\frac{\Delta t}{2})}[\ub^{(2)} + \Delta t \; \Rb(\ub^{(2)})]
	\end{align}
\end{subequations}
and the expressions for the matrices $\Cb$ and the exponential matrix have already been given in the previous sections. If the matrix $\Bb$ does not depend on time, the integrating factor SSPRK3 scheme takes the following form,
\begin{subequations}
	\label{eq:expssprk3b}
	\begin{align}
	\ub^{1} & = e^{\Bb \Delta t} [\ub^{n} + \Delta t \; \Rb(\ub^{n})] \\
	\ub^{2} & = \frac{3}{4} e^{\frac{1}{2} \Bb \Delta t}\ub^{n} + \frac{1}{4} e^{-\frac{1}{2}\Bb \Delta t}[\ub^{1} + \Delta t \; \Rb(\ub^{1})] \\
	\ub^{n+1} & = \frac{1}{3} e^{\Bb \Delta t} \ub^{n} + \frac{2}{3} e^{\frac{1}{2} \Bb \Delta t}[\ub^{2} + \Delta t \; \Rb(\ub^{2})]
	\end{align}
\end{subequations}
If there is no source term in the model, then the exponential factors must be replaced with an identity, leading to the standard SSPRK3 scheme. We observe from the above update equations that it is composed of a forward Euler scheme in a convex combination with some exponential factors. The exponential factors provide an update to the source ODE and we have already proven that this step is positivity preserving under any time step $\Delta t$. One of the exponential factors in the second stage has a negative sign in the exponent which means that we are evolving the source ODE backward in time. However, the solution of the source ODE is shown to be positive for both forward and backward in time integration. It therefore remains to ensure that the forward Euler scheme is positivity preserving which then ensures that the third order update scheme is also positivity preserving.
\begin{thm}
Assume that the forward Euler scheme is positivity preserving, i.e.,
\[
\ub \in \uad \quad\implies\quad \ub + \Delta t \Rb(\ub) \in \uad \qquad \textrm{for any} \quad \Delta t \le \Delta t^*(\ub)
\]
Then the third order SSPRK scheme \eqref{eq:expssprk3a} or \eqref{eq:expssprk3b} is positivity preserving under the same restriction on the time step.
\end{thm}
The remaining part of the paper will be concerned with constructing a finite difference scheme such that the forward Euler scheme is positive. Since the multi-dimensional scheme is based on a combination of one dimensional schemes, we start by considering the one dimensional scheme in more detail.
\section{1-D finite difference WENO scheme}
\label{sec:weno1d}
We discretize the domain into a uniform mesh with nodes $x_{i}$ and define $x_{i+\frac{1}{2}}=\frac{1}{2}(x_{i+1}+x_{i})$ which are the location of the cell faces, and the cells $I_{i}=[x_{i-\frac{1}{2}}, x_{i+\frac{1}{2}}]$ are of size $\Delta x=x_{i+1}-x_{i}$. A conservative finite difference scheme for one-dimensional model,
\begin{equation*}
\label{eq:con_law_1d}
\pa_{t} \ub + \pa_{x} \fb(\ub) = 0
\end{equation*} 
is given by the forward Euler scheme
\begin{equation}
\label{eq:fdm1d}
\ub_{i}^{n+1} = \ub_{i}^{n} - \lambda (\fb^n_{i+\frac{1}{2}} - \fb^n_{i-\frac{1}{2}})
\end{equation}
where $\lambda = \frac{\Delta t }{\Delta x}$ with time step $\Delta t$ and $\fb_{i + \frac{1}{2}}$ is the numerical flux defined at the interface $x_{i+\frac{
		1}{2}}$. It is enough to consider the positivity of the above forward Euler scheme, since higher order accuracy in time is achieved by a SSPRK scheme which involves convex combinations of forward Euler scheme.

To ensure the stability of the numerical scheme for hyperbolic PDEs, we first split the flux into positive and negative parts,
\[
\fb(\ub) = \fb^{+}(\ub) + \fb^{-}(\ub),
\]
such that $\frac{d \fb^{+}(\ub)}{d \ub} \geq 0$ and $\frac{d \fb^{-}(\ub)}{d \ub} \leq 0$. The simplest splitting is a Lax-Friedrich splitting of the form,
\begin{equation*}
\fb^{\pm}(\ub) = \frac{1}{2} \left(\fb(\ub) \pm \alpha \ub \right)
\end{equation*}
where $\alpha$ is an estimate of largest wave speed. In practice, we use a local Lax-Friedrich splitting where $\alpha$ is estimated using some local data around each face as explained below.
The WENO reconstruction can be performed directly over split fluxes $\fb^\pm$ or after performing a transformation to characteristic variables. We perform it via characteristic variables since this is known to give better control on the oscillations.

Let us fix the cell face located at $x_{i+\frac{1}{2}}$ and briefly explain the WENO procedure to compute $\fb_\iph$. We consider the two states $\ub_{i}$ and $\ub_{i+1}$, and compute their average $\bar{\ub}_{i+\frac{1}{2}} = \frac{1}{2}(\ub_{i} + \ub_{i+1})$. We denote the matrix of right eigenvectors and matrix of left eigenvectors of the flux Jacobian $\frac{\pa \fb}{\pa \ub}(\bar{\ub}_{i+\frac{1}{2}}^{n})$ by $\Rb_{i+\frac{1}{2}}$ and $\Lb_{i+\frac{1}{2}}$.  The expressions for the eigenvectors can be found in \cite{don-gro_05a}. For fifth order WENO reconstruction, we need a five point stencil. So, for reconstruction of $\fb^{+}_{i+\frac{1}{2}}$, we need following values,
\begin{equation}
\label{eq:fp_cell}
\fb_{j}^{+} = \frac{1}{2}\left( \fb(\ub_{j}) + \alpha_{i+\frac{1}{2}} \ub_{j} \right), \quad j = i-2, \ldots, i+2
\end{equation}
and for reconstruction of $\fb_{i+\frac{1}{2}}^{-}$, we need,
\begin{equation}
\label{eq:fm_cell}
\fb_{j}^{-} = \frac{1}{2}\left( \fb(\ub_{j}) - \alpha_{i+\frac{1}{2}} \ub_{j} \right), \quad j = i-1, \ldots, i+3
\end{equation}
where 
\begin{equation}
\alpha_{i+\frac{1}{2}}=\max\{ \alpha(\ub_{i}), \alpha(\ub_{i+1})\}, \qquad \alpha(\ub) = \alpha^x(\ub) = |v_1| + \sqrt{ \frac{3 p_{11}}{\rho} }
\label{eq:alpha1d}
\end{equation}
The quantities $\fb^{\pm}_{j}$ for all given $j$ in \eqref{eq:fp_cell} and \eqref{eq:fm_cell} are transformed to the local characteristic field as,
\[
{\bf F}_{j}^{\pm} = \Lb_{i+\frac{1}{2}} \fb_{j}^{\pm}
\]
Using these local characteristic states, we perform the WENO reconstruction \cite{zha-shu_12a} and denote the reconstructed values by ${\bf F}_{i+\frac{1}{2}}^{\pm}$. The reconstructed values are transformed back to the conserved variable by multiplying with the matrix of right eigenvectors as,
\[
\fb_{i+\frac{1}{2}}^{\pm} = \Rb_{i+\frac{1}{2}} {\bf F}_{i+\frac{1}{2}}^{\pm}
\]
Finally, the numerical flux is given by
\[
\fb_\iph = \fb_\iph^- + \fb_\iph^+
\]
In order to establish the positivity property, we first show some preliminary results and then discuss the positivity limiter.
\subsection{Positivity property}
We introduce the new variables inspired by the Lax-Friedrich splitting \cite{zha-shu_12a} as,
\begin{equation}
\label{eq:w_pm}
\Vb^{\pm}(\ub) = \frac{1}{2}\left( \ub \pm \frac{\fb(\ub)}{\alpha(\ub)}
\right)
\end{equation}
which provides a splitting of the conserved variables since
\begin{equation*}
\Vb^{+}(\ub) + \Vb^{-}(\ub) = \ub
\end{equation*}
The definition of $\alpha(\ub)$ is given in \eqref{eq:alpha1d}. We have the following positivity result for $\Vb^{\pm}(\ub)$. 
\begin{thm}
	\label{thm:positivity1}
	If $\ub \in \uad$, then
	the states $\Vb^{\pm}(\ub) \in \uad$.
	\begin{proof}{\rm
		From the definition of $\Vb^\pm$, we have,
		\begin{equation*}
		2 \Vb^{\pm}(\ub) = \left( \begin{array}{c}
		\rho \pm \frac{1}{\alpha} \rho v_{1} \\
		\rho v_{1} \pm \frac{1}{\alpha} (\rho v_{1}^{2} + p_{11}) \\
		\rho v_{2} \pm \frac{1}{\alpha} (\rho v_{1} v_{2} + p_{12}) \\
		E_{11} \pm \frac{1}{\alpha}(E_{11} + p_{11}) v_{1} \\
		E_{12} \pm \frac{1}{\alpha} (E_{12} v_{1} + \frac{1}{2}(p_{11} v_{2} + p_{12} v_{1})) \\
		E_{22} \pm \frac{1}{\alpha} (E_{22} v_{1} + p_{12} v_{2})
		\end{array} \right) := \Vvb^\pm = \left( \begin{array}{c}
		V_{1}^{\pm} \\ V_{2}^{\pm} \\ V_{3}^{\pm} \\ V_{4}^{\pm} \\ V_{5}^{\pm} \\ V_{6}^{\pm}
		\end{array} \right)
		\end{equation*}
		We show that $\Vb^{+}\in \uad$ and the proof for $\Vb^-$ can be shown similarly.  Let us denote the primitive variables corresponding to $\Vvb^{+}$ as, $$(\rho(\Vvb^{+}), v_{1}(\Vvb^{+}), v_{2}(\Vvb^{+}), p_{11}(\Vvb^{+}), p_{12}(\Vvb^{+}), p_{22}(\Vvb^{+}))^{\top}$$ 
For positivity of density, we have,
		\[
		\rho(\Vvb^{+}) =  V_{1}^{+} =  \left( 1 + \frac{v_{1}}{\alpha} \right) \rho > 0
		\]
since
\begin{equation}
v_1 + \alpha = v_1 + |v_{1}| + \sqrt{\frac{3 p_{11}}{\rho}} \ge \sqrt{\frac{3 p_{11}}{\rho}} > 0
\label{eq:ineq1}
\end{equation}
		For positivity of symmetric pressure tensor, first we show the positivity of $p_{11}(\Vvb^{+})$ as follows,  
		\begin{eqnarray*}
			\rho(\Vvb^{+}) \: p_{11}(\Vvb^{+}) & = & 2 V_{1}^{+} V_{4}^{+} - (V_{2}^{+})^{2} \\ & = & 2 \left( \rho + \frac{1}{\alpha} \rho v_{1} \right) \left( E_{11} + \frac{1}{\alpha}(E_{11} + p_{11}) v_{1} \right) - \left( \rho v_{1} + \frac{1}{\alpha} (\rho v_{1}^{2} + p_{11}) \right)^{2} \\
			& = & \frac{p_{11}}{\alpha^{2}} \left(\rho (v_{1}+ \alpha)^{2} - p_{11} \right) \\
		\end{eqnarray*}
But using \eqref{eq:ineq1}, we have
\begin{equation}
v_1 + \alpha \ge \sqrt{\frac{3 p_{11}}{\rho}} \quad\implies\quad \rho (v_{1}+ \alpha)^{2} - p_{11} \ge 2 p_{11}
\label{eq:ineq2}
\end{equation}
so that
\[
			\rho(\Vvb^{+}) \: p_{11}(\Vvb^{+}) 
			 \geq   \frac{2 p_{11}^{2}}{\alpha^{2}} > 0
\]
		Hence $p_{11}(\Vvb^{+})> 0$ since $\rho(\Vvb^{+}) >0$. Let us introduce the notation $\det\pb(\Vvb^{+})$ for determinant of pressure tensor $p_{11}(\Vvb^{+})p_{22}(\Vvb^{+})-(p_{12}(\Vvb^{+}))^{2}$. Now, we show its positivity as follows,
		{\small \begin{eqnarray*}
				\left( \rho(\Vvb^{+}) \right)^{2} \det\pb(\Vvb^{+}) & = & (2 V_{1}^{+} V_{4}^{+} - (V_{2}^{+})^{2}) (2 V_{1}^{+} V_{6}^{+} - (V_{3}^{+})^{2}) - (2 V_{1}^{+} V_{5}^{+} - (V_{2}^{+} V_{3}^{+}))^{2} \\ 
				& = & \frac{(p_{11} p_{22}-p_{12}^{2}) \rho (v_{1}+\alpha)^{2}(\rho(v_{1}+\alpha)^{2}-p_{11})}{\alpha^{4}} \\
				&  \geq & \frac{(p_{11} p_{22}-p_{12}^{2}) \rho (v_{1}+\alpha)^{2} (2 p_{11})}{\alpha^{4}}, \quad \textrm{using \eqref{eq:ineq2}}\\
				& > & 0, \quad \textrm{since each term is strictly $>0$}
		\end{eqnarray*} }
		This implies $\det\pb(\Vvb^{+})>0$. Hence we can conclude that $\Vb^{+}(\ub) \in \uad$ for any $\ub \in \uad$.
}
	\end{proof}
\end{thm} 
\subsection{Sufficient condition for positivity}
The key idea is now to make use of the positivity of $\Vb^\pm(\ub_i^n)$ to obtain the positivity of the finite difference scheme under a forward Euler update.  For simplicity of notation we drop the superscript $n$ denoting the current time level in some of the equations below. At the interface $x_{i+\frac{1}{2}}$ let us define
\begin{equation*}
\Vb_{i+\frac{1}{2}}^{\pm} = \pm \frac{\fb^{\pm}_{i+\frac{1}{2}}}{\alpha_{i+\frac{1}{2}}}
\end{equation*}
where $\alpha_\iph$ is the estimate of the local wave speed used in the flux splitting of the WENO scheme as defined in equation~(\ref{eq:alpha1d}). In terms of these variables, the numerical flux can be written as,
\begin{eqnarray*}
\nonumber
\fb_{i+\frac{1}{2}} 
= \alpha_{i+\frac{1}{2}} (\Vb^{+}_{i+\frac{1}{2}} - \Vb^{-}_{i+\frac{1}{2}})
\end{eqnarray*}
Then the finite difference scheme \eqref{eq:fdm1d} can be written as a sum of two terms
\begin{equation*}
\ub_{i}^{n+1} = \ub^{n+1,+}_{i} + \ub^{n+1,-}_{i}
\end{equation*}
where
\[
\ub_{i}^{n+1,+} = \Vb^{+}_{i} - \lambda \alpha_{i+\frac{1}{2}} \Vb^{+}_{i+\frac{1}{2}} + \lambda \alpha_{i-\frac{1}{2}} \Vb^{+}_{i-\frac{1}{2}}, \qquad
\ub_{i}^{n+1,-} = \Vb^{ -}_{i} + \lambda \alpha_{i+\frac{1}{2}} \Vb^{-}_{i+\frac{1}{2}} - \lambda \alpha_{i-\frac{1}{2}} \Vb^{-}_{i-\frac{1}{2}}
\]
which are two finite volume-type schemes. For positivity, we have to achieve $\ub_{i}^{n+1,\pm}\in \uad$. First we consider the state $\ub_{i}^{n+1,+}$. Let us introduce a $4^{th}$ degree polynomial $\qb_{i}^{+}(x)$, such that $\qb_{i}^{+}(x_{i+\frac{1}{2}}) = \Vb_{i+\frac{1}{2}}^{+}$ and whose cell average coincides with $\Vb_{i}^{+}$, i.e.
\[
\Vb_{i}^{+} = \frac{1}{\Delta x} \int_{I_{i}} \qb_{i}^{+}(x) \: dx
\]
We compute the integral involved in the cell average by Gauss-Lobatto quadrature rule of $N$ points as,
\[
\Vb_{i}^{+} = \sum_{\alpha=1}^{N} \hat{w}_{\alpha} \qb_{i}^{+} (\hat{x}_{i}^{\alpha}) 
 =  \sum_{\alpha=1}^{N-1} \hat{w}_{\alpha} \qb_{i}^{+} (\hat{x}_{i}^{\alpha}) + \hat{w}_{N} \Vb^{+}_{i+\frac{1}{2}}
= (1-\hat{w}_{N}) \qb_{i}^{+,*} + \hat{w}_{N} \Vb^{+}_{i+\frac{1}{2}}
\]
where we have defined 
\[
\qb_{i}^{+,*} = \frac{1}{1-\hat{w}_{N}} \sum_{\alpha=1}^{N-1} \hat{w}_{\alpha} \qb_{i}^{+}(\hat{x}_{i}^{\alpha})
\]
So, we have,
\begin{equation*}
\ub_{i}^{n+1,+} = (1-\hat{w}_{N}) \qb_{i}^{+,*} + (\hat{w}_{N} - \lambda \alpha_{i+\frac{1}{2}}) \Vb^{+}_{i+\frac{1}{2}} + \lambda \alpha_{i-\frac{1}{2}} \: \Vb^{+}_{i-\frac{1}{2}}
\end{equation*}
Similarly, we can write,
\begin{equation*}
\ub_{i}^{n+1,-} = (1-\hat{w}_{1}) \qb_{i}^{-,*} + (\hat{w}_{1} - \lambda \alpha_{i-\frac{1}{2}}) \Vb^{-}_{i-\frac{1}{2}} + \lambda \alpha_{i+\frac{1}{2}} \: \Vb^{-}_{i+\frac{1}{2}}
\end{equation*}
where 
\[
\qb_{i}^{-,*} = \frac{1}{1-\hat{w}_{1}} \sum_{\alpha=2}^{N} \hat{w}_{\alpha} \qb_{i}^{-}(\hat{x}_{i}^{\alpha})
\]

\begin{thm}
\label{thm:pos1d}
Assume that the following conditions hold for all $i$.
\begin{enumerate}
\item $\ub_{i}^{n} \in \uad$
\item $\left \lbrace \qb_{i}^{\pm,*}, \Vb^{\pm}_{i+\frac{1}{2}} \right \rbrace \in \uad$
\end{enumerate}
where 
\begin{equation}
\label{eq:qstar}
\qb_{i}^{+,*} = \frac{1}{1-\hat{w}_{N}}(\Vb_{i}^{+} - \hat{w}_{N} \Vb_{i + \frac{1}{2}}^{+}), \qquad
\qb_{i}^{-,*} = \frac{1}{1-\hat{w}_{1}} (\Vb_{i}^{-} - \hat{w}_{1} \Vb_{i- \frac{1}{2}}^{-})
\end{equation}
The finite difference WENO scheme \eqref{eq:fdm1d} is positivity-preserving if the CFL condition
\[
\frac{\alpha \Delta t}{\Delta x} \le \hat{w}_1
\]
holds, i.e., $\ub_i^{n+1} \in \uad$.
\end{thm}
\begin{proof}
The proof follows easily since we have written the update scheme as a convex combination of certain quantities, all of which are in the admissible set by the assumptions in the theorem. The CFL condition ensures that the coefficients are positive which is required to have a convex combination.
\end{proof}
\begin{remark}
In the computer implementation of the numerical scheme, we do not have to construct the polynomials $\qb_i^\pm(x)$ since we only require $\qb_i^{\pm,*}$ and these quantities can be obtained from~(\ref{eq:qstar}). The quadrature rule used for $\wb_i^\pm$ must be exact for a polynomial of degree four which requires the use of $N \ge 4$ point Gauss-Lobatto rule. For the $N=4$ point Gauss-Lobatto rule, we have the weight $\hat{w}_1 = \frac{1}{12}$ which is the CFL number required for the proof of positivity.
\end{remark}
\subsection{Positivity limiter}
The sufficient condition $(1)$ in Theorem \eqref{thm:pos1d} can be assumed directly and we have to enforce condition (2) by a positivity limiter. For given $\ub_{i}^{n} \in  \uad$, the Theorem \eqref{thm:positivity1} implies $\Vb_{i}^{\pm} \in \uad$. 
Let us say,
\begin{eqnarray*}
\Vb_{i}^{+} &=& (\rho_{i}^{+}, m_{1,i}^{+}, m_{2,i}^{+}, E_{11,i}^{+}, E_{12,i}^{+}, E_{22,i}^{+})^{\top} \\
\Vb_{i+\frac{1}{2}}^{+} &=& (\rho_{i+\frac{1}{2}}^{+}, m_{1,i+\frac{1}{2}}^{+}, m_{2,i+\frac{1}{2}}^{+},
E_{11, i+\frac{1}{2}}^{+},
E_{12, i+\frac{1}{2}}^{+},
E_{22, i+\frac{1}{2}}^{+})^{\top} \\
\qb_{i}^{+,*} &=& (\rho_{i}^{*}, m_{1,i}^{*}, m_{2,i}^{*}, E_{11,i}^{*}, E_{12,i}^{*}, E_{22,i}^{*})^{\top}
\end{eqnarray*}
To achieve the sufficient condition (2) of Theorem \eqref{thm:pos1d}, we modify $\Vb_{i+\frac{1}{2}}^{+}$ into $\tilde{\Vb}_{i+\frac{1}{2}}^{+}$. If $\Vb_{i+\frac{1}{2}}^{+}$, $\qb_{i}^{+,*}$ are not in the admissible set, then the following limiter essentially pushes those quantities towards $\Vb_{i}^{+}$ which we know is in the admissible set.

\noindent {\bf Step 1:} Replace densities $\rho_{i+\frac{1}{2}}^{+} $ and $\rho_{i}^{*}$ by,
\[
\hat{\rho}_{i+\frac{1}{2}}^{+} = \theta_{1} \left(\rho_{i+\frac{1}{2}}^{+}-\rho_{i}^{+} \right) + \rho_{i}^{+}, \qquad
\hat{\rho}_{i}^{*} = \theta_{1} \left(\rho_{i}^{*}-\rho_{i}^{+} \right) + \rho_{i}^{+}
\]
where $\theta_{1} = \mbox{min} \left\{ \frac{\rho^{+}_{i} - \epsilon}{\rho^{+}_{i}-\rho_{min} },1 \right\}$, with $\rho_{min} = \min \left\{ \rho_{i+\frac{1}{2}}^{+}, \rho_{i}^{*}\right\}$ and define the modified states,
$$
\hat{\Vb}_{i+\frac{1}{2}}^{+} =  ( \hat{\rho}_{i+\frac{1}{2}}^{+},m_{1,i+\frac{1}{2}}^{+}, m_{2,i+\frac{1}{2}}^{+}, E_{11, i+\frac{1}{2}}^{+}, E_{12, i+\frac{1}{2}}^{+}, E_{22, i+\frac{1}{2}}^{+})^{\top},
$$
and
$$
\hat{\qb}_{i}^{+,*} = (\hat{\rho}_{i}^{*}, m_{1,i}^{*}, m_{2,i}^{*}, E_{11,i}^{*}, E_{12,i}^{*}, E_{22,i}^{*})^{\top}
$$
{\bf Step 2:}
If we have $ p_{11} (\hat{\Vb}_{i+\frac{1}{2}}^{+}) < \epsilon$, then we have to solve the following quadratic equation for $t_{1} \in [0,1]$,
$$
p_{11} \left((1-t_{1})\Vb_{i}^{+} + t_{1} \hat{\Vb}_{i+\frac{1}{2}}^{+}\right) = \epsilon
$$
Similarly, if $p_{11} (\hat{\qb}_{i}^{+,*}) < \epsilon$, then we have to solve following equation for $t_{2}\in[0,1]$,
$$
p_{11} \left((1-t_{2})\Vb_{i}^{+} + t_{2} \hat{\qb}_{i}^{+,*}\right) = \epsilon
$$ 
We define $t_{1}$ and $t_{2}$ as,
\begin{equation*}
t_{1} = \left\lbrace \begin{array}{l l l} \mbox{Root of} \; p_{11} \left((1-t_{1})\Vb_{i}^{+} + t_{1} \hat{\Vb}_{i+\frac{1}{2}}^{+}\right) = \epsilon, & \mbox{if} & p_{11} (\hat{\Vb}_{i+\frac{1}{2}}^{+}) < \epsilon  \\
1, & \mbox{if} &  p_{11} (\hat{\Vb}_{i+\frac{1}{2}}^{+}) \geq  \epsilon
\end{array} \right.	
\end{equation*}
\begin{equation*}
t_{2} = \left\lbrace \begin{array}{l l l} \mbox{Root of} \; p_{11} \left((1-t_{2})\Vb_{i}^{+} + t_{2} \hat{\qb}_{i}^{+,*}\right) = \epsilon, & \mbox{if} & p_{11}(\hat{\qb}_{i}^{+,*}) < \epsilon \\
1, & \mbox{if} &  p_{11} (\hat{\qb}_{i}^{+,*}) \geq  \epsilon \end{array} \right.	
\end{equation*}
Now, If we have $ p_{22} (\hat{\Vb}_{i+\frac{1}{2}}^{+}) < \epsilon$, then we have to solve the following quadratic equation for $t_{3} \in [0,1]$,
$$
p_{22} \left((1-t_{3})\Vb_{i}^{+} + t_{3} \hat{\Vb}_{i+\frac{1}{2}}^{+}\right) = \epsilon
$$
Similarly, if $p_{22} (\hat{\qb}_{i}^{+,*}) < \epsilon$, then we have to solve following equation for $t_{4}\in[0,1]$,
$$
p_{11} \left((1-t_{4})\Vb_{i}^{+} + t_{4} \hat{\qb}_{i}^{+,*}\right) = \epsilon
$$ 
We define $t_{3}$ and $t_{4}$ as,
\begin{equation*}
t_{3} = \left\lbrace \begin{array}{l l l} \mbox{Root of} \; p_{22} \left((1-t_{3})\Vb_{i}^{+} + t_{3} \hat{\Vb}_{i+\frac{1}{2}}^{+}\right) = \epsilon, & \mbox{if} & p_{22} (\hat{\Vb}_{i+\frac{1}{2}}^{+}) < \epsilon  \\
1, & \mbox{if} &  p_{22} (\hat{\Vb}_{i+\frac{1}{2}}^{+}) \geq  \epsilon,  \end{array} \right.	
\end{equation*}
\begin{equation*}
t_{4} = \left\lbrace \begin{array}{l l l} \mbox{Root of} \; p_{22} \left((1-t_{4})\Vb_{i}^{+} + t_{4} \hat{\qb}_{i}^{+,*}\right) = \epsilon, & \mbox{if} & p_{22}(\hat{\qb}_{i}^{+,*}) < \epsilon \\
1, & \mbox{if} &  p_{22} (\hat{\qb}_{i}^{+,*}) \geq  \epsilon \end{array} \right.	
\end{equation*}
Let us define $\theta_{2}$ as,
$$\theta_{2} = \min \left(t_{1}, t_{2}, t_{3}, t_{4} \right)$$
We introduce new states,
\[
\check{\wb}_{i+\frac{1}{2}}^{+} = \theta_{2} \left( \hat{\wb}_{i+\frac{1}{2}}^{+} - \wb_{i}^{+} \right) + \wb_{i}^{+}, \qquad
\check{\qb}^{+,*}_{i} = \theta_{2} \left( \hat{\qb}^{+,*}_{i} - \wb_{i}^{+} \right) + \wb_{i}^{+}
\]
{\bf Step 3:}
If we have $\det\pb(\check{\Vb}_{i+\frac{1}{2}}^{+}) < \epsilon$, we have to find solution of the following cubic equation in [0,1],
$$
\det\pb\left((1-t_{5})\Vb_{i}^{+} + t_{5} \check{\Vb}_{i+\frac{1}{2}}^{+}\right) = \eps
$$
and if $\det\pb(\check{\qb}^{+,*}_{i}) < \epsilon$, we also have to solve,
$$
\det\pb\left((1-t_{6})\Vb_{i}^{+} + t_{6} \check{\qb}_{i}^{+,*}\right) = \eps
$$
We define $t_{5}$ and $t_{6}$ as,
\begin{equation*}
t_{5} = \left\lbrace \begin{array}{l l l} \mbox{Root of} \det\pb\left((1-t_{5})\Vb_{i}^{+} + t_{5} \check{\Vb}_{i+\frac{1}{2}}^{+}\right) = \epsilon & \mbox{if} & \det\pb(\check{\Vb}_{i+\frac{1}{2}}^{+}) < \epsilon  \\
1 & \mbox{if} & \det\pb (\check{\Vb}_{i+\frac{1}{2}}^{+}) \geq  \epsilon \end{array} \right.	
\end{equation*}
\begin{equation*}
t_{6} = \left\lbrace \begin{array}{l l l} \mbox{Root of} \det\pb\left((1-t_{6})\Vb_{i}^{+} + t_{6} \check{\qb}_{i}^{+,*}\right) = \epsilon & \mbox{if} & \det\pb(\check{\qb}_{i}^{+,*}) < \epsilon \\
1 & \mbox{if} &  \det\pb (\check{\qb}_{i}^{+,*}) \geq  \epsilon \end{array} \right.	
\end{equation*}
Let us define $\theta_{3}$,
$$\theta_{3} = \min \left(t_{5}, t_{6} \right)$$
Using $\theta_{3}$, we define the new state,
\begin{equation*}
\label{eq:u_tilde}
\tilde{\Vb}_{i+\frac{1}{2}}^{+} = \theta_{3}\left(\check{\Vb}_{i+\frac{1}{2}}^{+}-\Vb_{i}^{+} \right) + \Vb_{i}^{+}
\end{equation*}
{\bf Step 4:} Similarly, using the values $\Vb_{i+1}^-, \Vb_\iph^-, \qb_{i+1}^{-,*}$ we can compute $\tilde{\Vb}_\iph^-$. Then we obtain the scaled numerical flux as
\[
\fb_{\iph} = \alpha_{\iph}( \tilde{\Vb}_\iph^+ - \tilde{\Vb}_\iph^-)
\]
which is the flux which is used in the finite difference scheme.
\begin{remark}
The computer implementation of the above method is best achieved by performing a loop over all the faces of the cells in the mesh. At each face, the positive and negative fluxes are computed using WENO scheme, the scaling limiter is applied to them and the numerical flux is obtained as the sum of the scaled positive and negative fluxes. Then in a loop over all cells, the conserved variables are updated.
\end{remark}
\section{2-D finite difference WENO scheme}\label{sec:weno2d}
For the two-dimensional system \eqref{eq:balance_law}, the domain is discretized into a uniform mesh with nodes $(x_{i},y_{j})$. We define $x_{i+\frac{1}{2}} = \frac{1}{2}(x_{i+1}+x_{i})$, 
$y_{j+\frac{1}{2}} = \frac{1}{2}(y_{j+1}+y_{j})$ which are the mid-points of the faces, and with $I_{ij} = [x_{i-\frac{1}{2}},x_{i+\frac{1}{2}}]\times[y_{j-\frac{1}{2}},y_{j+\frac{1}{2}}]$ being a rectangular cell of size $\Delta x \times \Delta y$. The conservative finite difference scheme using forward Euler method is given by,
\begin{equation}
\label{eq:fdm2d}
\ub_{ij}^{n+1} = \ub_{ij}^{n} - \lambda^{x} \left(\fb^n_{i+\frac{1}{2},j} - \fb^n_{i-\frac{1}{2},j} \right) - \lambda^{y} \left( \gb^n_{i,j+\frac{1}{2}} - \gb^n_{i,j-\frac{1}{2}} \right)
\end{equation}
where $\lambda^{x} = \frac{\Delta t}{\Delta x}$ and $\lambda^{y} = \frac{\Delta t}{\Delta y}$. The fluxes are computed using the WENO scheme by applying it along each coordinate direction. For the fluxes $\fb_{\iph,j}$ along the $x$ axes, we use the splitting
\[
\fb^\pm(\ub) = \frac{1}{2}\left[ \fb(\ub) \pm \alpha^x_{\iph,j} \ub \right], \qquad \alpha^x_{\iph,j} = \max\{ \alpha^x(\ub_{i,j}), \alpha^x(\ub_{i+1,j} \}
\]
then compute the characteristic decomposition using the eigenvectors of the flux Jacobian along the $x$ direction, and apply the WENO reconstruction. A similar approach is used to compute the fluxes $\gb_{i,j+\frac{1}{2}}$ along the $y$ direction. We rewrite the scheme \eqref{eq:fdm2d} as,
\begin{align}
\nonumber
\ub_{ij}^{n+1} & = \left(\frac{\alpha^{x}\Delta y}{\alpha^{x} \Delta y + \alpha^{y}\Delta x}\right)\ub_{ij}^{n} +  \left(\frac{\alpha^{y}\Delta x}{\alpha^{x} \Delta y + \alpha^{y}\Delta x}\right) \ub_{ij}^n - \lambda^{x} \left(\fb^n_{i+\frac{1}{2},j} - \fb^n_{i-\frac{1}{2},j} \right) - \lambda^{y} \left( \gb^n_{i,j+\frac{1}{2}} - \gb^n_{i,j-\frac{1}{2}} \right)
\end{align}
where $\alpha^{x} = \displaystyle{\max_{i,j} \alpha^x(\ub_{ij}^{n})}$ and $\alpha^{y} = \displaystyle{\max_{i,j} \alpha^y(\ub_{ij}^{n})}$.
This can be written as a convex combination of two one-dimensional schemes
\begin{align*}
\ub_{ij}^{n+1} = \left(\frac{\alpha^{x}\Delta y}{\alpha^{x} \Delta y + \alpha^{y}\Delta x}\right) \ub_{ij}^{n+1,x} + \left(1-\frac{\alpha^{x}\Delta y}{\alpha^{x} \Delta y + \alpha^{y}\Delta x}\right) \ub_{ij}^{n+1,y}
\end{align*}
where
\begin{align}
\label{eq:ux}
& \ub_{ij}^{n+1,x} = \ub_{ij}^{n} - \frac{\lambda^{x}(\alpha^{x} \Delta y + \alpha^{y}\Delta x)}{\alpha^{x}\Delta y} \left(\fb^n_{i+\frac{1}{2},j} - \fb^n_{i-\frac{1}{2},j} \right) \\
\label{eq:uy}
& \ub_{ij}^{n+1,y} = \ub_{ij}^{n} -\frac{\lambda^{y}(\alpha^{x} \Delta y + \alpha^{y}\Delta x)}{\alpha^{y}\Delta x} \left( \gb^n_{i,j+\frac{1}{2}} - \gb^n_{i,j-\frac{1}{2}} \right) 
\end{align}
To prove the positivity of the scheme, we apply the same procedure to both the one-dimensional schemes as we discussed in previous section. So, we introduce the new variables
\[
\wb^{x,\pm}(\ub) = \frac{1}{2} \left(\ub \pm \frac{\fb(\ub)}{\alpha^{x}(\ub)} \right), \qquad \wb^{y,\pm}(\ub) = \frac{1}{2} \left(\ub \pm \frac{\gb(\ub)}{\alpha^{y}(\ub)} \right)
\]
and rewrite \eqref{eq:ux} as
\[
\ub_{ij}^{n+1,x} = \ub_{ij}^{n+1,x,+} + \ub_{ij}^{n+1,x,-}
\]
with
\begin{align}
\nonumber
\ub_{ij}^{n+1,x,+} & = \wb^{x,+}_{ij} - \frac{\lambda^{x}(\alpha^{x} \Delta y + \alpha^{y}\Delta x)}{\alpha^{x}\Delta y} \left( \alpha^{x}_{i+\frac{1}{2},j} \wb^{x,+}_{i+\frac{1}{2},j}- \alpha^{x}_{i-\frac{1}{2},j} \wb^{x,+}_{i-\frac{1}{2},j} \right)\\
\nonumber
& = (1-\hat{w}_{1}) \qb_{ij}^{x,+,*} + \left(\hat{w}_{1} -\frac{\lambda^{x}(\alpha^{x} \Delta y + \alpha^{y}\Delta x)}{\alpha^{x}\Delta y} \alpha^{x}_{i+\frac{1}{2},j} \right) \wb_{i+\frac{1}{2},j}^{x,+} \\
& + \frac{\lambda^{x}(\alpha^{x} \Delta y + \alpha^{y}\Delta x)}{\alpha^{x}\Delta y} \alpha^{x}_{i-\frac{1}{2},j} \wb^{x,+}_{i-\frac{1}{2},j}
\label{eq:u+}
\end{align}
and
\begin{align}
\nonumber
\ub_{ij}^{n+1,x,-} & = \wb^{x,-}_{ij} - \frac{\lambda^{x}(\alpha^{x} \Delta y + \alpha^{y}\Delta x)}{\alpha^{x}\Delta y} \left( \alpha^{x}_{i+\frac{1}{2},j} \wb^{x,-}_{i+\frac{1}{2},j}- \alpha^{x}_{i-\frac{1}{2},j} \wb^{x,-}_{i-\frac{1}{2},j} \right)\\
\nonumber
& = (1-\hat{w}_{1}) \qb_{ij}^{x,-,*} + \left(\hat{w}_{1} -\frac{\lambda^{x}(\alpha^{x} \Delta y + \alpha^{y}\Delta x)}{\alpha^{x}\Delta y} \alpha^{x}_{i-\frac{1}{2},j} \right) \wb_{i-\frac{1}{2},j}^{x,-} \\
&+ \frac{\lambda^{x}(\alpha^{x} \Delta y + \alpha^{y}\Delta x)}{\alpha^{x}\Delta y} \alpha^{x}_{i+\frac{1}{2},j} \wb^{x,-}_{i+\frac{1}{2},j}
\label{eq:u-}
\end{align}
where
\[
\Vb_{\iph,j}^{x,\pm} = \pm \frac{1}{\alpha^x_{\iph,j}} \fb_{\iph,j}^\pm, \qquad \Vb_{i,\jph}^{y,\pm} = \pm \frac{1}{\alpha^y_{i,\jph}} \gb_{i,\jph}^\pm
\]
\[
\qb^{x,+,*}_{ij} = \frac{1}{1-\hat{w}_{1}} \sum_{r=1}^{N-1} \hat{w}_{r} \qb_{ij}^{x,+}(\hat{x}_{i}^{r}), \qquad
\qb^{x,-,*}_{ij} = \frac{1}{1-\hat{w}_{1}} \sum_{r=2}^{N} \hat{w}_{r} \qb_{ij}^{x,-}(\hat{x}_{i}^{r})
\]
The $\qb^{x,\pm,*}_{ij}$ quantities are defined in terms of $\qb_{ij}^{x,\pm}(x)$ which are $4^{th}$ degree polynomials such that
$$
\qb_{ij}^{x,+}(x_{i+\frac{1}{2}}) = \wb^{x,+}_{i+\frac{1}{2},j}, \qquad
\wb_{ij}^{x,+} = \frac{1}{\Delta x} \int_{x_{i-\frac{1}{2}}}^{x_{i+\frac{1}{2}}} \qb_{ij}^{x,+}(x) \: dx
$$
$$\qb_{ij}^{x,-}(x_{i-\frac{1}{2}}) = \wb^{x,-}_{i-\frac{1}{2},j}, \qquad
\wb_{ij}^{x,-} = \frac{1}{\Delta x} \int_{x_{i-\frac{1}{2}}}^{x_{i+\frac{1}{2}}} \qb_{ij}^{x,-}(x) \: dx
$$
and $N$ is the number of quadrature points of Gauss-Lobatto quadrature rule used to approximate the states $\wb^{x}_{ij}$ in $x$-direction which must be exact for a polynomial of degree four. A similar procedure is followed to write $\ub_{ij}^{n+1,y}$ as a sum of two parts $\ub_{ij}^{n+1,y,\pm}$. Each of these terms is written as a linear combination of $\qb_{ij}^{y,\pm,*}$, $\wb_{i,\jph}^{y,\pm}$. Then, we have the following result for positivity of the scheme \eqref{eq:fdm2d},
\begin{thm}
\label{thm:pos2D}
Assume that the following conditions hold for all $i,j$.
	\begin{enumerate}
		\item $\ub_{ij}^{n} \in \uad$
		\item $\left \lbrace \qb_{ij}^{x,\pm,*}, \wb^{x,\pm}_{i+\frac{1}{2},j} \right \rbrace \in \uad$
		\item $\left \lbrace \qb_{ij}^{y,\pm,*}, \wb^{y,\pm}_{i,j+\frac{1}{2}} \right \rbrace \in \uad$
	\end{enumerate}
Then the two-dimensional finite difference WENO scheme \eqref{eq:fdm2d} is positivity-preserving if the following CFL condition is satisfied
\begin{equation}
\left(\frac{\alpha^x}{\Delta x}+\frac{\alpha^y}{\Delta y}\right) \Delta t \leq \hat{w}_1
\label{eq:cfl2d}
\end{equation}
i.e., $\ub_{ij}^{n+1} \in \uad$.
\end{thm}	
We skip the detailed proof as it is simply based on using the one dimensional schemes along each coordinate direction and the convex combination property. The convexity requires the coefficients to be positive which leads to the CFL condition. The positivity of the second coefficient in (\ref{eq:u+}) requires that
\[
\left(\frac{\alpha^x}{\Delta x}+\frac{\alpha^y}{\Delta y}\right) \Delta t \leq \hat{w}_1 \frac{\alpha^x}{\alpha^x_{\iph,j}}
\]
and since $\alpha^x \ge \alpha^x_{\iph,j}$, we find that (\ref{eq:cfl2d}) is a sufficient condition. This condition is enough to ensure that $\ub_{ij}^{n+1,x,\pm}$, $\ub_{ij}^{n+1,y,\pm}$ are convex combinations of certain other quantities. The positivity limiter as described in the one dimensional case is applied to achieve the conditions $(2)$ and $(3)$ of the above theorem from which the positivity of $\ub_{ij}^{n+1}$ follows due to convex combination of positive states. We note that the CFL number is again $\hat{w}_1 = \frac{1}{12}$.	
\section{Implementation details}
\label{sec:impl}
In computing numerical solutions of the Ten-Moment equations, we must ensure  positivity of the density and positive definiteness of pressure tensor under the restriction on CFL condition derived in Section~(\ref{sec:weno1d})-(\ref{sec:weno2d}). The time step can be computed from
\[
 \Delta t(\ub,\mbox{CFL})=\frac{\mbox{CFL}}{\max_{i,j}\left( \frac{\alpha^x(\ub_{ij})}{\Delta x} + \frac{\alpha^y(\ub_{ij})}{\Delta y} \right)}
\]
and for positivity property to hold, we have to choose $CFL=\hat{w}_1 = 1/12$ which is twice the CFL number used in \cite{zha-shu_12a}. Even then, this CFL number is too small compared to a $CFL \approx 1$ which is enough for $L^2$ stability, and this reduced CFL will lead to more time steps and hence higher computational cost.  To decrease the computational costs, we have implemented an adaptive CFL strategy. For the iteration from time $t^{n}$ to $t^{n} + \Delta t$, we first apply the Runge-Kutta scheme without any positivity limiter and using a CFL $\approx 1$. After each Runge-Kutta stage, we check for the positivity of the solution.  If the solution fails to satisfy positivity conditions in any cell, the simulations exit the Runge-Kutta iteration and we go back to the previous time level $t^n$. Now, we compute the time step $\Delta t$ by $\Delta t(\ub^{n},\hat{w}_{1})$ which ensures positivity preservation, and apply the Runge-Kutta scheme with positivity limiter to produce the solution at time $t^{n} + \Delta t$. The implementation of the algorithm with adaptive CFL is explained in Algorithm \eqref{alg:adaptiveCFL} which performs one time step of the finite difference scheme.
\begin{algorithm}
	\label{alg:adaptiveCFL}
	\SetAlgoLined
	Given $\ub^{n}$\;
	Set $\Delta t = \Delta t(\ub^{n},\mbox{CFL})$\;
	Set poslim = false\;
	$\ub^{n+1}= \mbox{SSPRK3}(\ub^{n},\Delta t,\mbox{poslim})$\;
	\If{poslim=true}{
		$\Delta t = \Delta t(\ub^{n},\hat{w}_{1})$\;
		$\ub^{n+1}= \mbox{SSPRK3}(\ub^{n},\Delta t,\mbox{poslim})$\;
	}
	\caption{Adaptive CFL for positivity limiter}
\end{algorithm}
The implementation of the SSPRK3 scheme for one time step is explained in  Algorithm~\eqref{alg:ssprk3} for the case of no source term and in Algorithm \eqref{alg:expssprk3} when there is source term present in the Ten-Moment model. The use of such an adaptive time step condition helps to reduce the computational cost by allowing us to take large time steps in most of the iterations and we provide numerical evidence of this in later sections.
\begin{algorithm}
	\label{alg:ssprk3}
	\SetAlgoLined
	Given $\ub^{n}$, $\Delta t$, poslim\;
	\For{each RK stage}{
		\eIf{poslim == true}{
			Compute residual $\Rb$ using positivity limiter;
		}
		{
			Compute $\Rb$ without using positivity limiter\;
		}
		\For{each cell}{
		    Update solution to next stage\;
		    \If{density or pressure is negative}
		    {
		    Set poslim = true and return\tcp*[r]{Exit algorithm}
		    }
		}
	}
	Return $\ub^{n+1}$\;
	\caption{SSPRK3 Scheme}
\end{algorithm}

\begin{algorithm}
	\label{alg:expssprk3}
	\SetAlgoLined
	Given $\ub^{n}$, $\Delta t$, poslim\;
	Compute $\Rb(\ub^n)$ using positivity limiter if poslim == true\tcp*[r]{Stage 1}	\For{each cell}{
		Compute 
		$
		\hat{\ub} = \ub^n + \Delta t  \Rb(\ub^n)
		$ \;
		\If{density or pressure is negative}
		{
		Set poslim=true and return\;
		}
		Compute 
		$
		\hat{\ub} = e^{\Cb(t^n,\Delta t)} \hat{\ub}$
	}	
	Compute $\Rb(\hat{\ub})$ using positivity limiter if poslim == true\tcp*[r]{Stage 2}
	\For{each cell}{
		Compute $\hat{\ub} = \hat{\ub} + \Delta t  \Rb(\hat{\ub})$\;
		\If{density or pressure is negative}
		{
		Set poslim=true and return\;
		}
		Compute $\hat{\ub} = \frac{3}{4}  e^{\Cb(t^n,\frac{\Delta t}{2})}  \ub^n + \frac{1}{4} e^{-\Cb(t^n+\frac{\Delta t}{2},\frac{\Delta t}{2})} \hat{\ub}$\;
	}
	Compute $\Rb(\ub^n)$ using positivity limiter if poslim == true\tcp*[r]{Stage 3}
	\For{each cell}{
		Compute $\hat{\ub} = \hat{\ub} + \Delta t  \Rb(\hat{\ub})$\;
		\If{density or pressure is negative}
		{
		Set poslim=true and return\;
		}
		Compute $\hat{\ub} = \frac{1}{3} e^{\Cb(t^n, \Delta t)}  \ub^n + \frac{2}{3}  e^{\Cb(t^n+\frac{\Delta t}{2},\frac{\Delta t}{2})} \hat{\ub}$\;
	}
	Return $\ub^{n+1} = \hat{\ub}$\;
	\caption{Integrating factor SSPRK3 scheme}
\end{algorithm}
\section{Numerical results}
\label{sec:nm}
In this section, we first test the accuracy and convergence rate of the proposed positivity preserving WENO scheme for Ten-Moment equations on simple problems with known solutions and then apply it to more complicated problems. Further, we have compared the resolution of proposed algorithms for the case where the solution has shocks or discontinuities. To test the positivity property of the scheme, we have performed many one- and two-dimensional test cases where the positivity limiter is shown to stabilize the computations. The proposed algorithm is designed to work with all variants of WENO schemes. But to demonstrate the idea, we consider
the fifth order accurate WENO-JS \cite{jia-shu_96a}, WENO-Z \cite{bor-etal_08d}, and WENO-AO \cite{bal-etal_16a} schemes for flux reconstruction. In the case of WENO-AO \cite{bal-etal_16a}, we have chosen the linear weights as follows
\begin{equation*}
 \gamma_0^5 = 0.5, \quad \gamma_{-1}^3=0.125, \quad \gamma_0^3=0.25, \quad \gamma_1^3=0.125
\end{equation*}
In the following sections, we specify the initial condition for each test case in terms of the primitive variables
\[
\bold{V} = [\rho, \ v_1, \ v_2, \ p_{11}, \ p_{12}, \ p_{22}]
\]
\subsection{Accuracy tests} 
We first perform numerical experiments to measure the accuracy and convergence rate of proposed positivity preserving WENO schemes for Ten-Moment equations. 
The accuracy of the schemes are measured in $L^{\infty}$-, $L^1$-, $L^2$-error norms, which are defined for error $e$ over the domain $[a,b]$ as follows
\[
 \|e\|_{\infty} = \max_{j} |u_j-(u_{\Delta x})_j|, \qquad  \| e \|_1 =  \Delta x \sum_{j}|u_j-(u_{\Delta x})_j| 
\]
\[
\| e \|_2 = \left(\Delta x \sum_{j}|u_j-(u_{\Delta x})_j|^2\right)^{\frac{1}{2}}
\]
where  $u_j$ and $(u_{\Delta x})_j$ denote the exact and approximate solutions at point $x_j$, respectively.
\begin{figure}
\begin{center}
	\begin{tabular}{c c}
\includegraphics[width=2.8in]{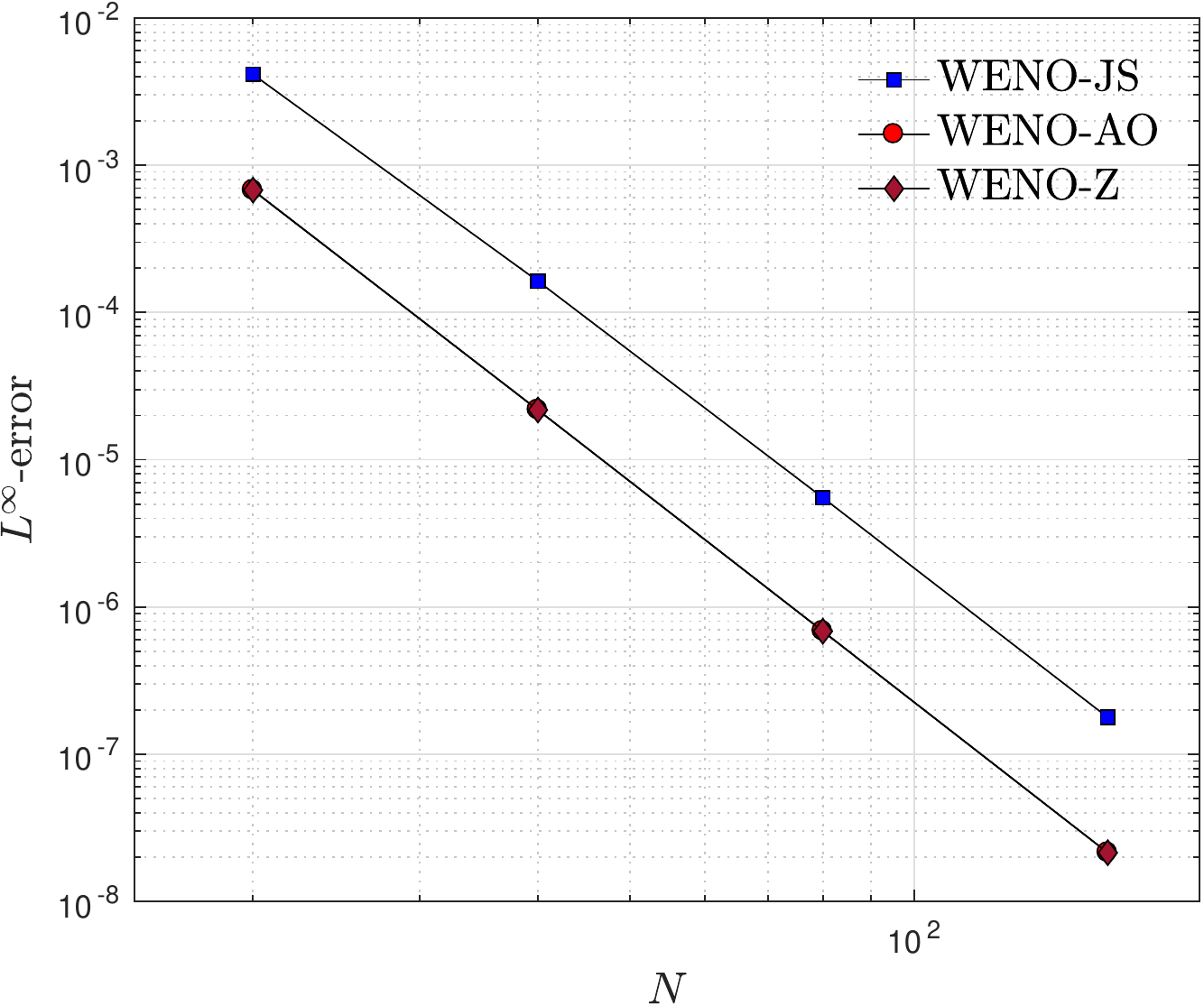}&
\includegraphics[width=2.8in]{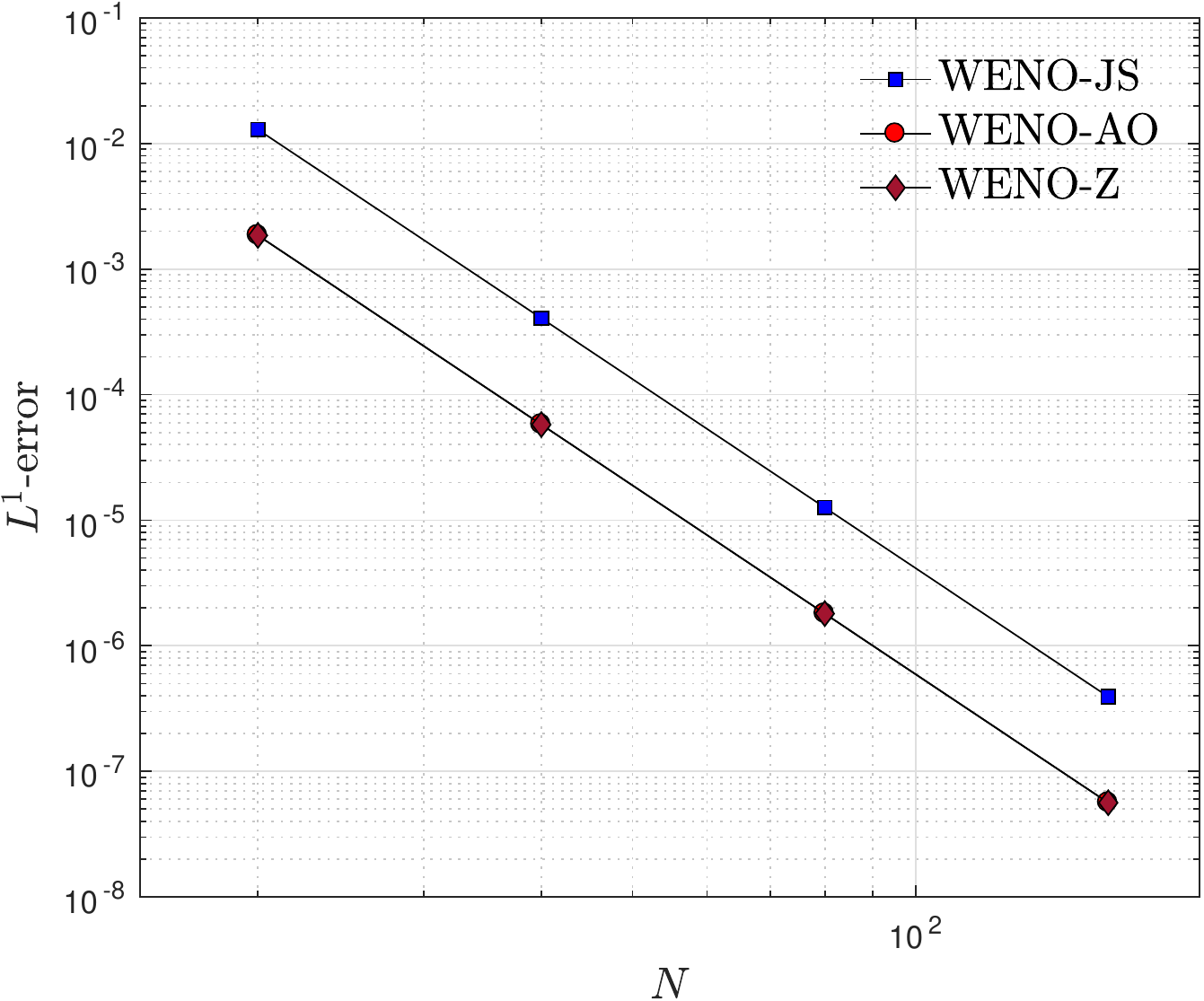}\\
	(a) Convergence plot: $L^{\infty}$-error& (b) Convergence plot:  $L^{1}$-error  
   	\end{tabular}
\end{center}
    \caption{Comparison of WENO-JS, WENO-AO, and WENO-Z schemes in terms of $L^{\infty}$- and $L^1$-errors for Example \ref{Ex.hom.acc} at time $T=0.5$. }
    \label{Figure:ex1}
\end{figure}
Since the time discretization is only third order accurate while the space discretization is fifth order accurate, to ensure fifth-order convergence of the full scheme, we  choose the time step $\Delta t= C \Delta x^{5/3}$ such that this also satisfies the CFL condition.
\begin{table}
\centering 
\begin{tabular}{|c| c|c| c| c| c|  c| c| r|}
\hline 
 $N$ & WENO-JS ($L^{\infty}$) & Order & WENO-AO ($L^{\infty}$) & Order & WENO-Z ($L^{\infty}$)& Order \\ 
\hline
20  &  4.142134e-03 & --     &6.748957e-04 &-- &6.779808e-04 &\\ 
\hline 
40  &  1.632954e-04 & 4.66     &2.176286e-05 &4.95  &2.177505e-05 & 4.96    \\
\hline
80  &  5.526777e-06 & 4.88     &6.843767e-07 &4.99  &6.844163e-07 & 4.99    \\
\hline
160  &  1.778311e-07 & 4.96     &2.142911e-08 &5.00  &2.142922e-08 & 5.00    \\
\hline
320  & 5.493679e-09 & 5.02      &6.721070e-10 & 4.99 & 6.720646e-10 & 4.99 \\
\hline
\hline
 $N$ & WENO-JS ($L^{1}$) & Order & WENO-AO ($L^{1}$) & Order & WENO-Z ($L^{1}$)& Order \\ 
\hline
20  &  1.289941e-02 & --     &1.856446e-03 &-- &1.853946e-03 &\\ 
\hline 
40  &  4.061796e-04 & 4.99     &5.784190e-05 &5.00  &5.781507e-05 & 5.00    \\
\hline
80  &  1.262294e-05 & 5.00     &1.803930e-06 &5.00  &1.803879e-06 & 5.00    \\
\hline
160  &  3.932262e-07 & 5.00    &5.631144e-08 &5.00  &5.631134e-08 & 5.00   \\
\hline
320  &  1.226749e-08 & 5.00    & 1.755926e-09 &5.00 & 1.755916e-09 & 5.00 \\
\hline
\end{tabular}
\caption{ Comparison of WENO-JS, WENO-AO and WENO-Z schemes in terms of $L^{\infty}$- and $L^1$-errors along with their convergence rates for Example \ref{Ex.hom.acc} at time $T=0.5$.}
\label{Table.ex1}
\end{table}

\begin{example}\label{Ex.hom.acc}
 {\rm (Homogeneous case)
  Consider the 1-D Ten-Moment equations with initial data for primitive variables given by
\[
 \bold{V}=(2+\sin(2\pi x), \ 1, \ 0, \ 1, \ 0, \ 1)
\]
over the domain [-0.5,0.5] with periodic boundary conditions.   The profile of density at time $t$ advects towards the right with speed $2\pi$ given as $\rho(x,t) = 2+\sin(2 \pi (x-t))$, whereas other variables remain constant. Numerical solutions are computed at time $T=0.5$ over the domain having different grid sizes. The $L^{\infty}$- and $L^1$-errors obtained using WENO-JS, WENO-Z, and WENO-AO schemes are depicted in Table~\ref{Table.ex1} and Figure~\ref{Figure:ex1}. We observe that all WENO schemes converge to the exact solution with a convergence rate of five. The WENO-Z and WENO-AO schemes are more accurate than the WENO-JS scheme, and WENO-AO and WENO-Z schemes are comparable to each other as it can be observed from  Table~\ref{Table.ex1} and Figure~\ref{Figure:ex1}.
}
\end{example}

\begin{example}\label{Ex.acc2}{\rm (With time independent source term)
 This test is performed to study the accuracy of the proposed schemes for Ten-Moment equations with time-independent source term. We consider the domain $[-0.5,0.5]$ and potential is taken to be $W(x)=x$ with following continuous initial profile,
 \[
  \bold{V}=(2+\sin(2\pi x), \ 1, \ 0, \ 5-x+\frac{1}{4\pi}\cos(2\pi x), \ 0, \ 1)
 \]
%
The exact solution of the problem at time $t$ contains the advected profile of density and pressure term $p_{11}$ given by
\[
\rho(x,t) = 2 + \sin(2 \pi (x-t)), \qquad  p_{11}(x,t) = 5 + t-x +\frac{1}{4\pi} \cos(2\pi (x-t))  
\]
whereas other variables  remain constant. The boundary values for the numerical solutions are computed using the exact solution.
The solutions are computed until the final time $T=0.5$ using  20, 40, 80, and 160 grid points. The numerical errors are computed and depicted in Table~\ref{Table.ex2} and Figure~\ref{Figure.ex2}. It can be observed that all considered WENO schemes achieve the expected convergence rate and WENO-JS is less accurate as compared to WENO-Z and WENO-AO schemes.
}
\end{example}

\begin{figure}
\begin{center}
	\begin{tabular}{c c}
	\includegraphics[width=2.8in]{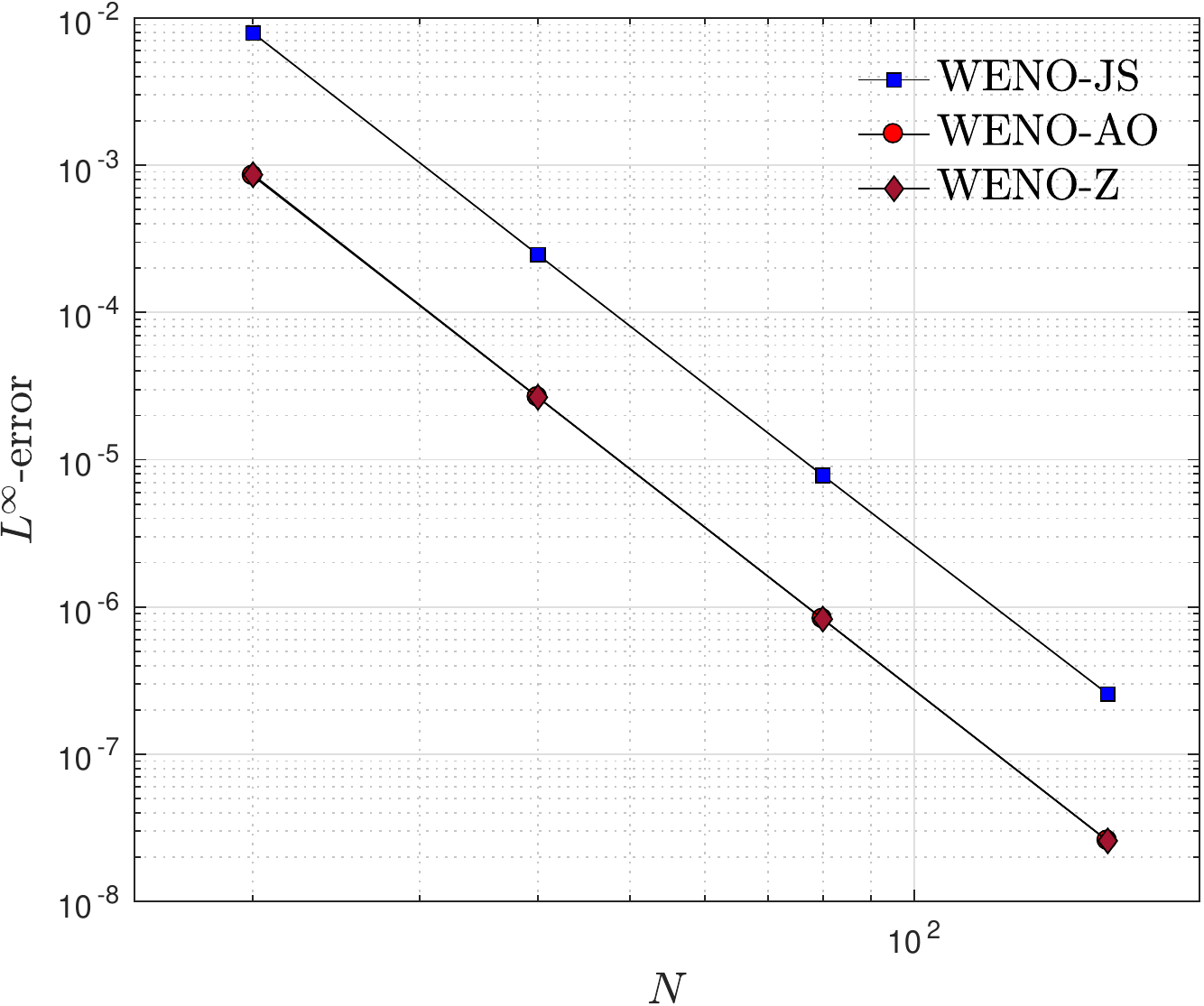} &
   	\includegraphics[width=2.8in]{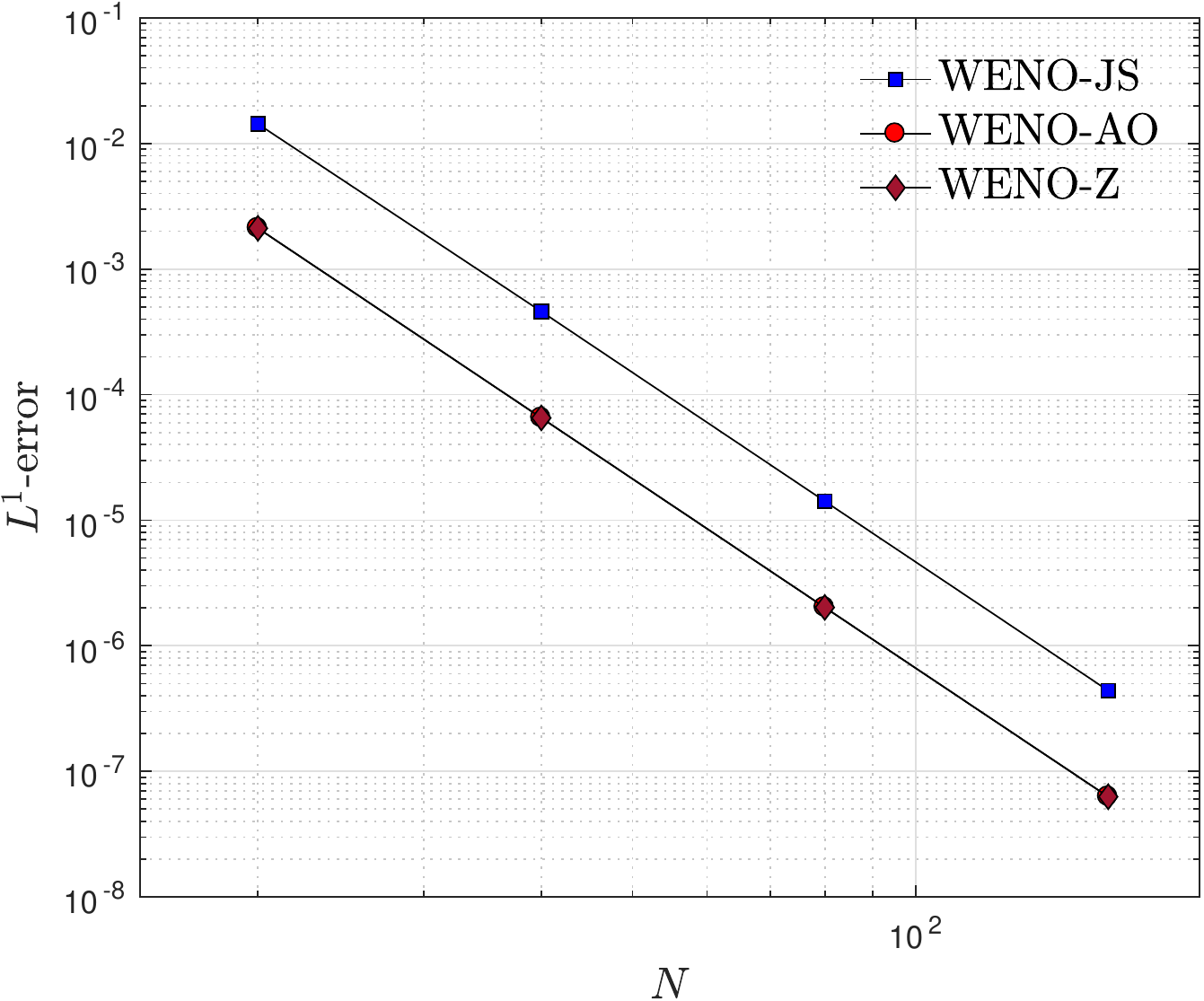}
    	\\
   	(a) Convergence plot: $L^{\infty}$-error& (b) Convergence plot:  $L^{1}$-error  
   	\end{tabular}
\end{center}
    \caption{Comparison of WENO-JS, WENO-AO and WENO-Z schemes in terms of $L^{\infty}$- and $L^1$-errors for Example \ref{Ex.acc2} at time $T=0.5$. }
    \label{Figure.ex2}
\end{figure}

\begin{table}
\centering 
\begin{tabular}{|c| c|c| c| c| c|  c| c| r|}
\hline 
 $N$ & WENO-JS ($L^{\infty}$) & Order & WENO-AO ($L^{\infty}$) & Order & WENO-Z ($L^{\infty}$)& Order \\ 
\hline
20  &  7.905498e-02 & --     &8.457498e-03 &-- &8.618240e-04 &\\ 
\hline 
40  &  2.468581e-04 & 5.00     &2.650183e-05 &5.00  &2.657482e-05 & 5.02    \\
\hline
80  &  7.813491e-06 & 4.98     &8.271779e-07 &5.00 &8.273620e-07 & 5.01    \\
\hline
160  &  2.563829e-07 & 4.93    &2.579152e-08 &5.00  &2.579116e-08 & 5.00    \\
\hline
320 & 9.350345e-09 & 4.78      & 8.050622e-10 & 5.00 & 8.050578e-10 & 5.00\\
\hline
\hline
 $N$ & WENO-JS ($L^{1}$) & Order & WENO-AO ($L^{1}$) & Order & WENO-Z ($L^{1}$)& Order \\ 
\hline
20  &  1.430735e-02 & --     &2.113319e-03 &-- &2.113968e-03 &\\ 
\hline 
40  &  4.594505e-04 & 4.96     &6.532232e-05 &5.02  &6.529159e-05 & 5.02    \\
\hline
80  &  1.417471e-05 & 5.02    &2.022391e-06 &5.01  &2.022334e-06 & 5.01    \\
\hline
160  &  4.397786e-07 & 5.01     &6.285120e-08 &5.01  &6.285107e-08 & 5.01    \\
\hline
320 &  1.369191e-08  & 5.01    & 1.958068e-09 & 5.00 & 1.958090e-09 & 5.00\\
\hline
\end{tabular}
\caption{ Comparison of WENO-JS, WENO-AO and WENO-Z schemes in terms of $L^{\infty}$- and $L^1$-errors along with their convergence rate for Example \ref{Ex.acc2} at time $T=0.5$.}
\label{Table.ex2}
\end{table}
\begin{figure}[htbp]
	\begin{center}
		\begin{tabular}{c c}
			\includegraphics[width=2.8in]{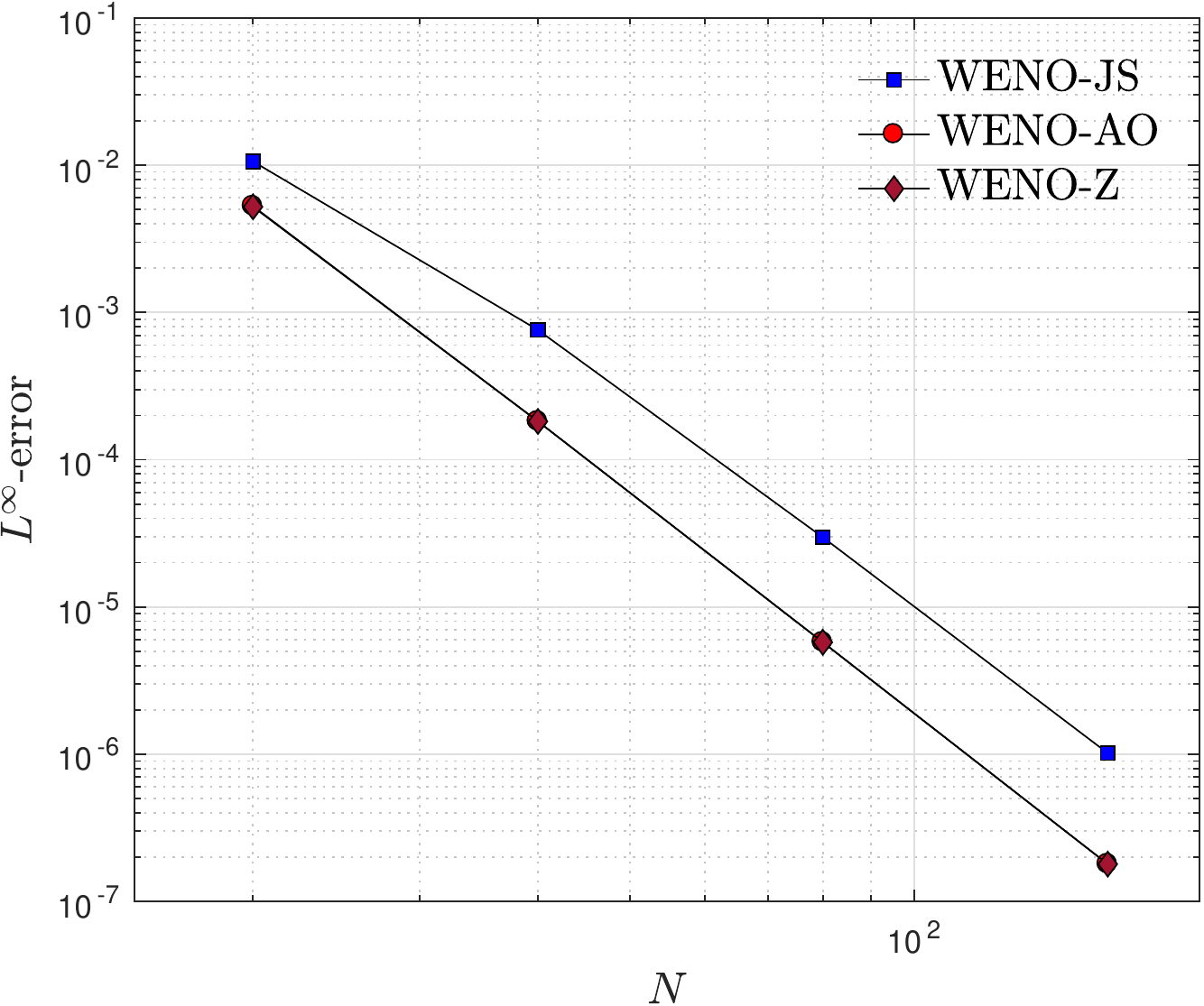} &
			\includegraphics[width=2.8in]{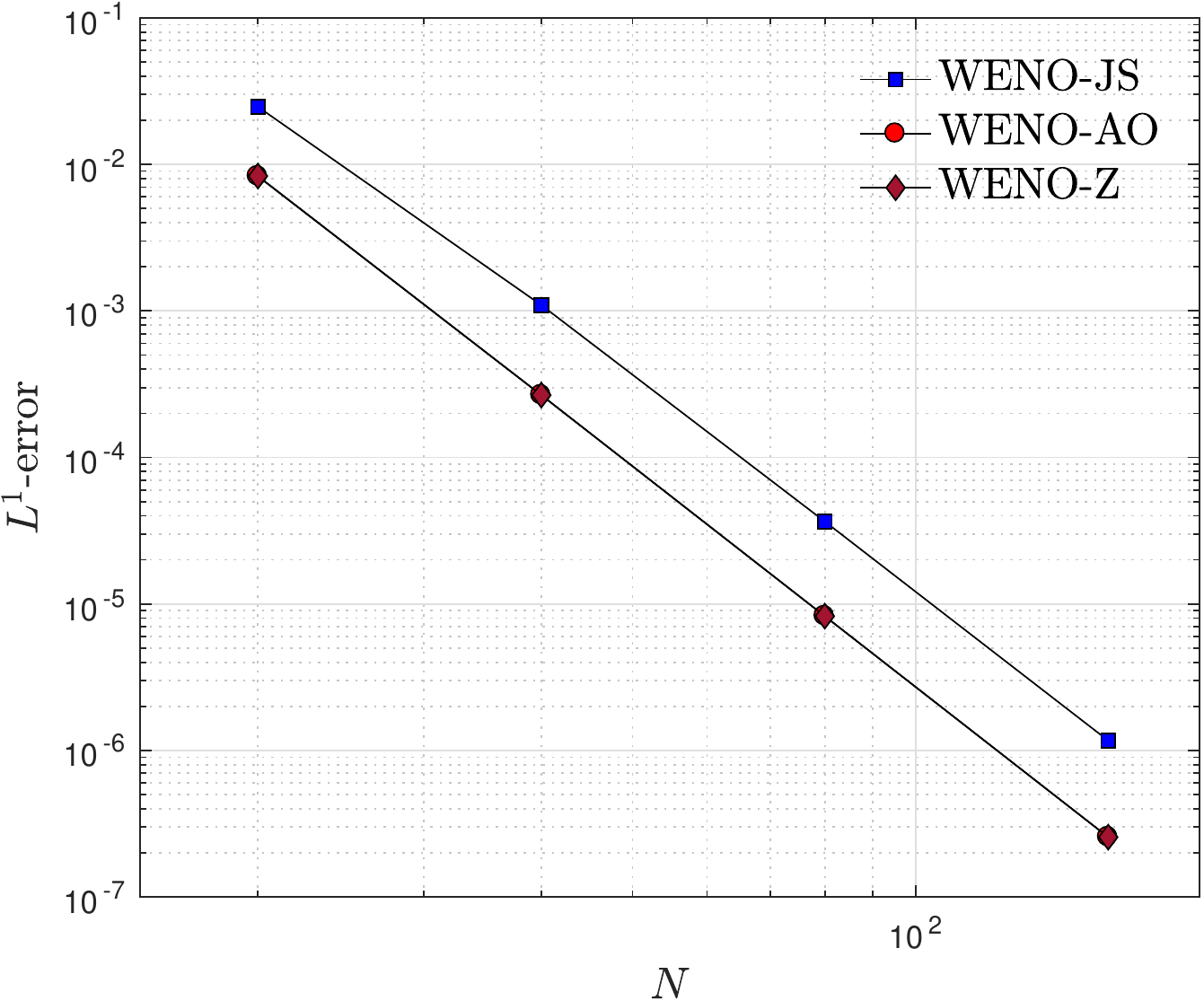} \\
			(a) Convergence rate: $L^{\infty}$-error of density & (b) Convergence rate: $L^{1}$-error of density
		\end{tabular}
	\end{center}
   \caption{Comparison of WENO-JS, WENO-AO and WENO-Z schemes in terms of $L^{\infty}$- and $L^1$-errors for Example \ref{Ex.acc3} at time $T=0.5$. }
   \label{Figure.ex3}
\end{figure}

\begin{table}
	\centering 
	\begin{tabular}{|c| c|c| c| c| c|  c| c| r|}
		\hline 
		$N$ & WENO-JS ($L^{\infty}$) & Order & WENO-AO ($L^{\infty}$) & Order & WENO-Z ($L^{\infty}$)& Order \\   
		\hline
20  &  1.061379e-02 & --     &5.276108e-03 &-- &5.221549e-03 &\\ 
\hline 
40  &  7.609719e-04 & 3.80     &1.822327e-04 &4.86  &1.820924e-04 & 4.84    \\
\hline
80  &  2.974258e-05 & 4.68     &5.759323e-06 &4.98  &5.759533e-06 & 4.98    \\
\hline
160  &  1.016390e-06 & 4.87     &1.792940e-07 &5.01  &1.792955e-07 & 5.01    \\
\hline
320 & 3.209556e-08 & 4.98      & 5.582646e-09 & 5.01 & 5.582708e-09 & 5.01 \\
\hline 
\hline
		$N$ & WENO-JS ($L^{1}$) & Order & WENO-AO ($L^{1}$) & Order & WENO-Z ($L^{1}$) & Order \\   
		\hline
		20  &  2.474187e-02 & --     &8.295712e-03 &-- &8.330823e-03 &\\ 
		\hline 
		40  &  1.096019e-03 & 4.50    &2.661359e-04 &4.96  &2.665033e-04 & 4.97   \\
		\hline
		80  &  3.654961e-05 & 4.91     &8.257317e-06 &5.01  &8.257391e-06 & 5.01    \\
		\hline
		160  &  1.167704e-06 & 4.97     &2.561385e-07 &5.01  &2.561381e-07 & 5.01   \\
		\hline
		320 & 3.663405e-08 & 4.99      &7.965435e-09 & 5.01 & 7.965479e-09 & 5.01\\
		\hline
	\end{tabular}
\caption{ Comparison of WENO-JS, WENO-AO and WENO-Z schemes in terms of $L^{\infty}$- and $L^1$-errors along with their convergence rate for Example \ref{Ex.acc3} at time $T=0.5$.}
	\label{Table.ex3}
\end{table}

\begin{example}\label{Ex.acc3}{\rm (With time-dependent source term) Here, we consider the time-dependent source term 
over the domain [-0.5,0.5]. The potential is taken as, $W(x,t)=\sin(2 \pi (x-t))$ and the initial profile is given by
\[
 \bold{V}=\left(2+\sin(2\pi x), \ 1, \ 0, \ 1.5+\frac{1}{8}(\cos(4\pi x)-8\sin(2\pi x)), \ 0, \ 1  \right)
\]
Similar to the previous cases,  velocity and pressure $p_{12}$ and $p_{22}$ remain constant with time $t$. The density and pressure $p_{11}$  at time $t$ are given by, 
\[
\rho(x,t) = 2+\sin(2 \pi (x-t)), \qquad p_{11}(x,t) = 1.5 + \frac{1}{8} [ \cos(4 \pi (x-t)) - 8 \sin(2 \pi (x-t)) ]
\]
Numerical solutions are computed till final time $T=0.5$ using periodic boundary conditions and errors are shown in Table \ref{Table.ex3} and Figure \ref{Figure.ex3}. All the considered schemes converge to the exact solution with a convergence rate of five.
} 
\end{example}
\begin{example}
\label{ex:lowdensity}
{\rm (Accuracy test for low density problem) 
This test is performed to study the accuracy of the proposed scheme when positivity limiter is active. We consider the domain $[-0.25,0.25]$ and the initial profile with the low density given as,
\begin{equation}\label{ex:ld}
 {\bf V} = \left( \tilde{\eps} + \sin^2(2 \pi x), \ 1, \ 0, \ 5 - x\left(\frac{\tilde\eps}{2} + \frac{1}{4}\right) + \frac{\sin(4 \pi x)}{16 \pi}, \ 0, \ 1 \right)
\end{equation}
with the potential $W(x)=x$. The exact solution at time $t$ is given by,
\[
\rho = \tilde{\eps} + \sin^2(2 \pi (x-t)), \qquad p_{11} = 5 + (t -x)\left(\frac{\tilde\eps}{2} + \frac{1}{4}\right) + \frac{\sin(4 \pi (x-t))}{16 \pi}
\]
All other variables remain constant at any time $t$. We compute the solutions with the parameter $\tilde{\eps}$ taking the value $10^{-2}$ or $10^{-6}$. The boundary values are evaluated using the exact solution. The solutions are computed at the final time $T=0.5$. For $\tilde \eps=10^{-2}$, the positivity limiter is not used since the base scheme maintains positivity during the simulations. In this case, the convergence results using WENO-AO scheme are tabulated in Table~\ref{Table.ex4a} and we observe close to optimal accuracy in all norms. For $\tilde \eps =10^{-6}$, convergence results are depicted in Table \ref{Table.ex4b}. In this case, positivity limiter is used during simulations many times. We can observe from  Table~\ref{Table.ex4b} that the positivity limiter does not degrade the accuracy and convergence results of the scheme atleast in $L^1$-error norm, which is consistent with the results observed in~\cite{zha-shu_10d}. Similar results have been obtained with WENO-JS and WENO-Z schemes (not shown here to save space).
}
\end{example}
\begin{table}
	\centering 
	\begin{tabular}{|c| c|c| c| c| c|  c| c| r|}
		\hline 
		$N$ & $L^{1}$-error & Order & $L^{2}$-error & Order & $L^{\infty}$-error& Order \\   
\hline
10  &  6.9266e-03 & --     &1.2016e-02 &-- &2.7763e-02 &\\ 
\hline 
20  &  2.2851e-04 & 4.92     &4.0967e-04 &4.87  &1.1932e-03 & 4.54    \\
\hline
40  &  7.0609e-06 & 5.02     &1.4424e-05 &4.83  &5.6716e-05 & 4.39    \\
\hline
80  &  2.0517e-07 & 5.10     &4.8369e-07 &4.90  &2.4188e-06 & 4.55   \\
\hline
160  &  6.1699e-09 & 5.06    &1.5446e-08 &4.97  &9.0870e-08 & 4.73    \\
\hline
320  &  1.9061e-10 & 5.02     &4.8046e-10 &5.01  &3.0445e-09 & 4.90    \\
\hline 
\end{tabular}
\caption{  Comparison of $L^1$-, $L^2$- and $L^{\infty}$-errors along with their convergence rates for low density problem Example \ref{ex:lowdensity} with $\tilde \eps=10^{-2}$ at time $T=0.5$.}
	\label{Table.ex4a}
\end{table}
\begin{table}
	\centering 
	\begin{tabular}{|c| c|c| c| c| c|  c| c| r|}
		\hline 
		$N$ & $L^{1}$-error & Order & $L^{2}$-error & Order & $L^{\infty}$-error& Order \\   
\hline
10  &  6.9886e-03 & --     &1.2155e-02 &-- &2.8640e-02 &\\ 
\hline 
20  &  2.3453e-04 & 4.90     &4.2055e-04 &4.85  &1.2504e-03 & 4.52    \\
\hline
40  &  7.6188e-06 & 4.94     &1.5722e-05 &4.74  &6.5124e-05 & 4.26    \\
\hline
80  &  2.4767e-07 & 4.94     &5.9952e-07 &4.71  &3.2944e-06 & 4.31    \\
\hline
160  &  7.2357e-09 & 5.10     &2.2418e-08 &4.74  &1.6556e-07 & 4.31    \\
\hline
320  &  2.1793e-10 & 5.05    &8.5867e-10 &4.71  &8.3254e-09 & 4.31    \\
\hline
\end{tabular}
\caption{   Comparison of $L^1$-, $L^2$- and $L^{\infty}$-errors along with their convergence rates for low density problem Example \ref{ex:lowdensity} with $\tilde \eps=10^{-6}$ at time $T=0.5$.}
	\label{Table.ex4b}
\end{table}

\subsection{Test cases with discontinuities}
\begin{figure}
\begin{center}
\begin{tabular}{ccc}
\includegraphics[width=0.32\textwidth]{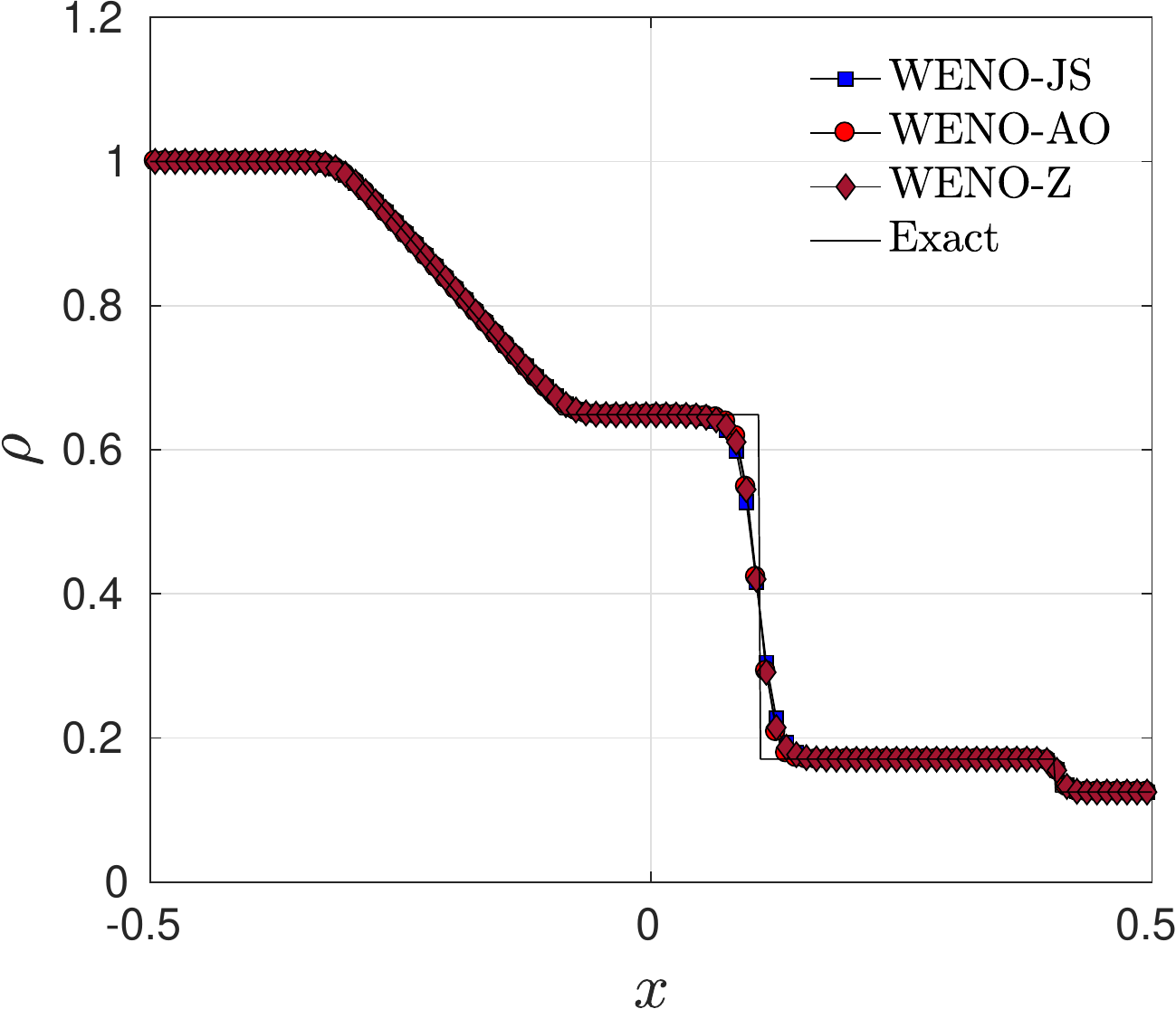} &
 \includegraphics[width=0.32\textwidth]{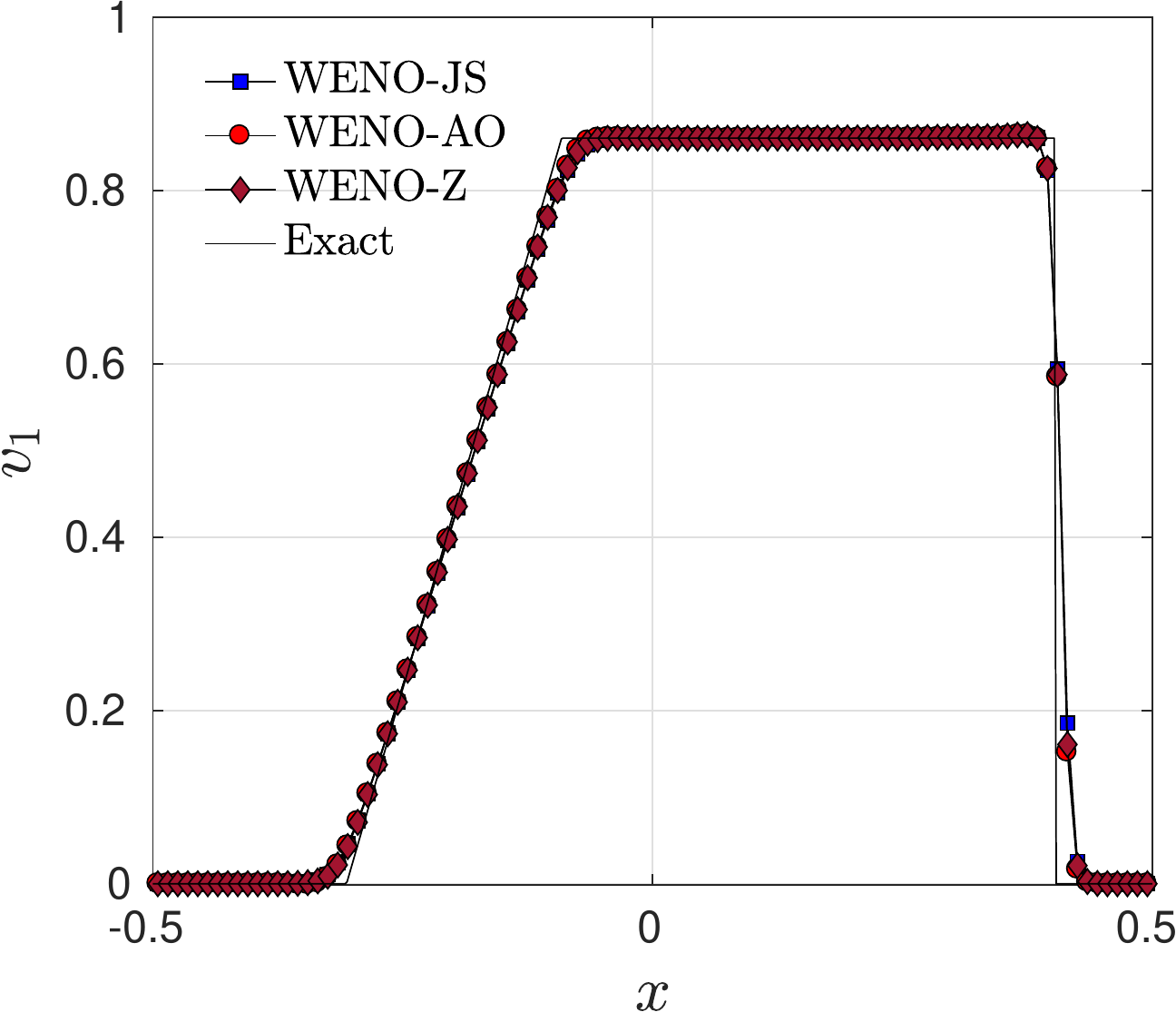} &
  \includegraphics[width=0.32\textwidth]{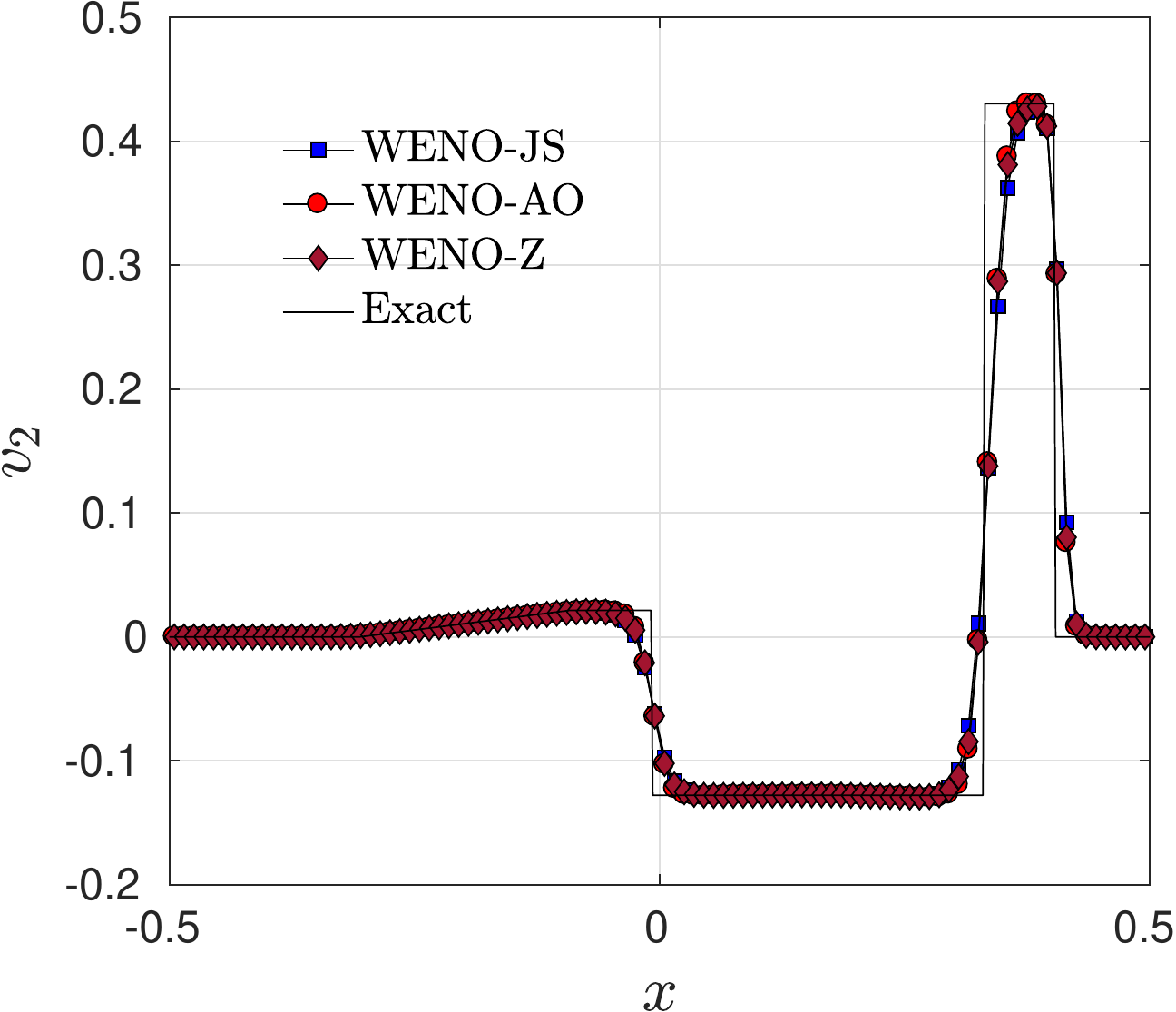} \\
 (a) $\rho$ & (b) $v_1$ &   (c) $v_2$  \\
 \includegraphics[width=0.32\textwidth]{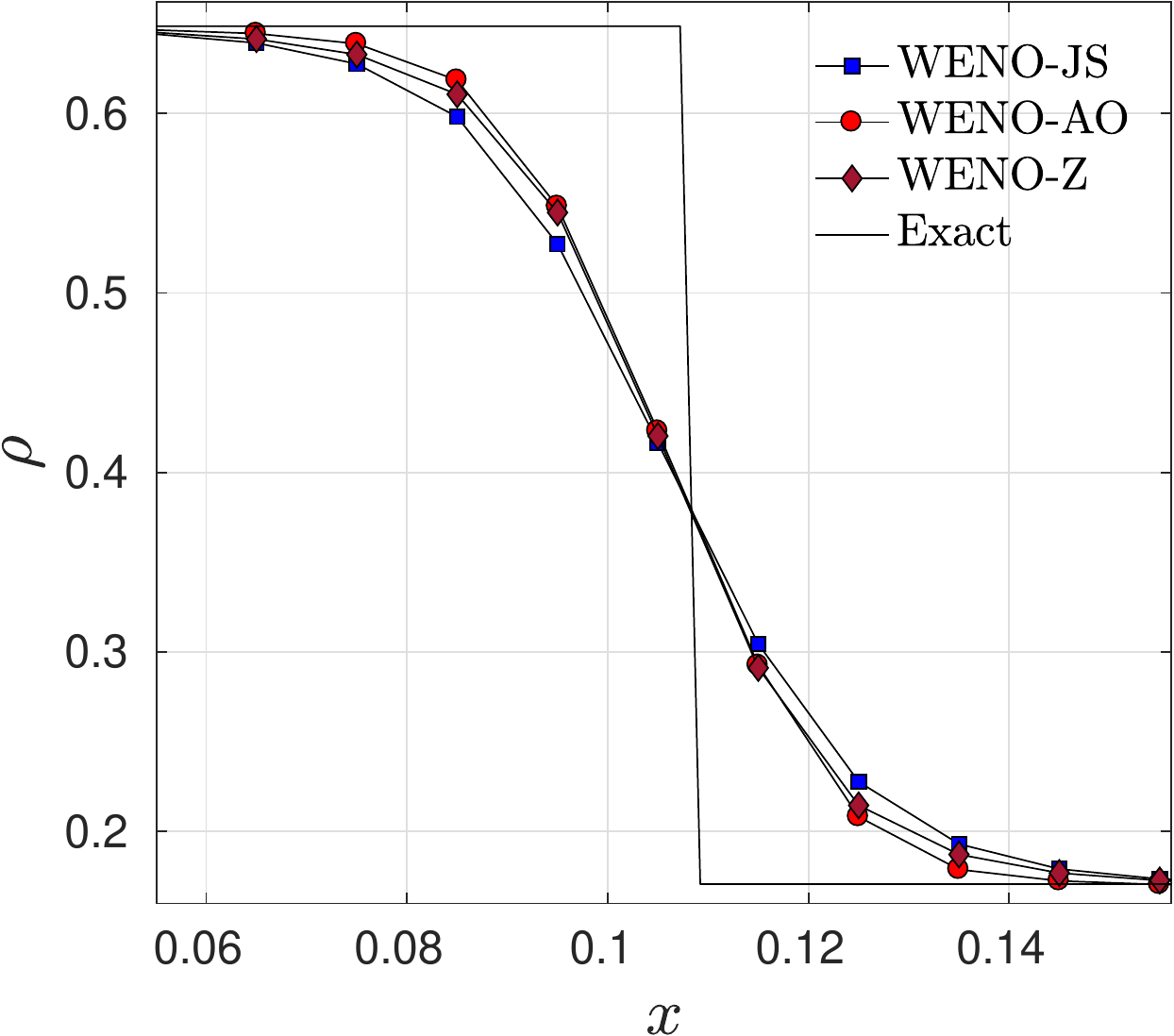} &
 \includegraphics[width=0.32\textwidth]{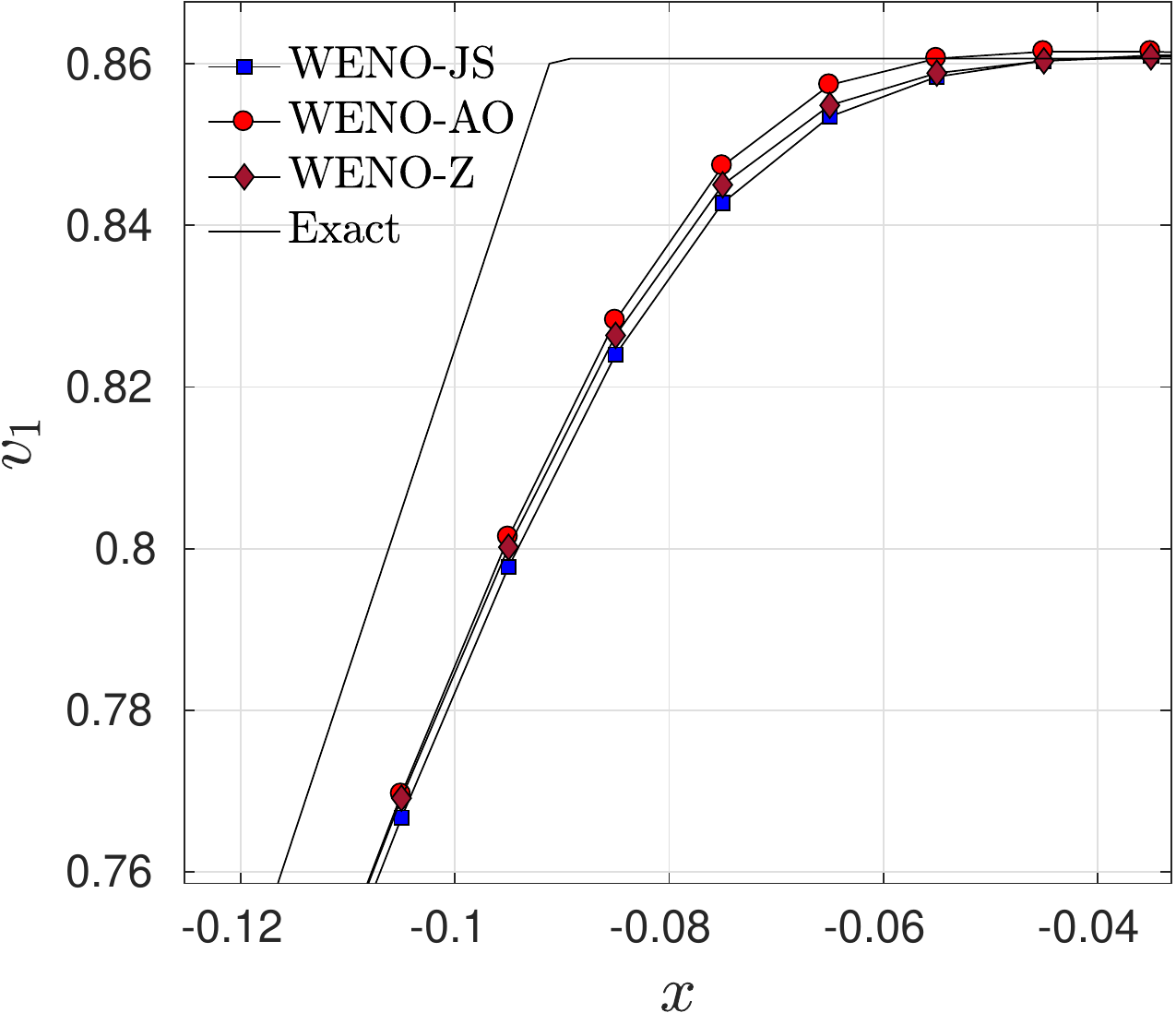} &
 \includegraphics[width=0.32\textwidth]{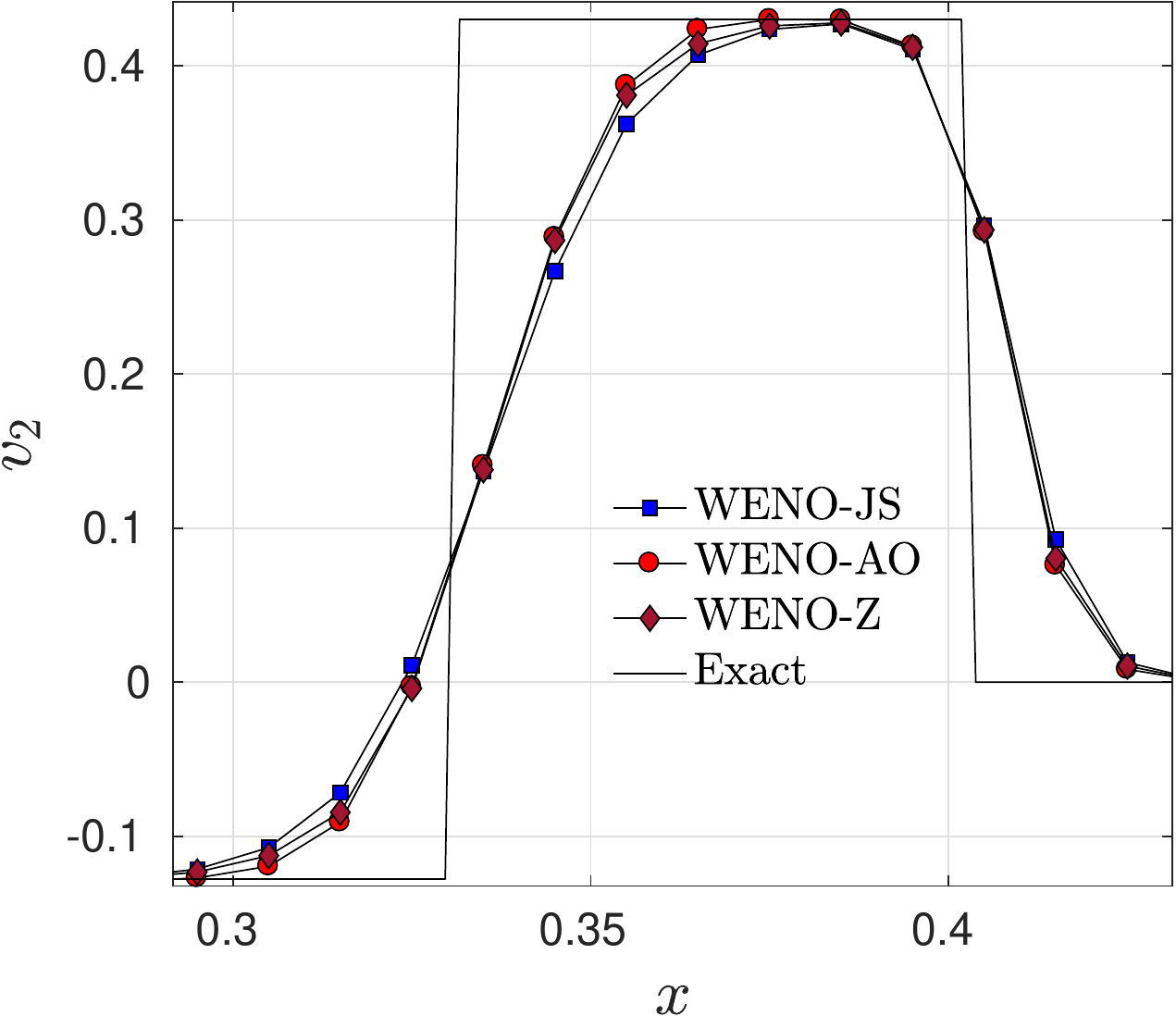}\\
(d) Zoom of $\rho$ &  (e) Zoom of $v_1$ & (f) Zoom of $v_2$
\end{tabular}
\end{center}
 \caption{Comparison of $\rho$, $v_1$ and $v_2$ obtained using WENO-JS, WENO-Z and WENO-AO schemes with the exact solution for Example \ref{ex:sod}. }
 \label{fig:sod1}
\end{figure}

\begin{figure}
\begin{center}
\begin{tabular}{ccc}
\includegraphics[width=0.32\textwidth]{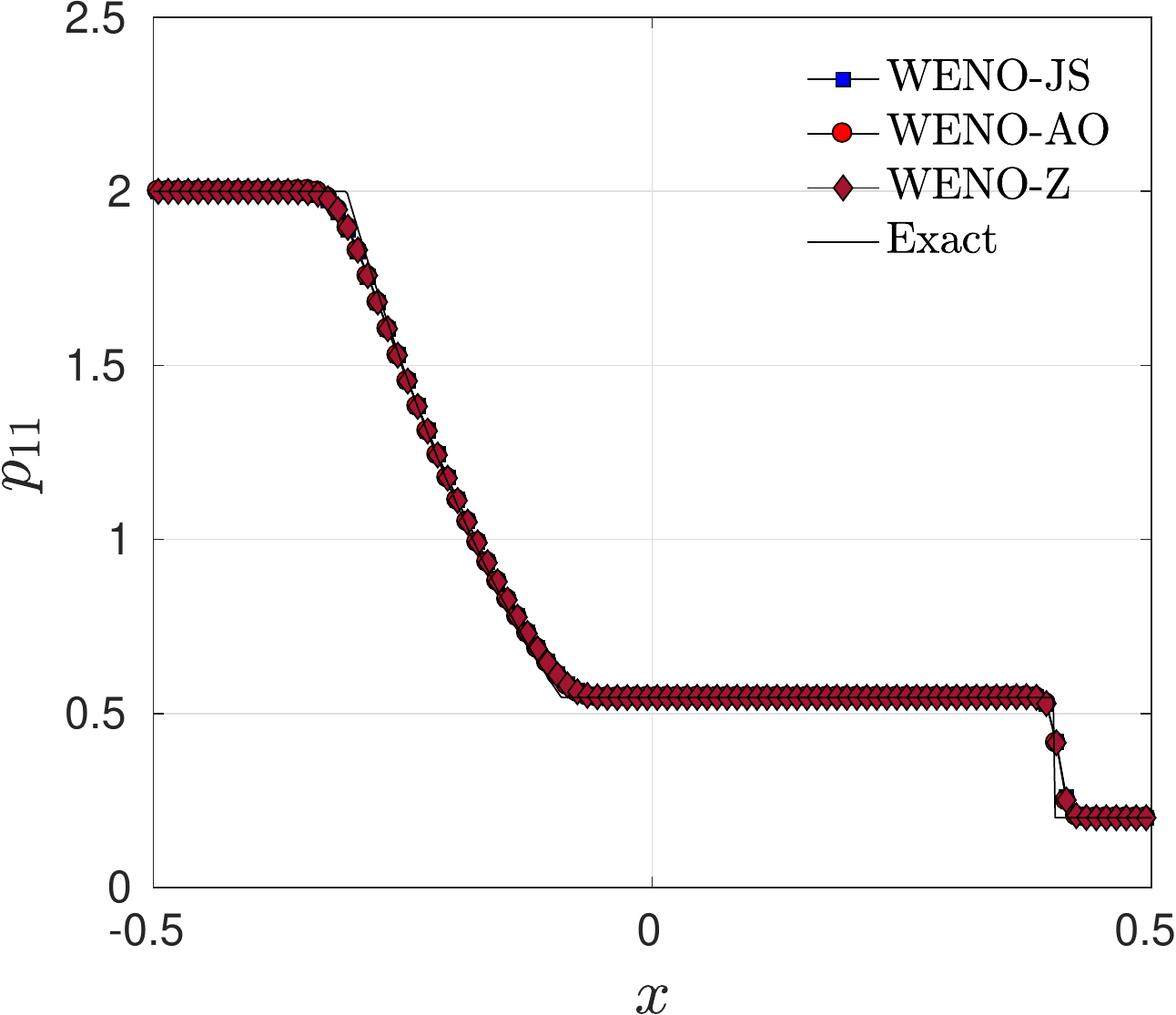} &
 \includegraphics[width=0.32\textwidth]{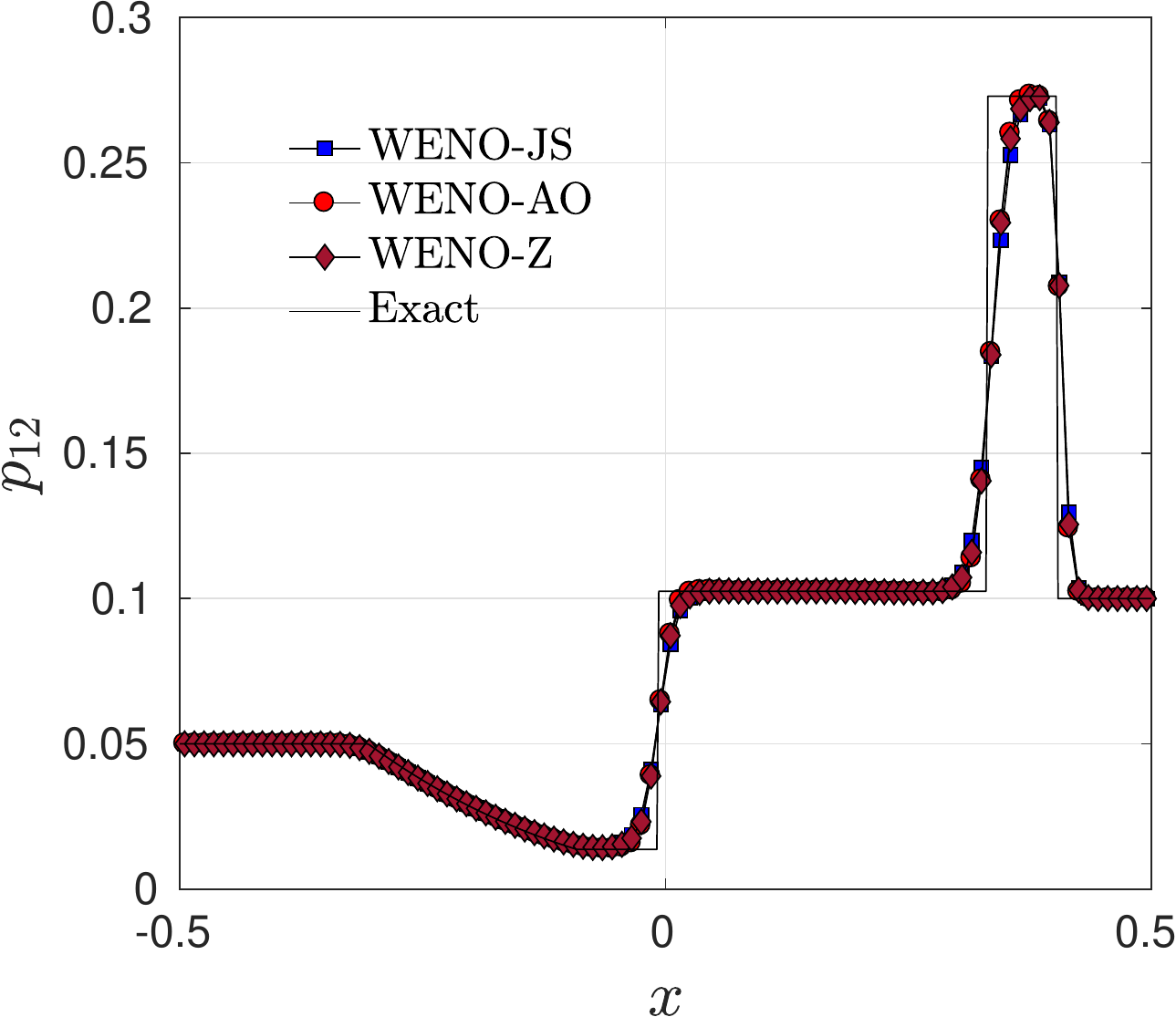} &
  \includegraphics[width=0.32\textwidth]{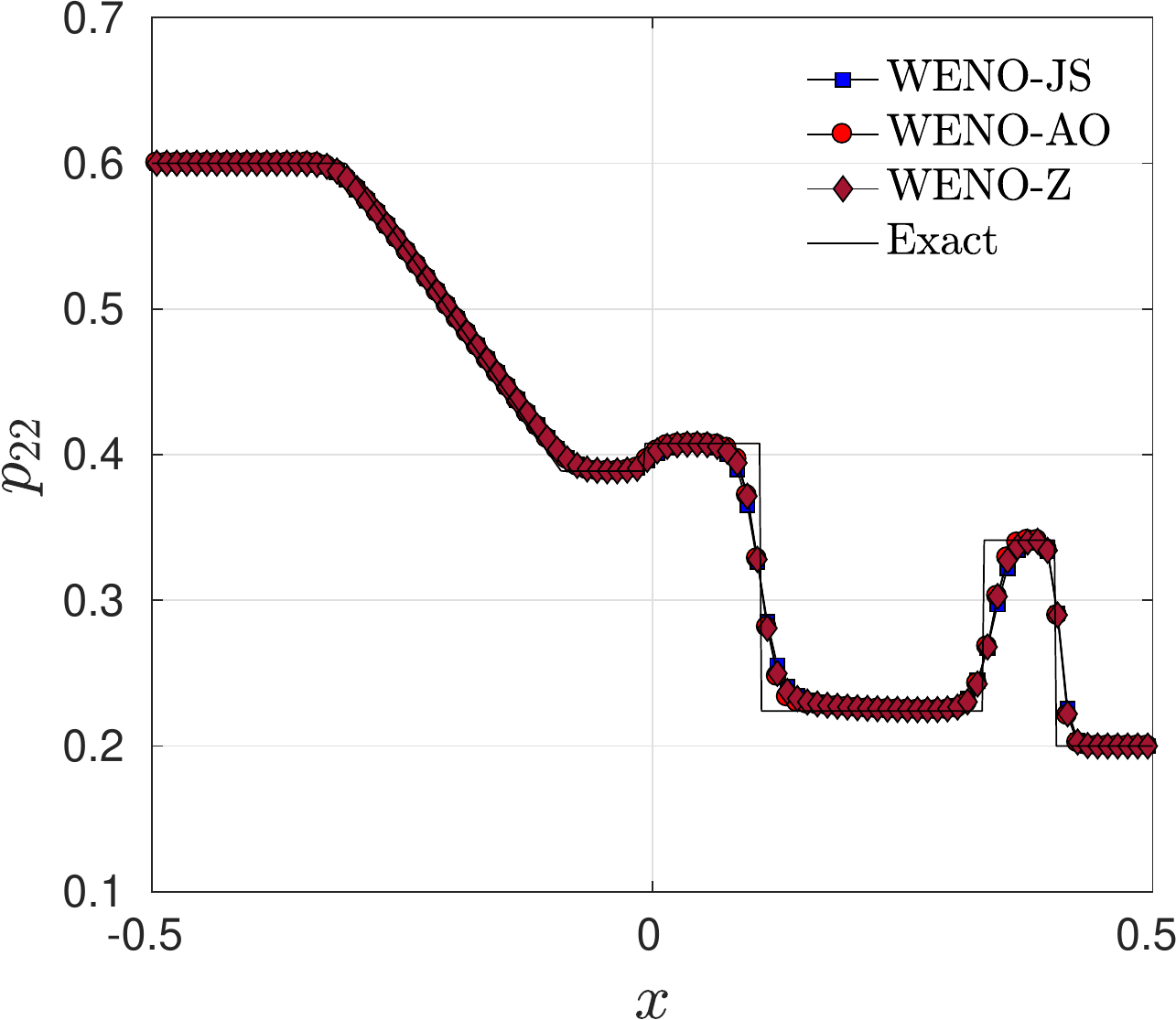} \\
 (a) $p_{11}$ & (b) $p_{12}$ &   (c) $p_{22}$  \\
 \includegraphics[width=0.32\textwidth]{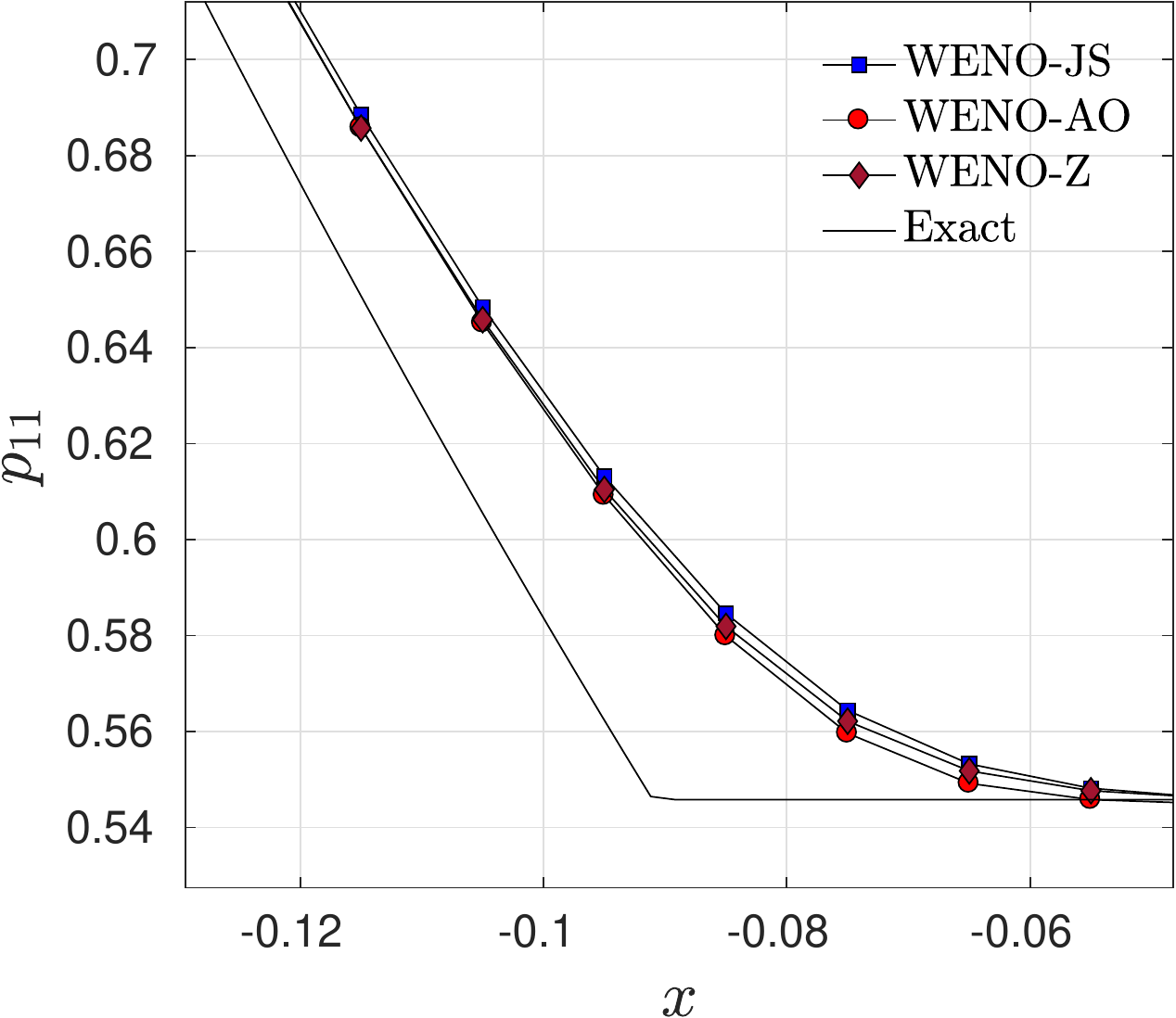} &
 \includegraphics[width=0.32\textwidth]{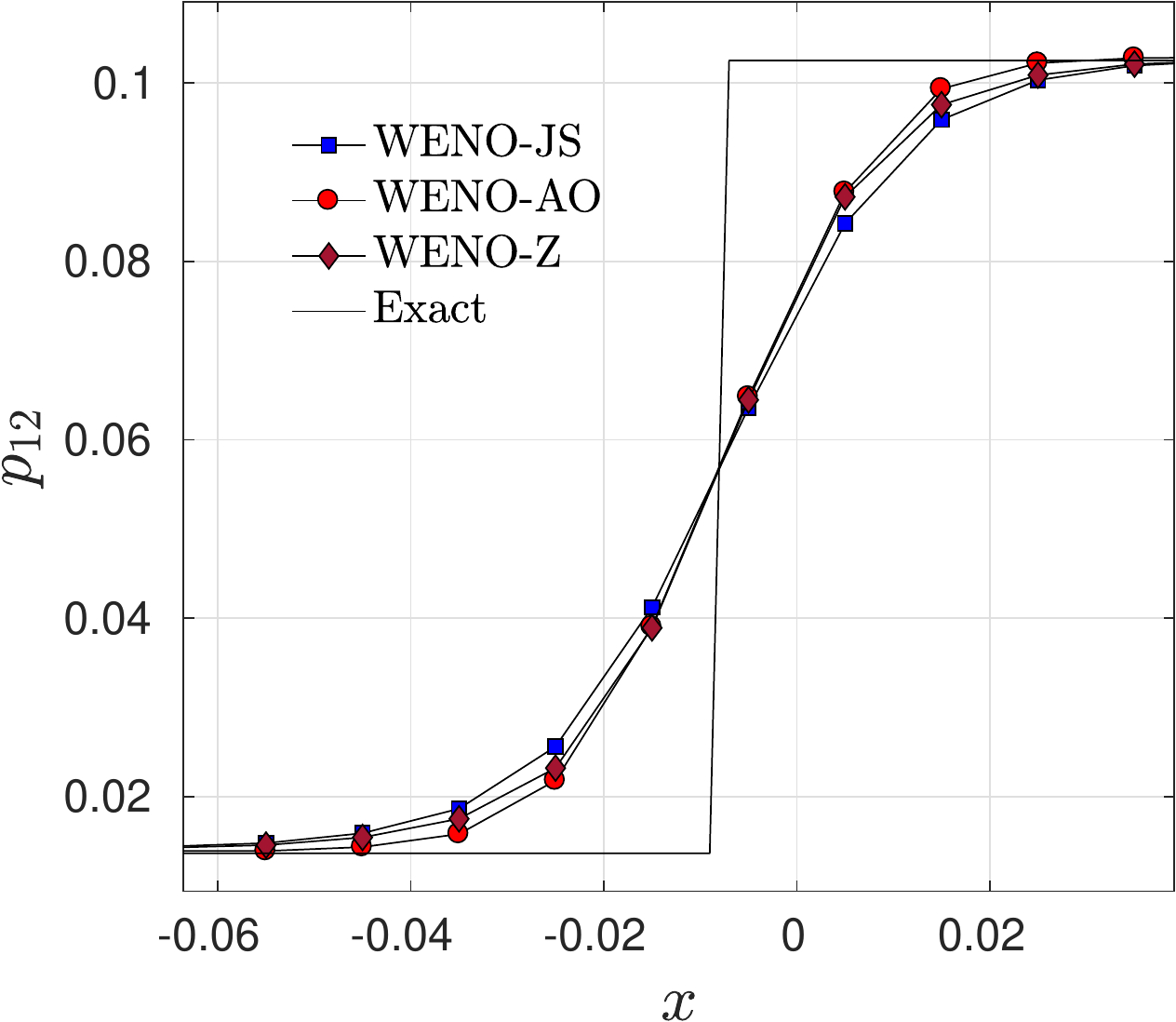} &
 \includegraphics[width=0.32\textwidth]{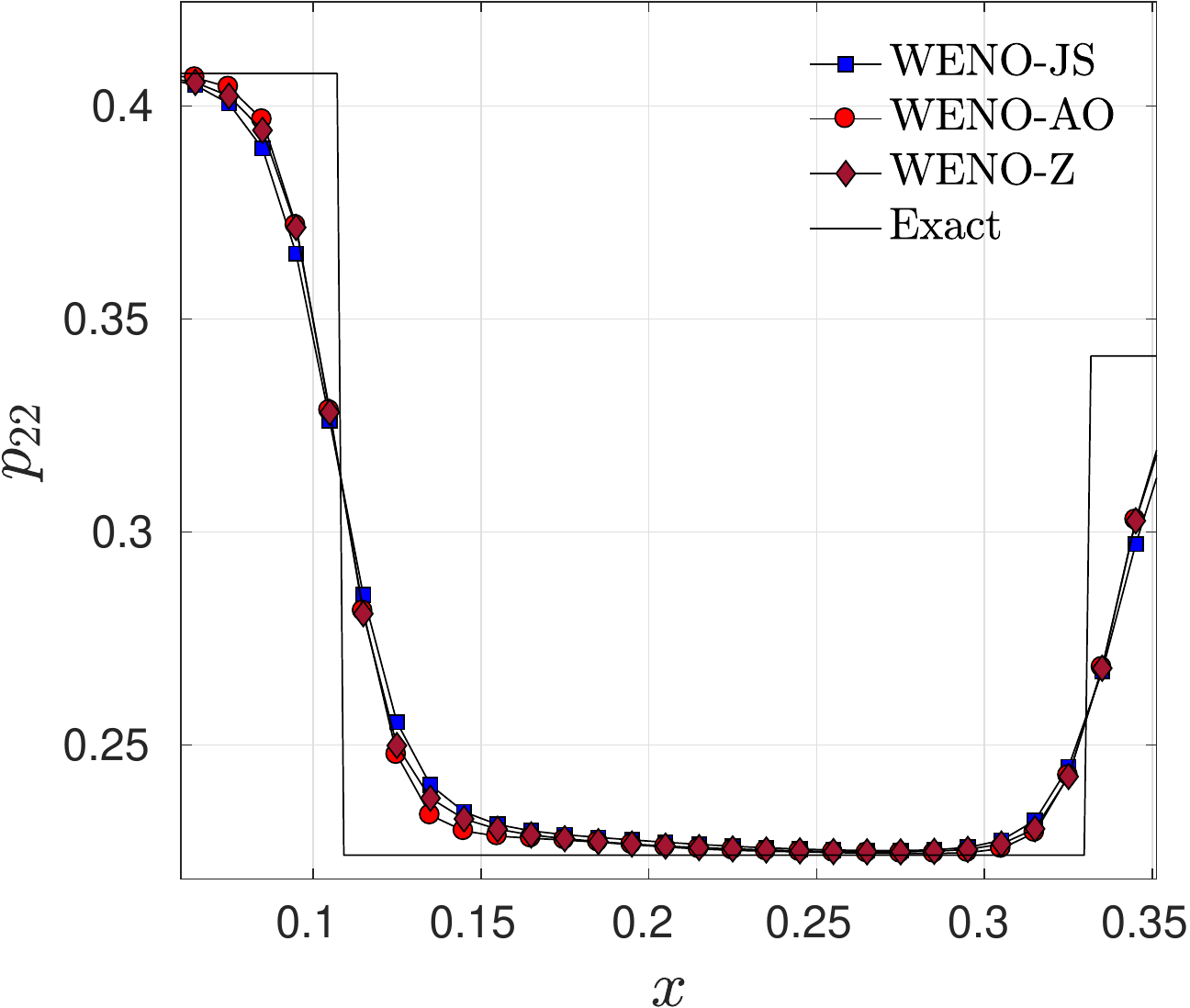}\\
(d) zoom of $p_{11}$  &  (e) zoom of $p_{12}$ & (f) zoom of $p_{22}$
\end{tabular}
\end{center}
 \caption{Comparison of pressure components obtained using WENO-JS, WENO-Z and WENO-AO schemes with the exact solution for Example \ref{ex:sod}.}
 \label{fig:sod2}
\end{figure}

\begin{figure}
\begin{center}
\begin{tabular}{ccc}
\includegraphics[width=0.32\textwidth]{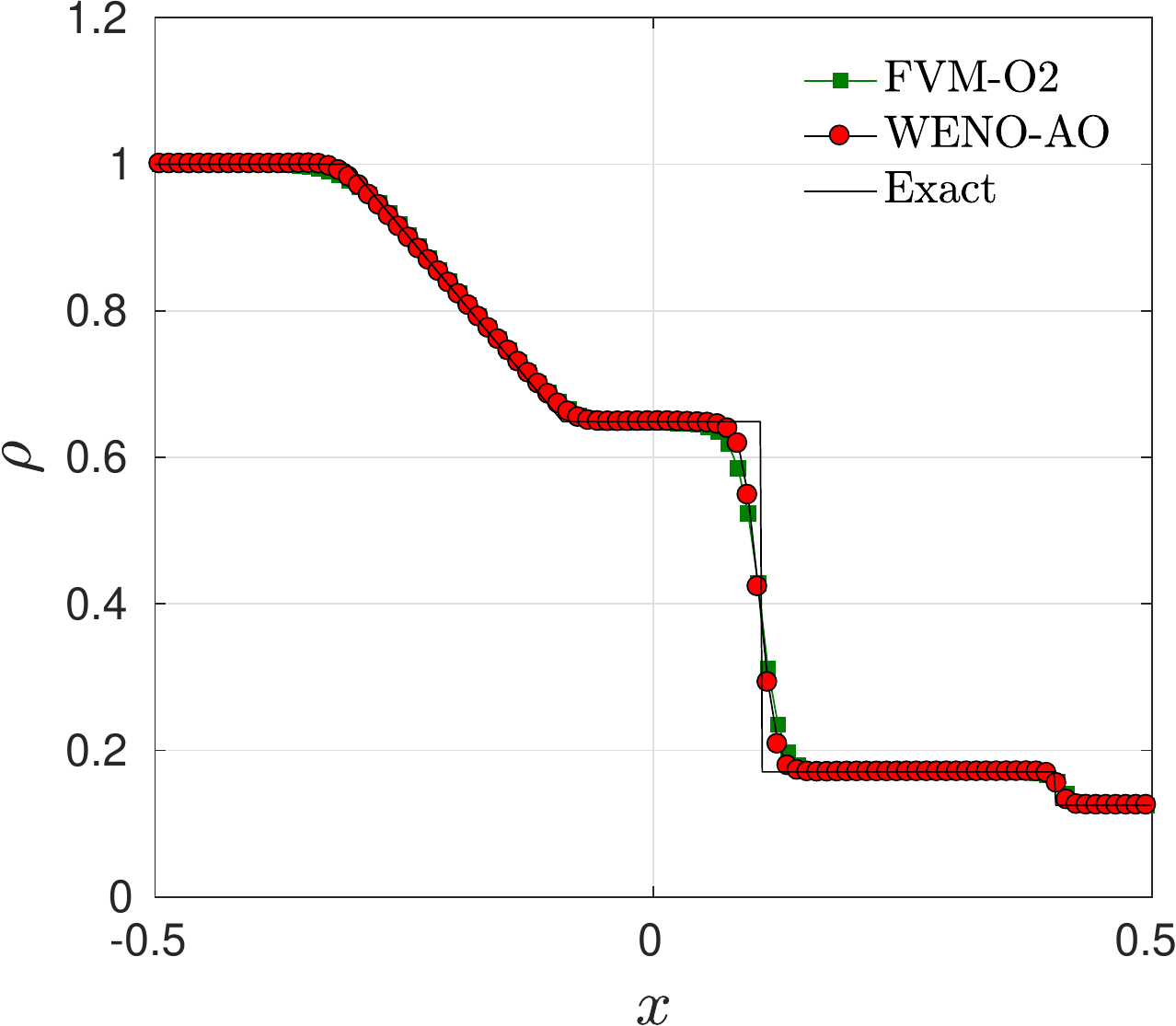} &
 \includegraphics[width=0.32\textwidth]{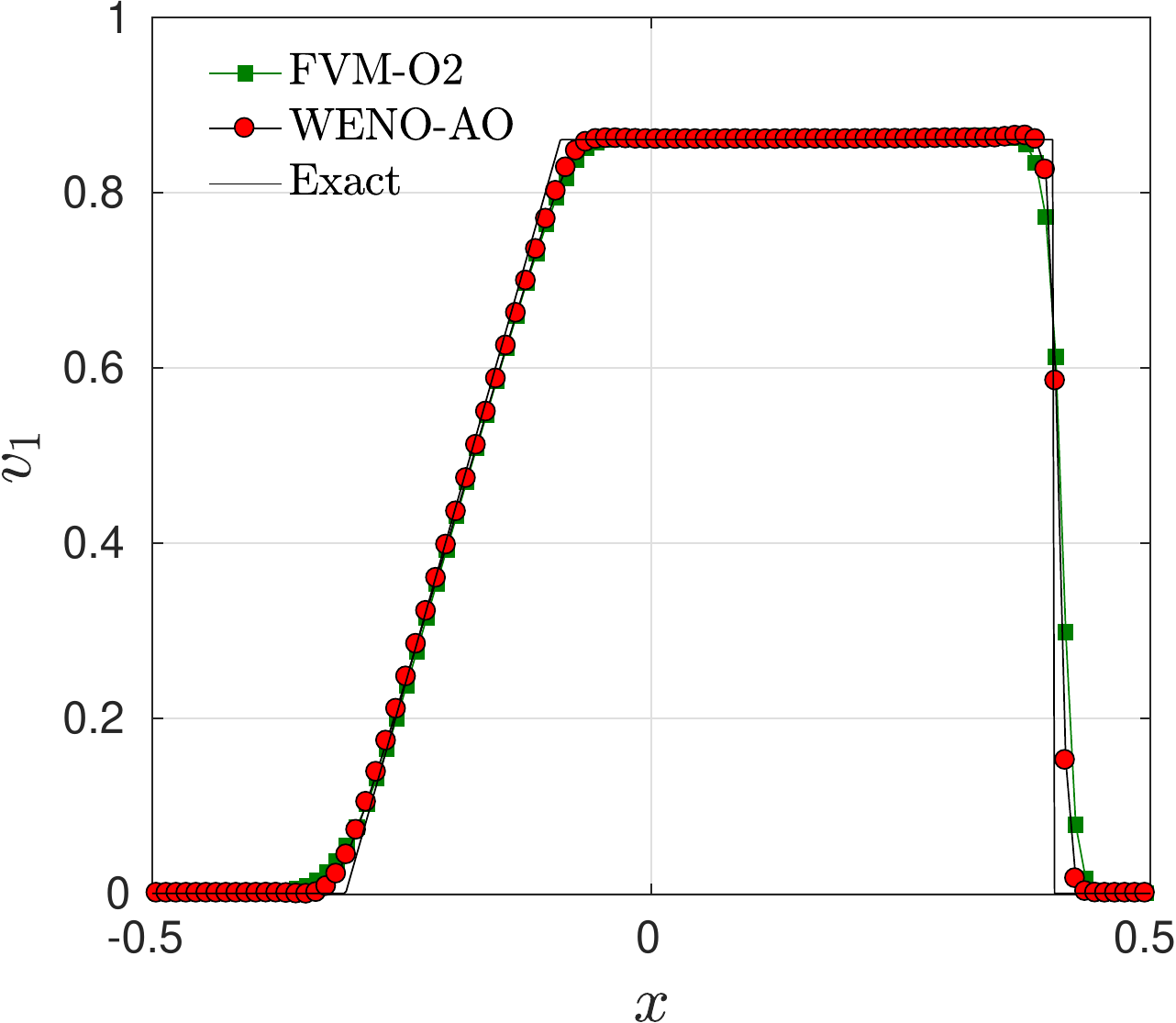} &
  \includegraphics[width=0.32\textwidth]{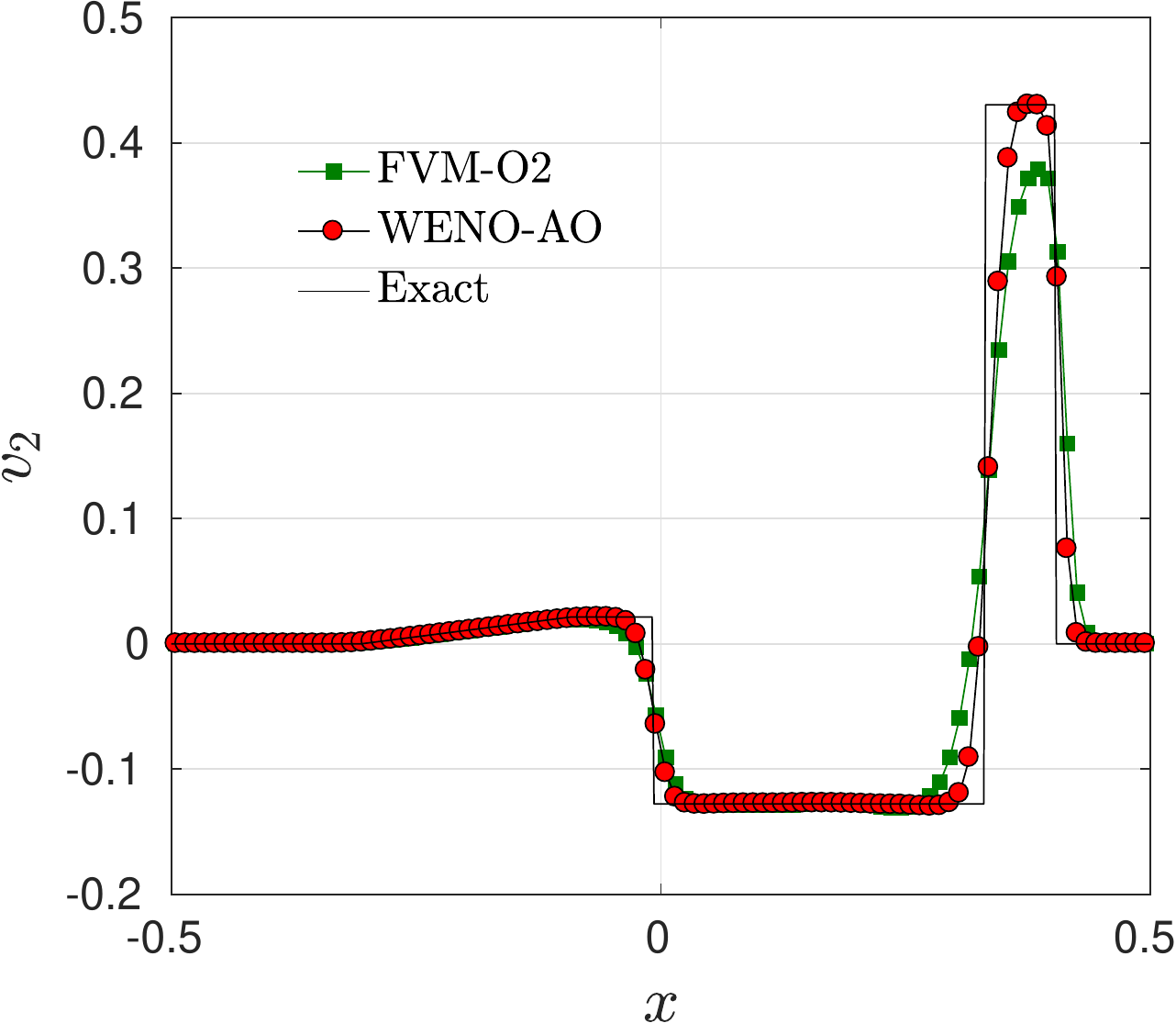} \\
 (a) $\rho$ & (b) $v_1$ &   (c) $v_2$  \\
 \includegraphics[width=0.32\textwidth]{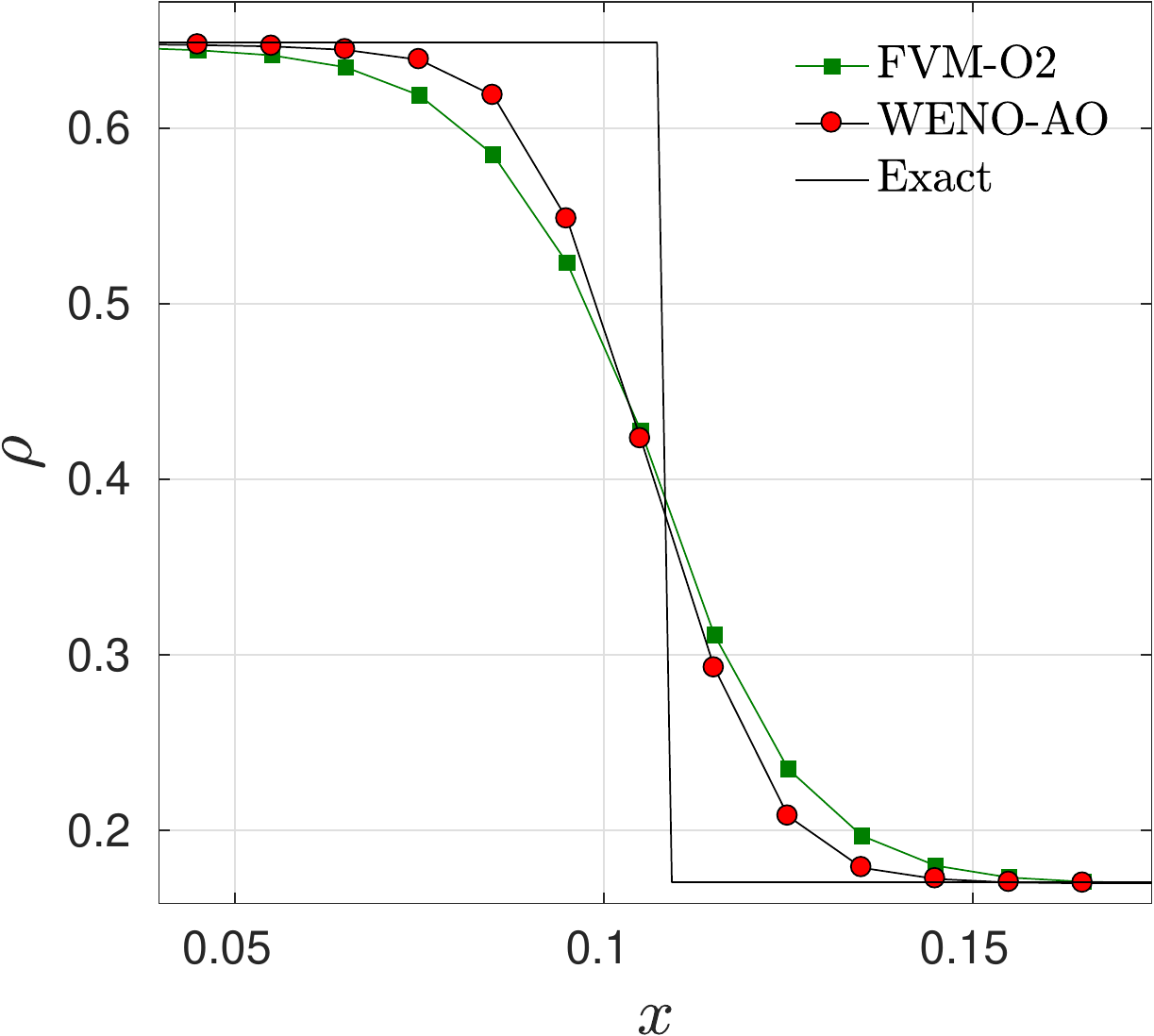} &
 \includegraphics[width=0.32\textwidth]{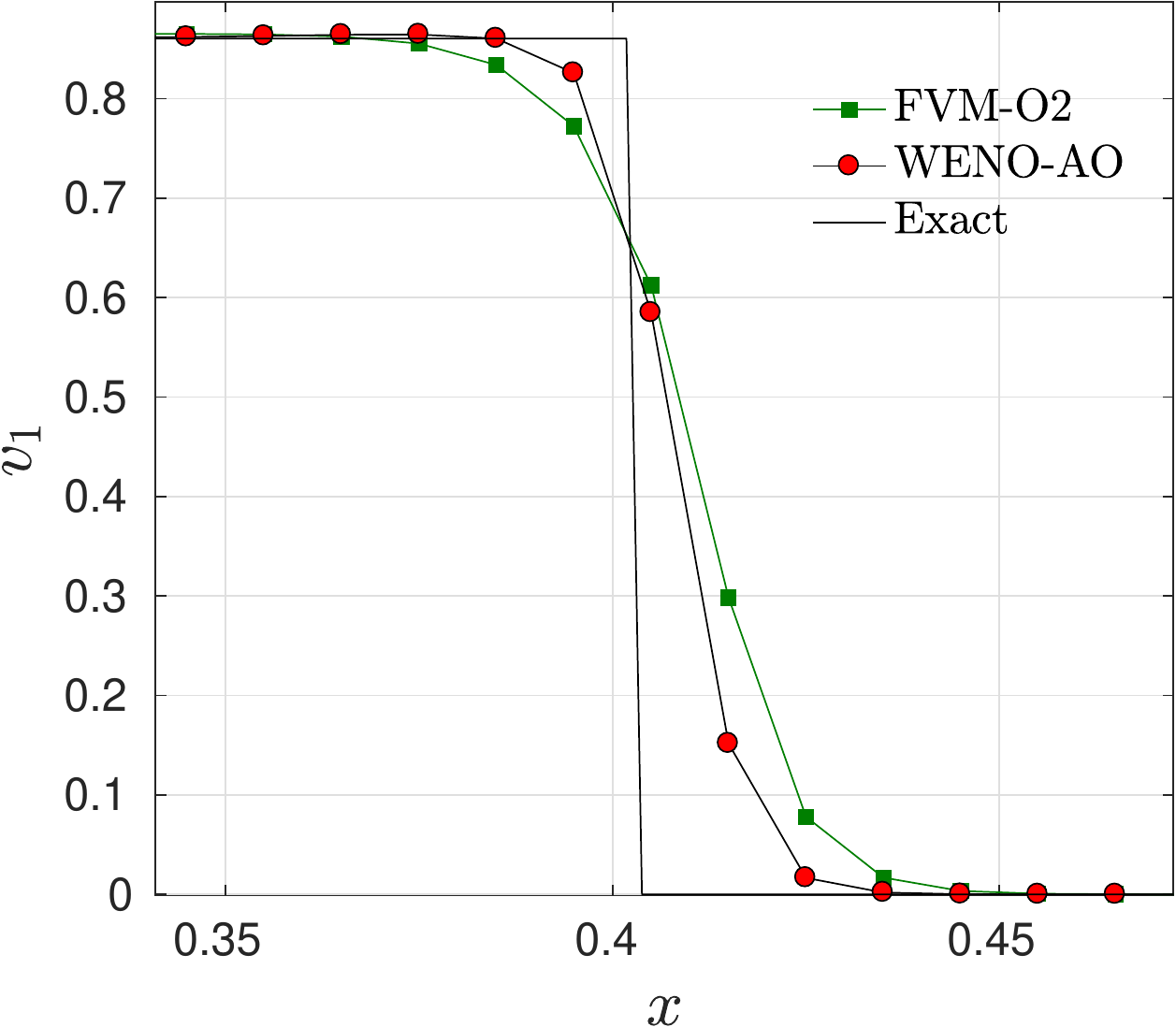} &
 \includegraphics[width=0.32\textwidth]{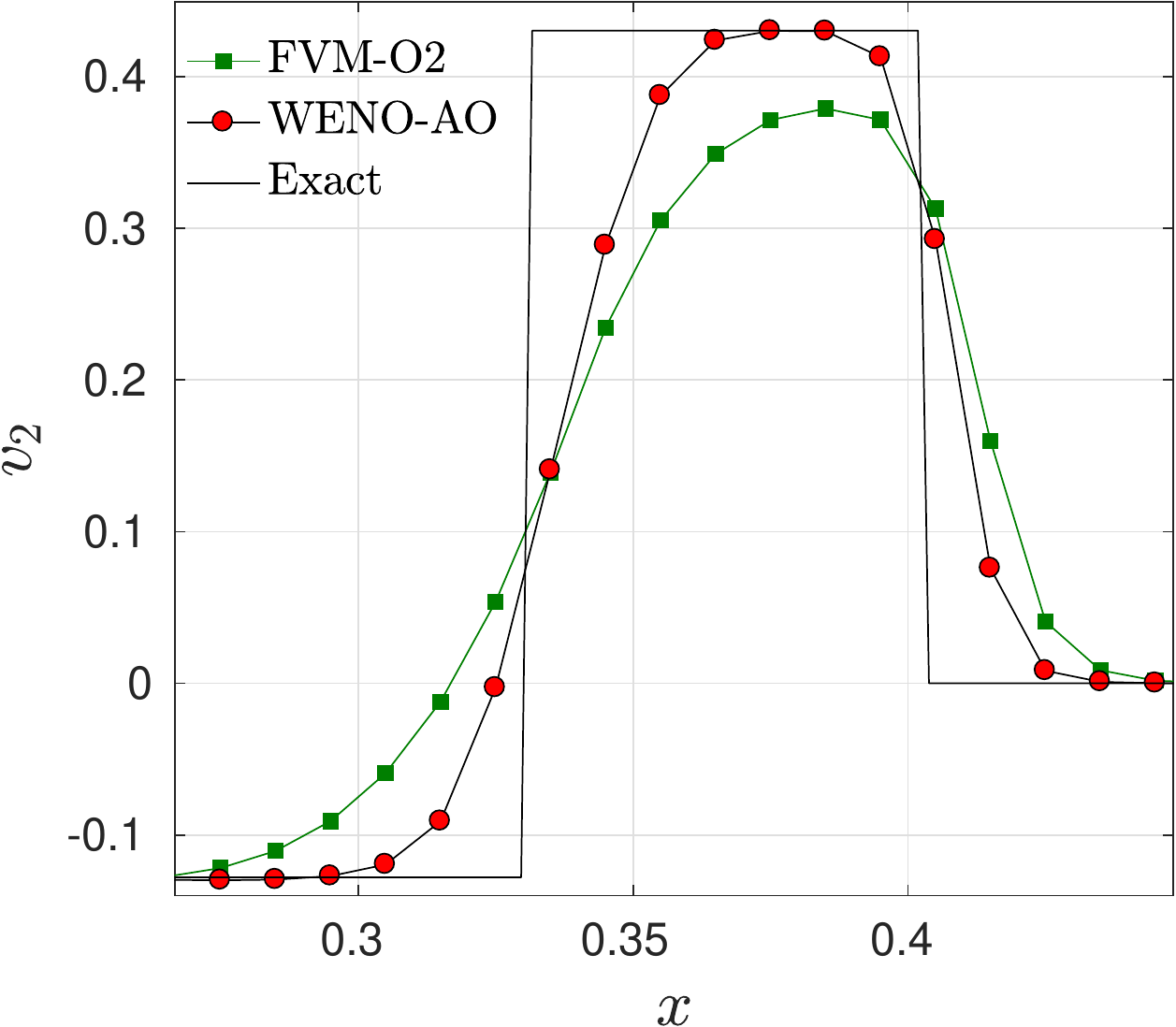}\\
(d) Zoom of $\rho$ &  (e) Zoom of $v_1$ & (f) Zoom of $v_2$
\end{tabular}
\end{center}
 \caption{Comparison of $\rho$, $v_1$ and $v_2$ obtained using WENO-AO scheme with second order finite volume method~\cite{meenaKumarFVM} for Example \ref{ex:sod}.}
 \label{fig:sodfvm}
\end{figure}

\begin{example}\label{ex:sod}{\rm (Sod shock tube test)
This test case involves the formation of shock, rarefaction, and contact wave. 
 The initial data, which has discontinuity at $x=0$ over the domain $[-0.5,0.5]$ is given by
  \[
  \bold{V}=\begin{cases}
            (1, \ 0, \ 0, \ 2, \ 0.05, \ 0.6), & \mbox{if}~~ x\leq 0\\
            (0.125, \ 0, \ 0, \ 0.2, \ 0.1, \ 0.2), & \mbox{if}~~ x>0
           \end{cases}
\]
}
\end{example}
The numerical solutions are computed at final time $T=0.125$ using 100 cells and outflow boundary conditions. In Figure \ref{fig:sod1} and \ref{fig:sod2}, we have compared the solutions of the primitive variables obtained using WENO-JS, WENO-Z, and WENO-AO schemes with the exact solution. Since this problem does not have low density or pressure, the
positivity limiter is not called in any WENO scheme. We observed from the Figure \ref{fig:sod1} and \ref{fig:sod2} that, all the proposed WENO schemes capture the shock wave in a non-oscillatory fashion. The resolution of WENO-AO scheme across the discontinuities is better than that of WENO-JS and WENO-Z schemes. In Figure \ref{fig:sodfvm}, we have compared our proposed algorithms with published results \cite{meenaKumarFVM} for density and velocity. We can easily observe that WENO-AO scheme has better resolution in comparison to the second order finite volume scheme. Similar observations are seen in the case of other variables (not shown here).
\begin{figure}
\begin{center}
\begin{tabular}{ccc}
\includegraphics[width=0.32\textwidth]{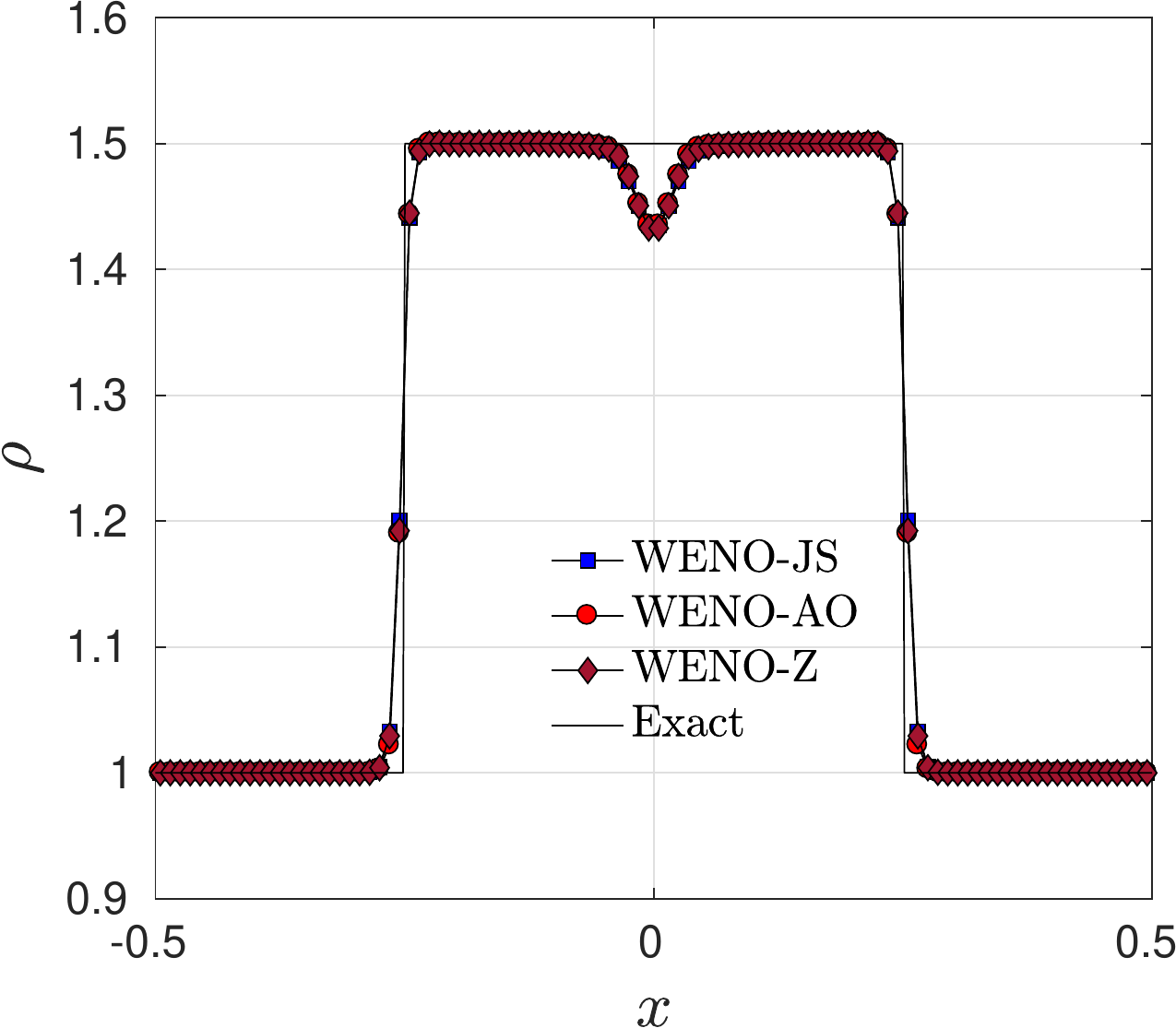} &
 \includegraphics[width=0.32\textwidth]{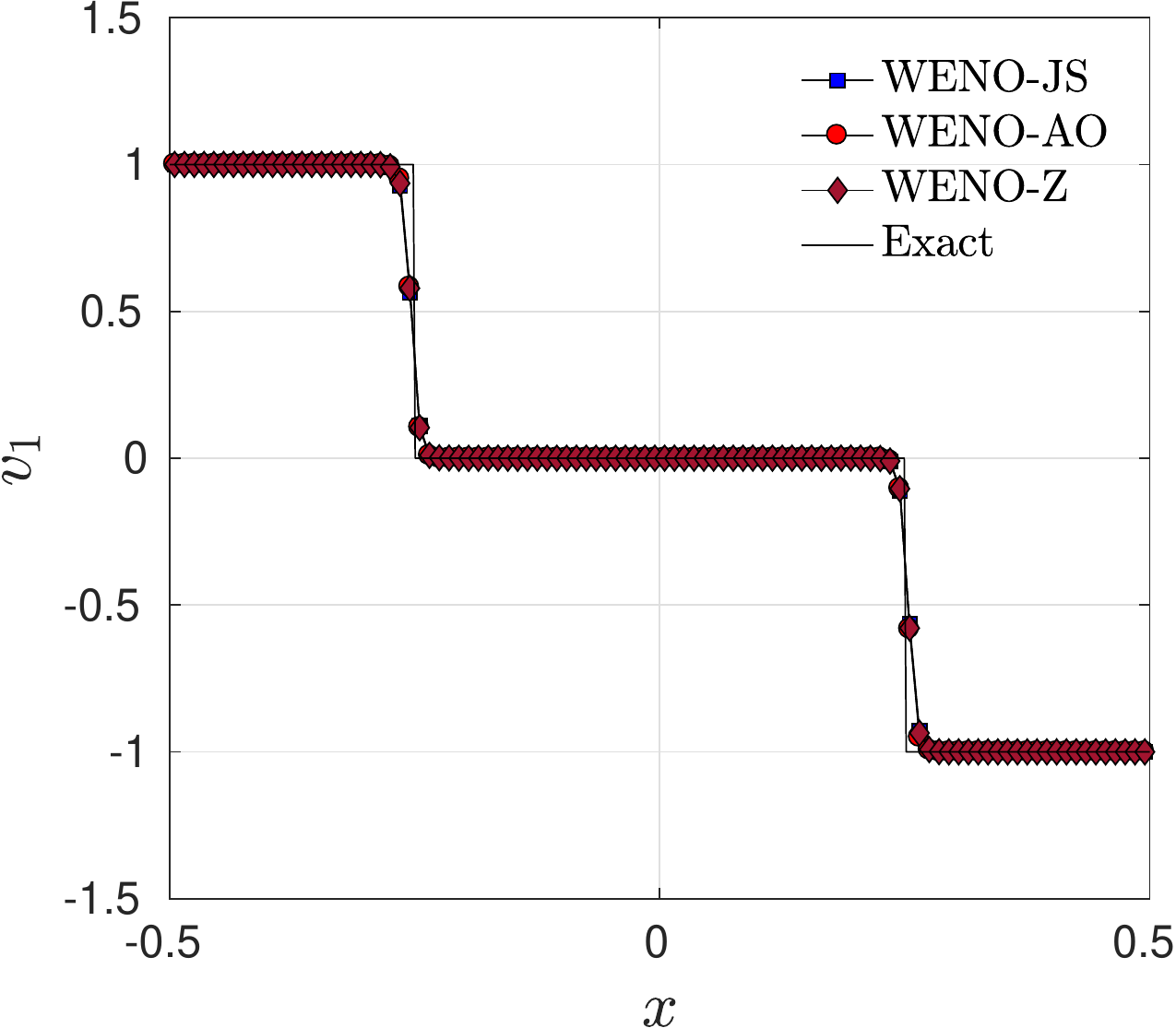} &
  \includegraphics[width=0.32\textwidth]{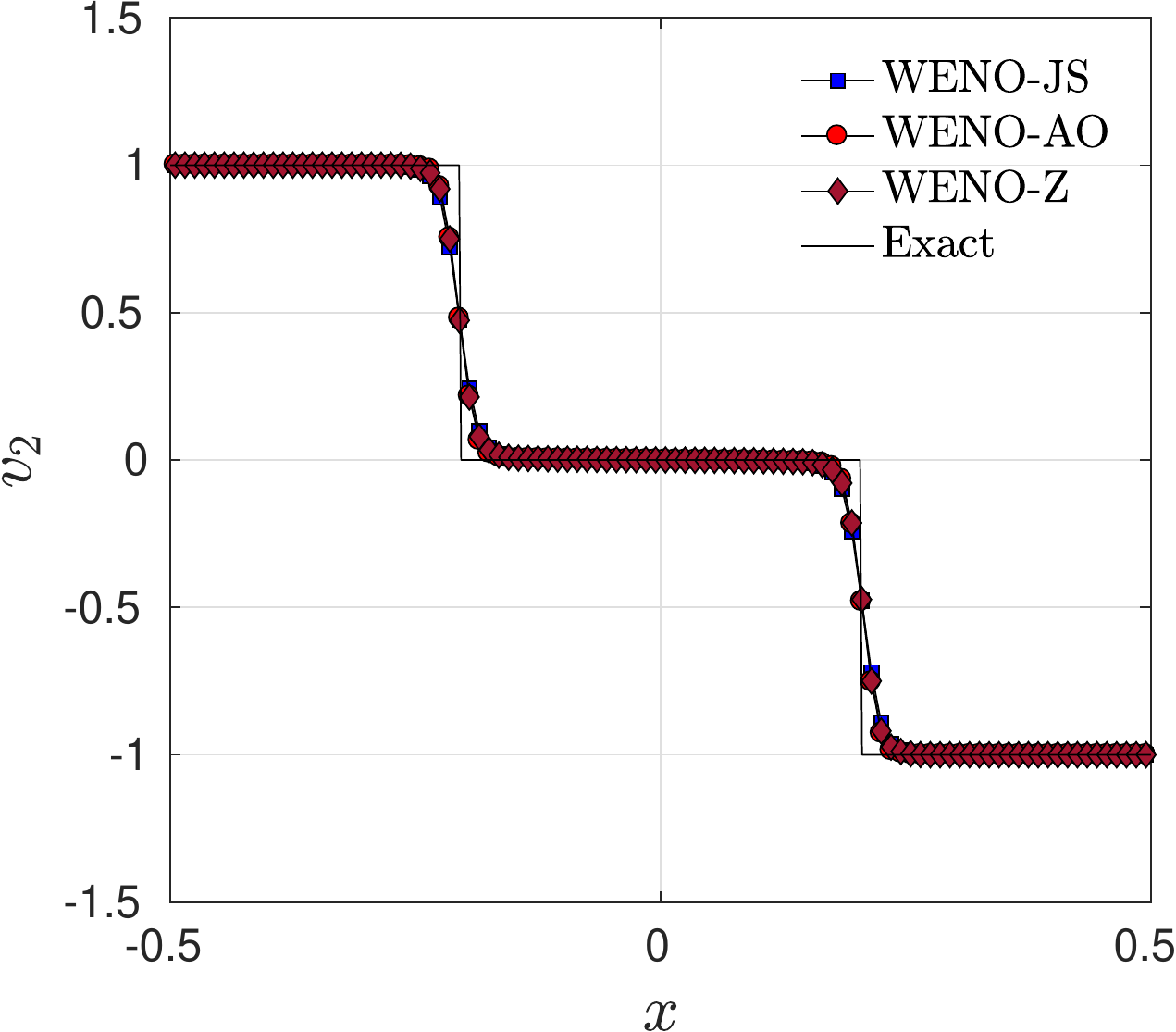} \\
 (a) $\rho$ & (b) $v_1$ &   (c) $v_2$  \\
 \includegraphics[width=0.32\textwidth]{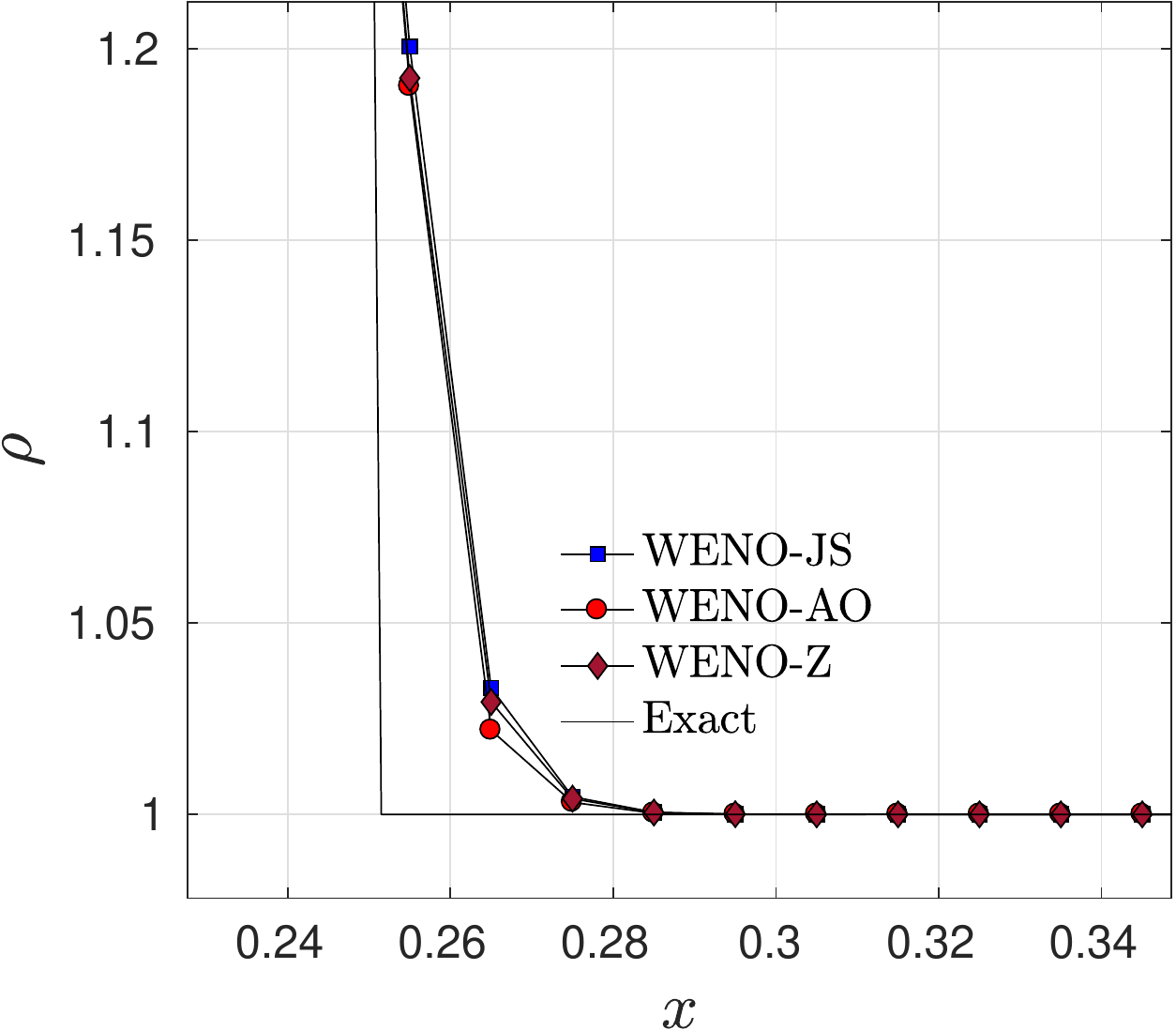} &
 \includegraphics[width=0.32\textwidth]{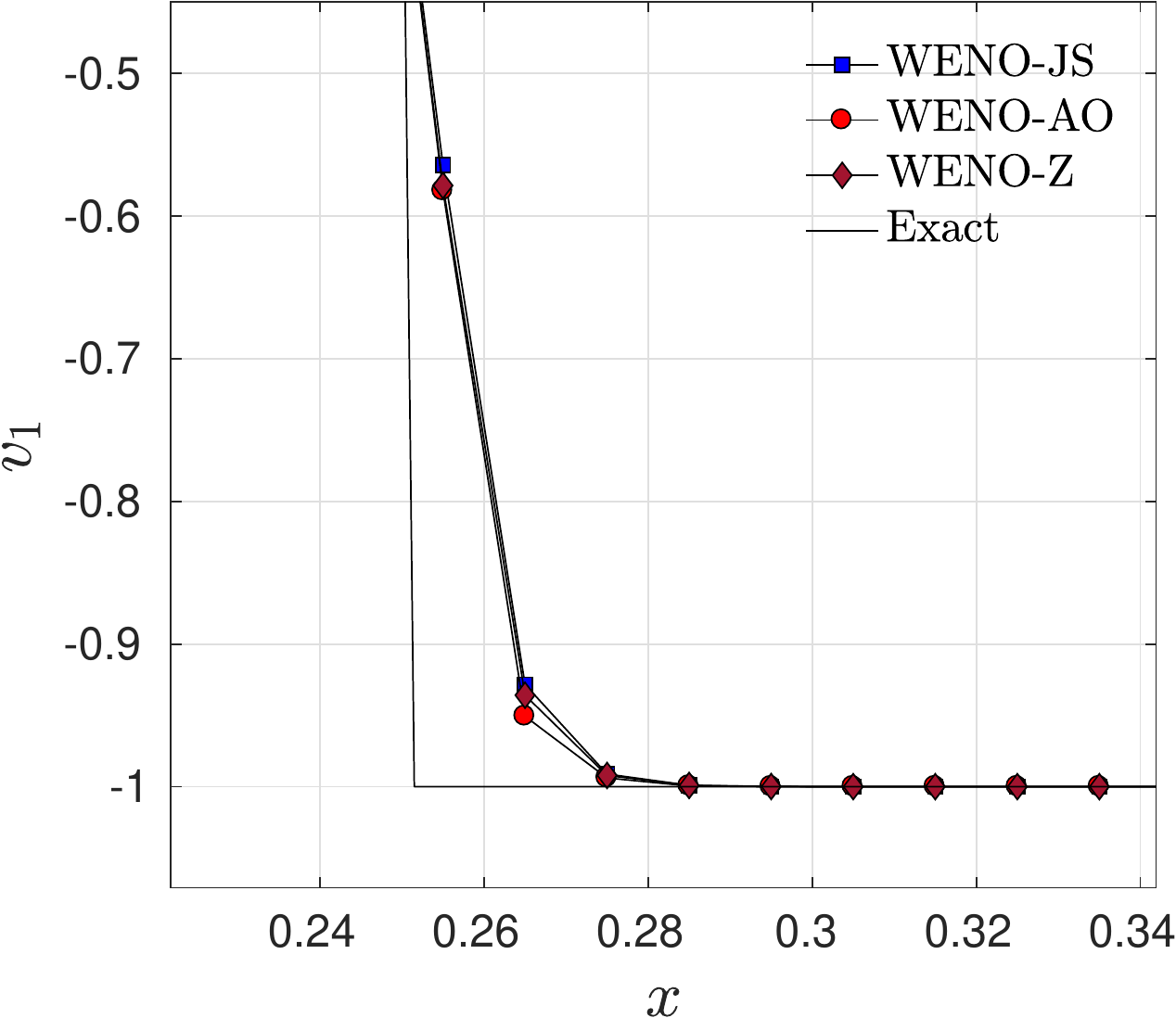} &
 \includegraphics[width=0.32\textwidth]{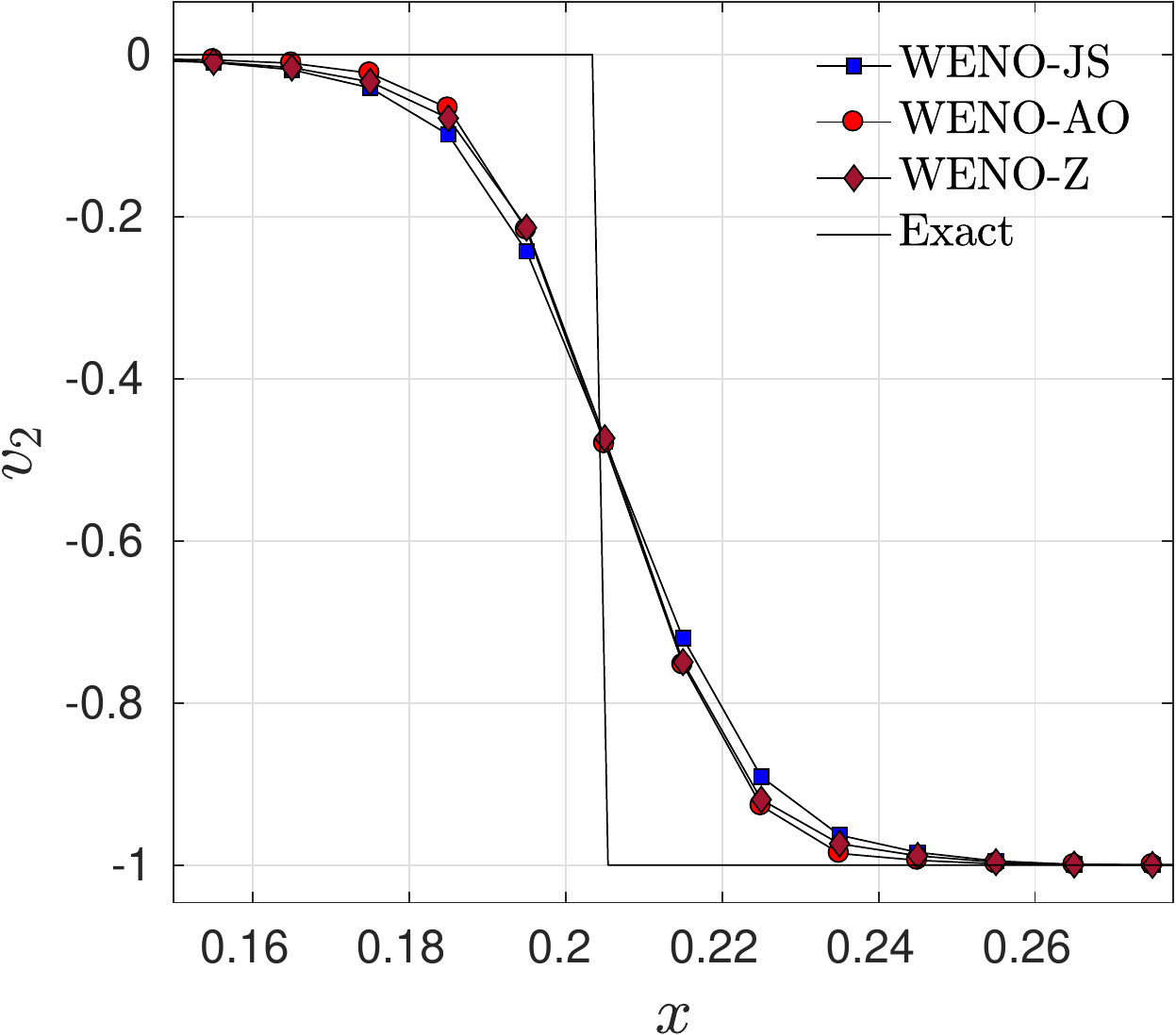}\\
(d) Zoom of $\rho$ &  (e) Zoom of $v_1$ & (f) Zoom of $v_2$
\end{tabular}
\end{center}
 \caption{Comparison of $\rho$, $v_1$ and $v_2$ obtained using WENO-JS, WENO-Z and WENO-AO schemes with the exact solution for Example \ref{ex:twoshock}.}
 \label{fig:twoshock1}
\end{figure}
\begin{figure}
\begin{center}
\begin{tabular}{ccc}
\includegraphics[width=0.32\textwidth]{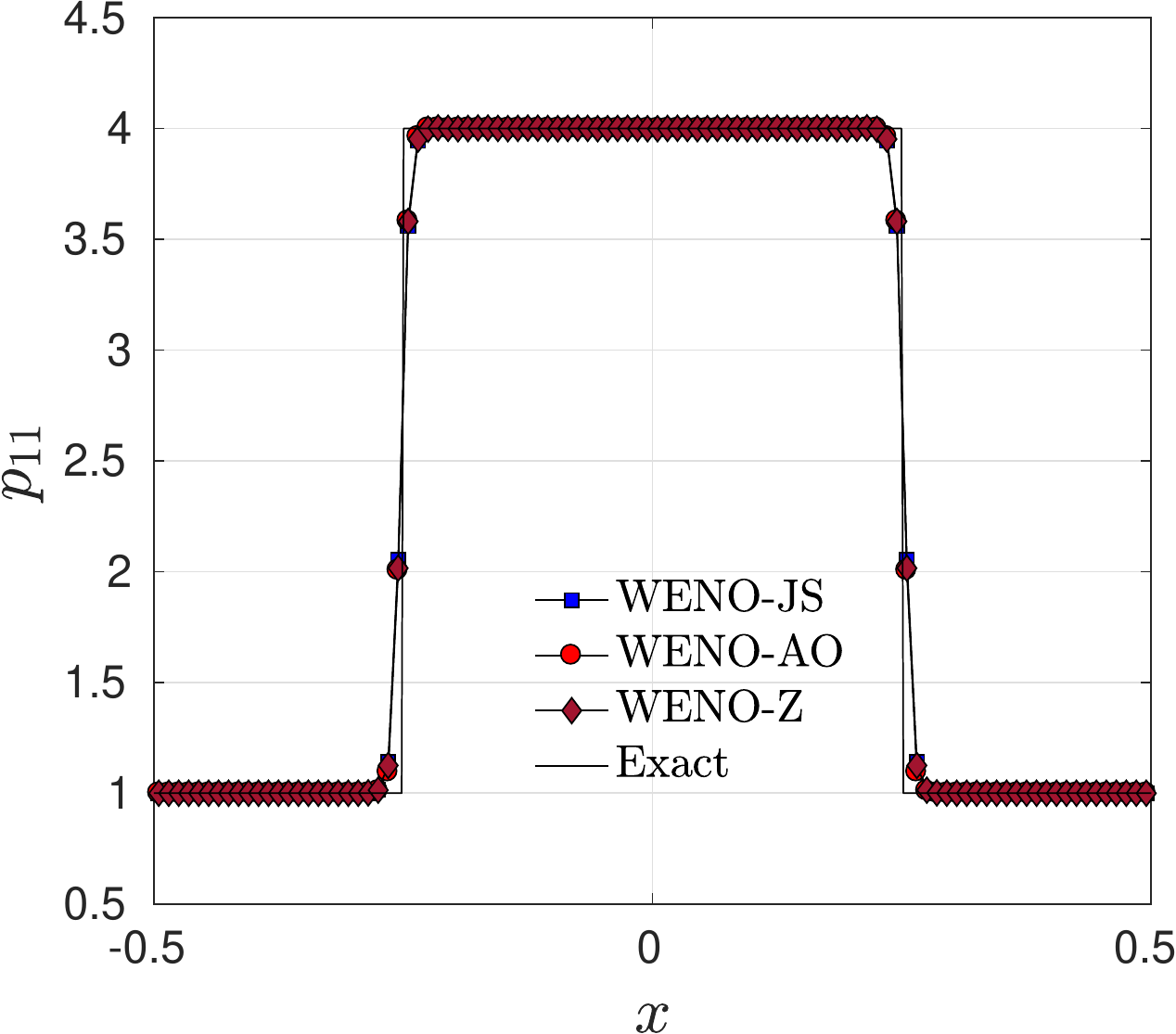} &
 \includegraphics[width=0.32\textwidth]{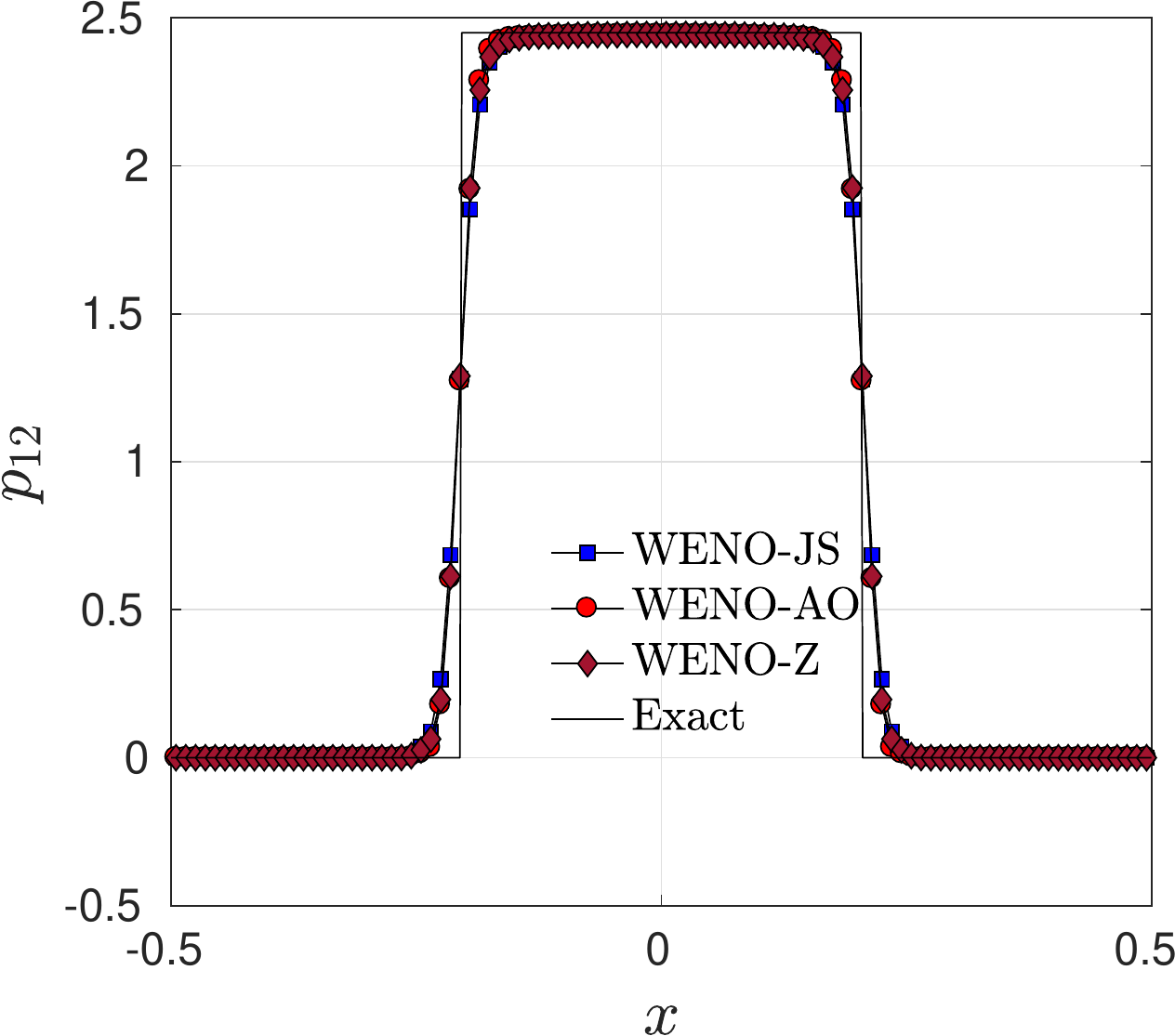} &
  \includegraphics[width=0.32\textwidth]{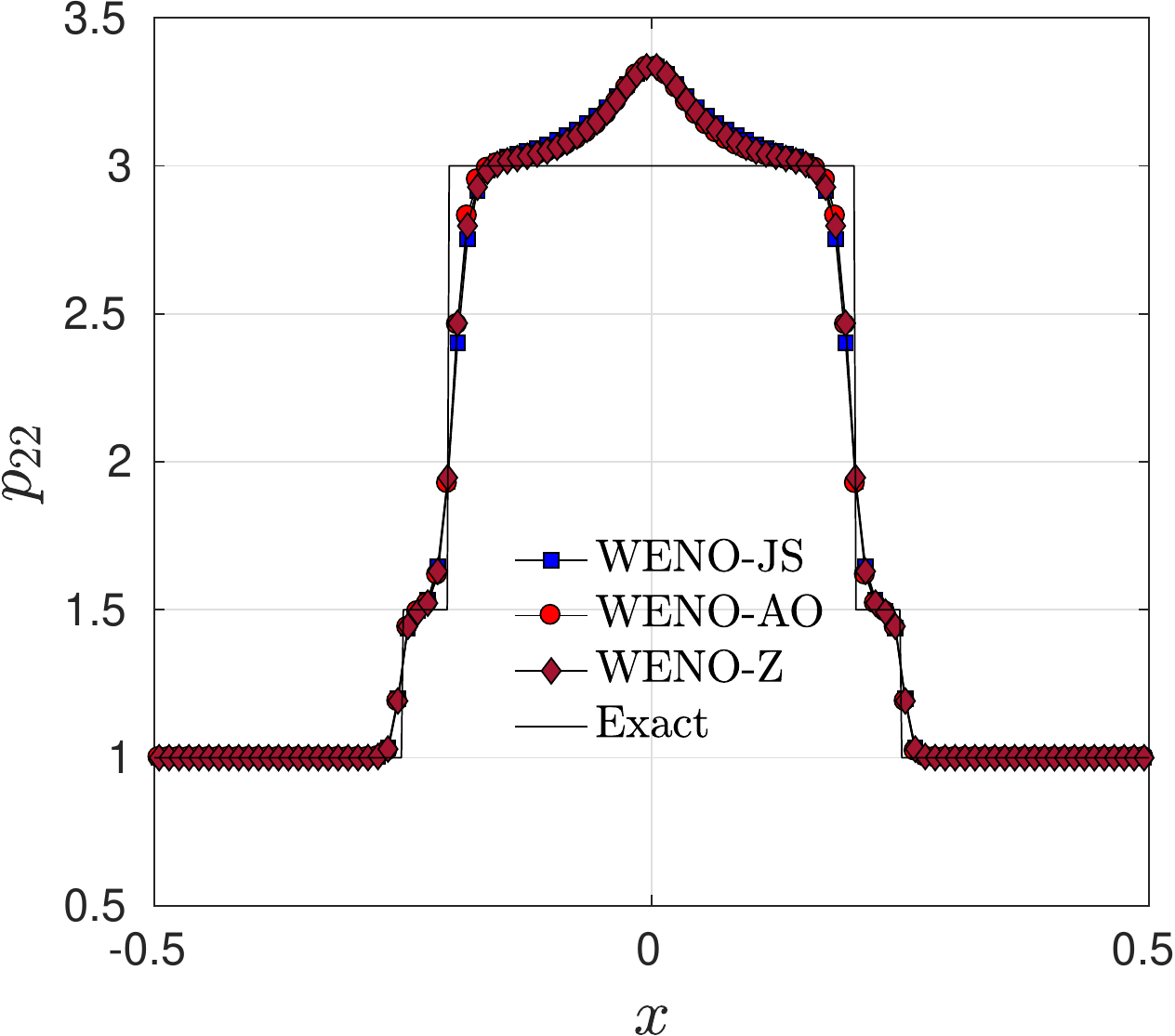} \\
 (a) $p_{11}$ & (b) $p_{12}$ &   (c) $p_{22}$  \\
 \includegraphics[width=0.32\textwidth]{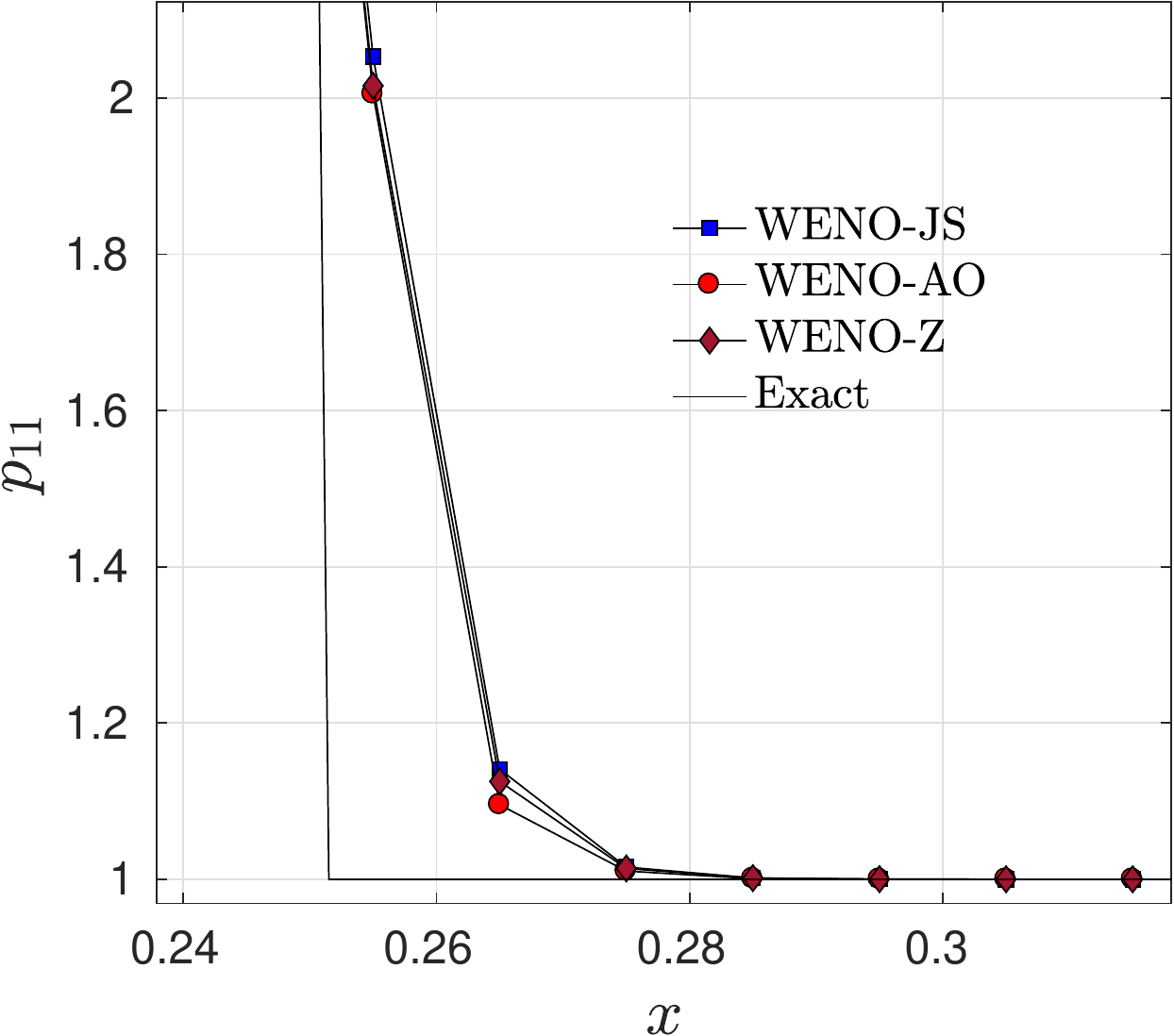} &
 \includegraphics[width=0.32\textwidth]{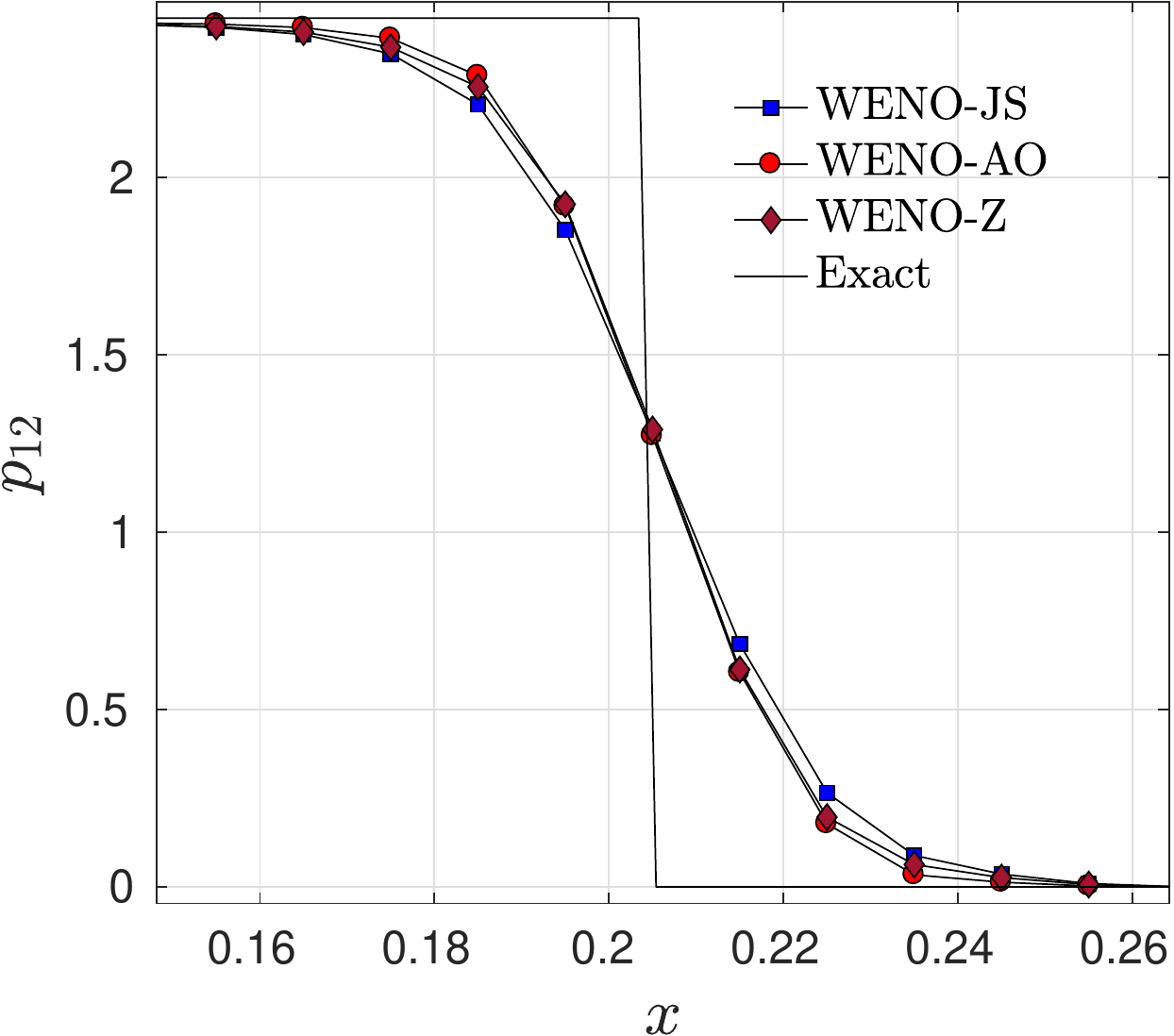} &
 \includegraphics[width=0.32\textwidth]{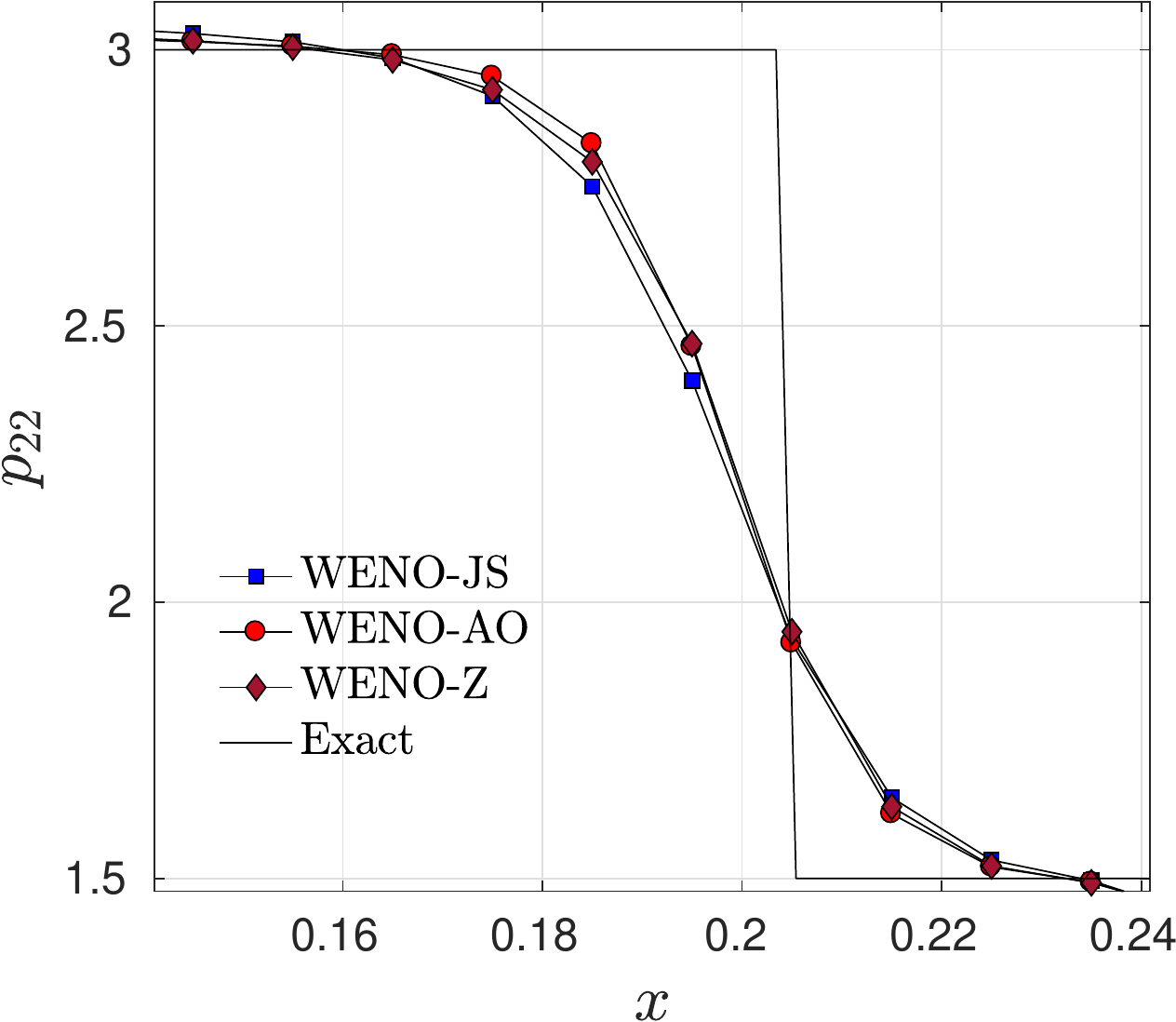}\\
(d) Zoom of $p_{11}$ &  (e) Zoom of $p_{12}$ & (f) Zoom of  $p_{22}$
\end{tabular}
\end{center}
 \caption{Comparison of pressure component obtained using WENO-JS, WENO-Z and WENO-AO schemes with the exact solution for Example \ref{ex:twoshock}.}
 \label{fig:twoshock2}
\end{figure}

\begin{example}\label{ex:twoshock}{\rm (Two-shock wave)
This problem is a standard test case for the Ten-Moment equations \cite{ber_06a, mee-kum_17a, sen-kum_18a}. It consists of a Riemann problem with an initial discontinuity in the velocities at $x=0$ over the domain $[-0.5,0.5]$. The initial data is given as
 \[
  \bold{V}=\begin{cases}
              (1, \ 1, \ 1, \ 1, \ 0, \ 1), &\mbox{if}~~x\leq 0\\
              (1, \ -1, \ -1, \ 1, \ 0, \ 1),&\mbox{if}~~x>0
             \end{cases}
 \]
 }
\end{example}

\begin{figure}
\begin{center}
\begin{tabular}{ccc}
\includegraphics[width=0.32\textwidth]{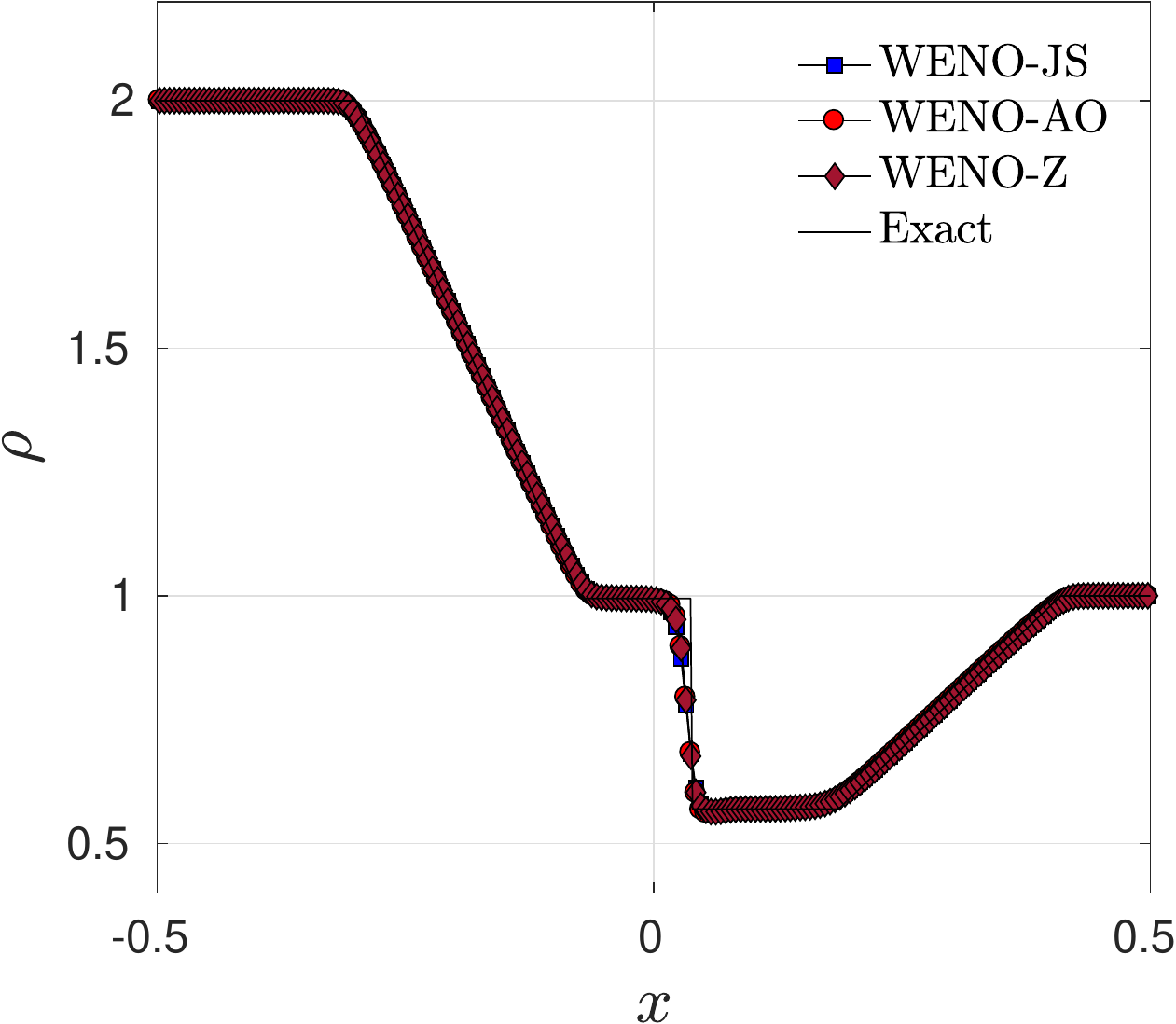} &
 \includegraphics[width=0.32\textwidth]{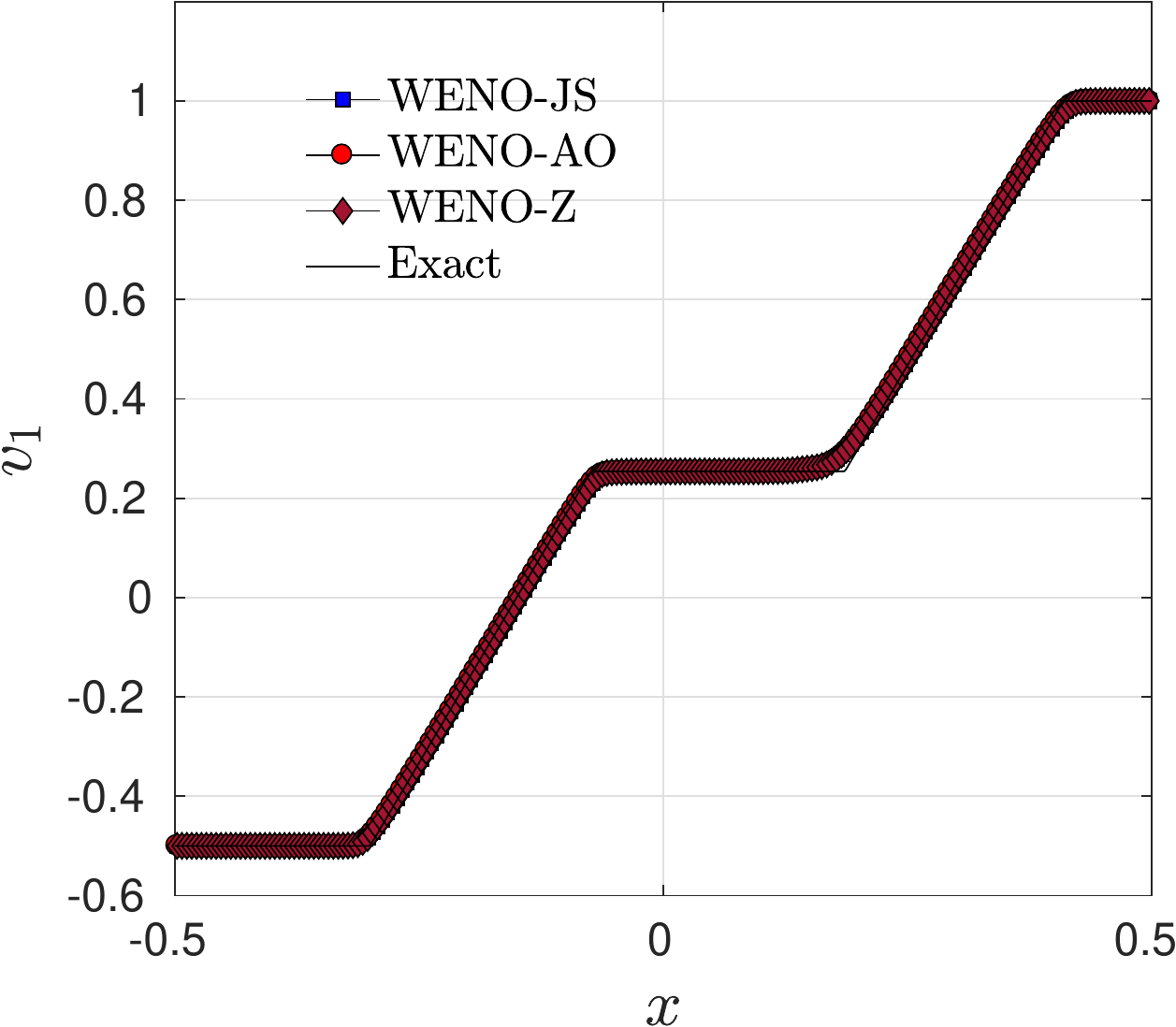} &
  \includegraphics[width=0.32\textwidth]{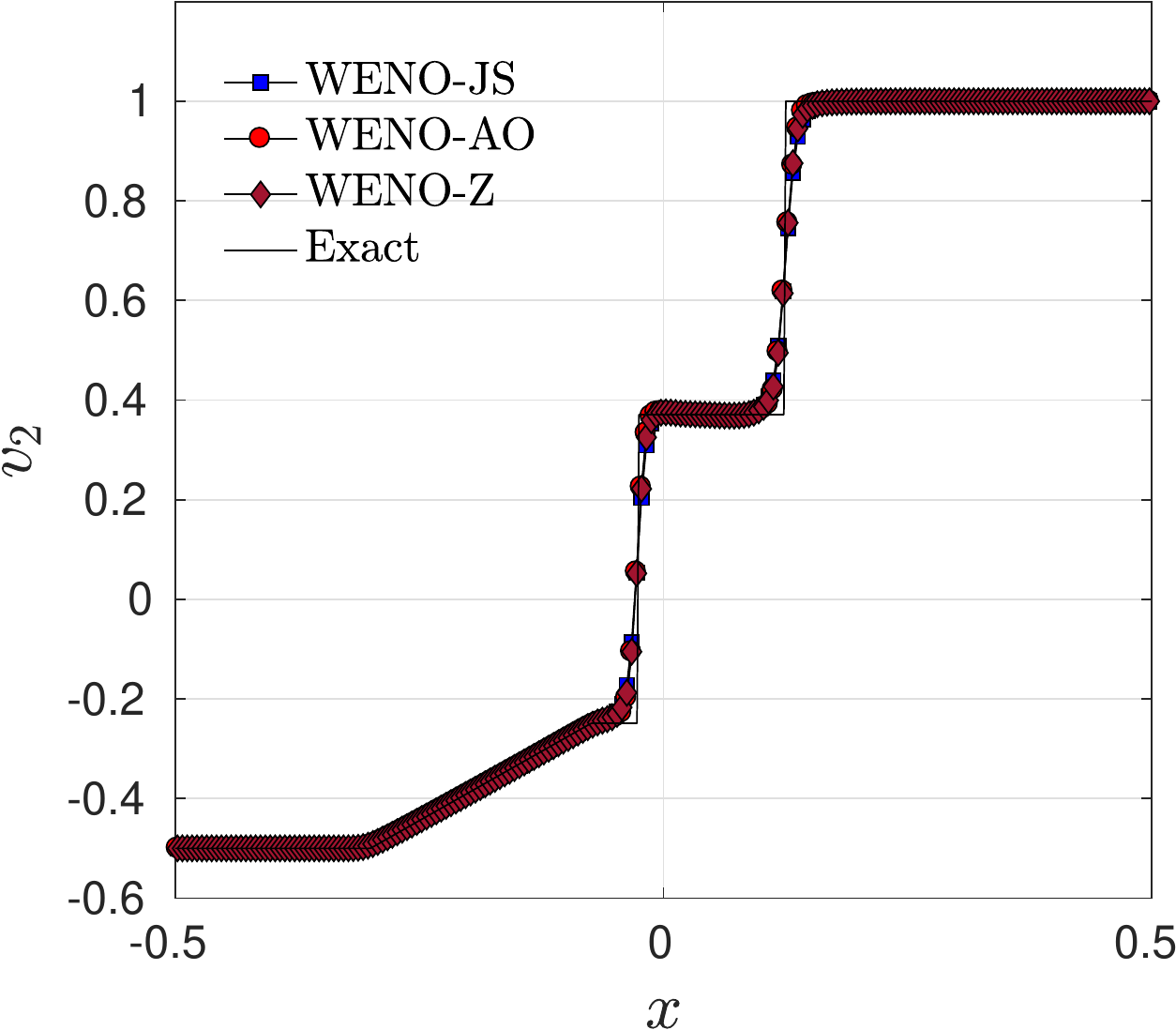} \\
 (a) $\rho$ & (b) $v_1$ &   (c) $v_2$  \\
 \includegraphics[width=0.32\textwidth]{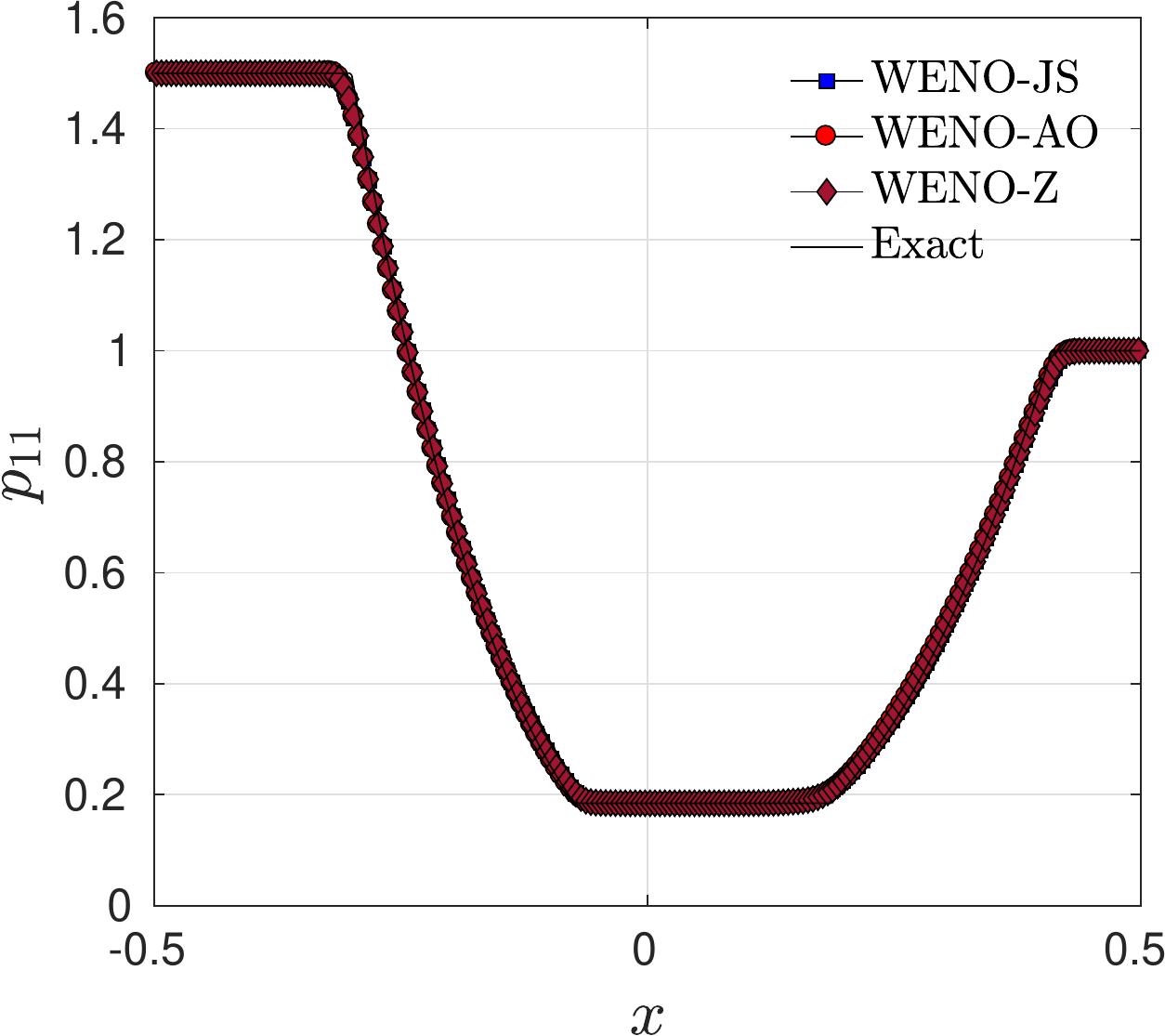} &
 \includegraphics[width=0.32\textwidth]{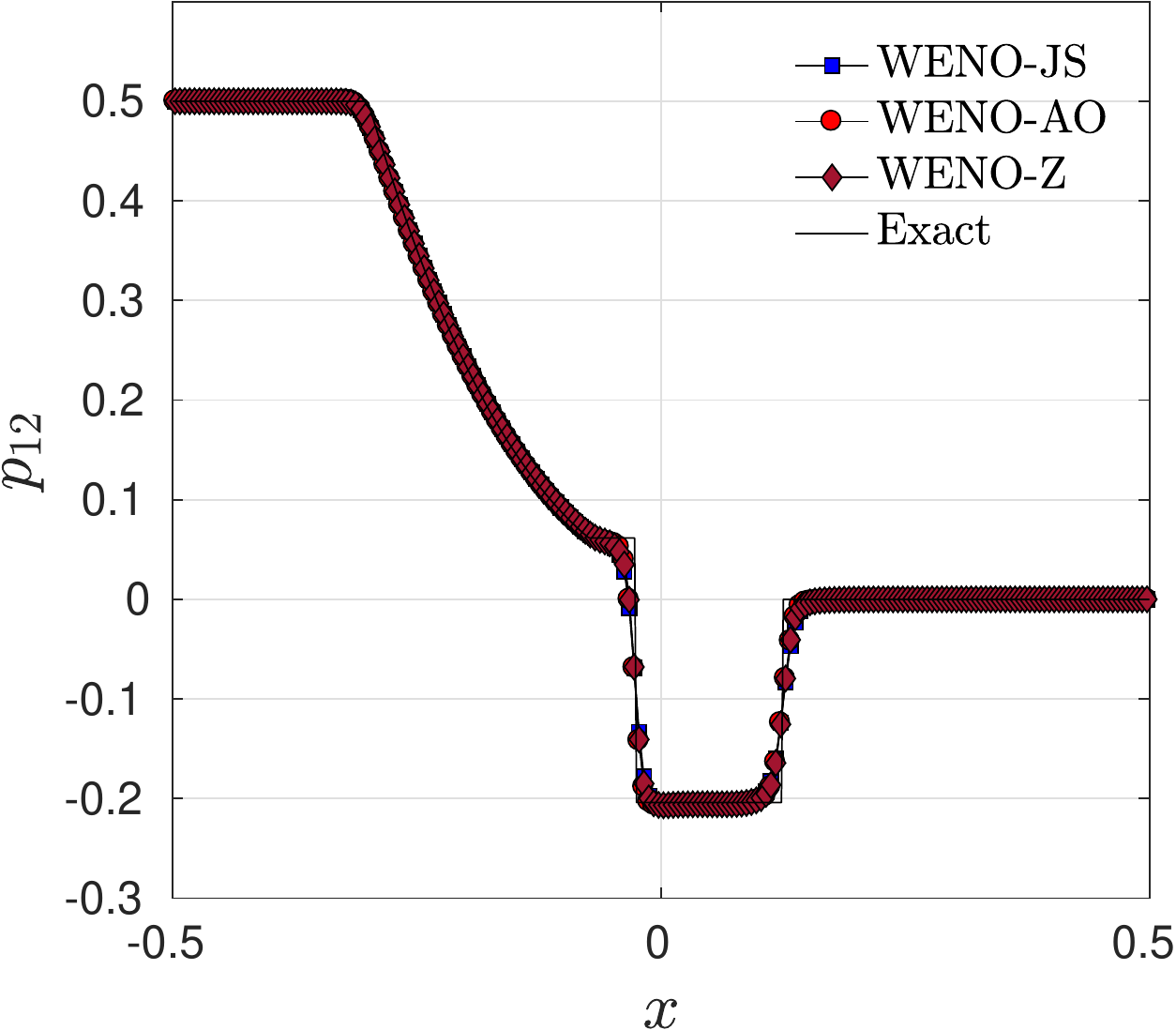} &
 \includegraphics[width=0.32\textwidth]{twor_p11}\\
(d) $p_{11}$ &  (e) $p_{12}$ & (f) $p_{22}$
\end{tabular}
\end{center}
 \caption{Comparison of primitive variables obtained using WENO-JS, WENO-Z and WENO-AO schemes with the exact solution for Example \ref{ex:tworar}.}
 \label{fig:tworarefaction1}
\end{figure}

\begin{figure}
\begin{center}
\begin{tabular}{ccc}
\includegraphics[width=0.32\textwidth]{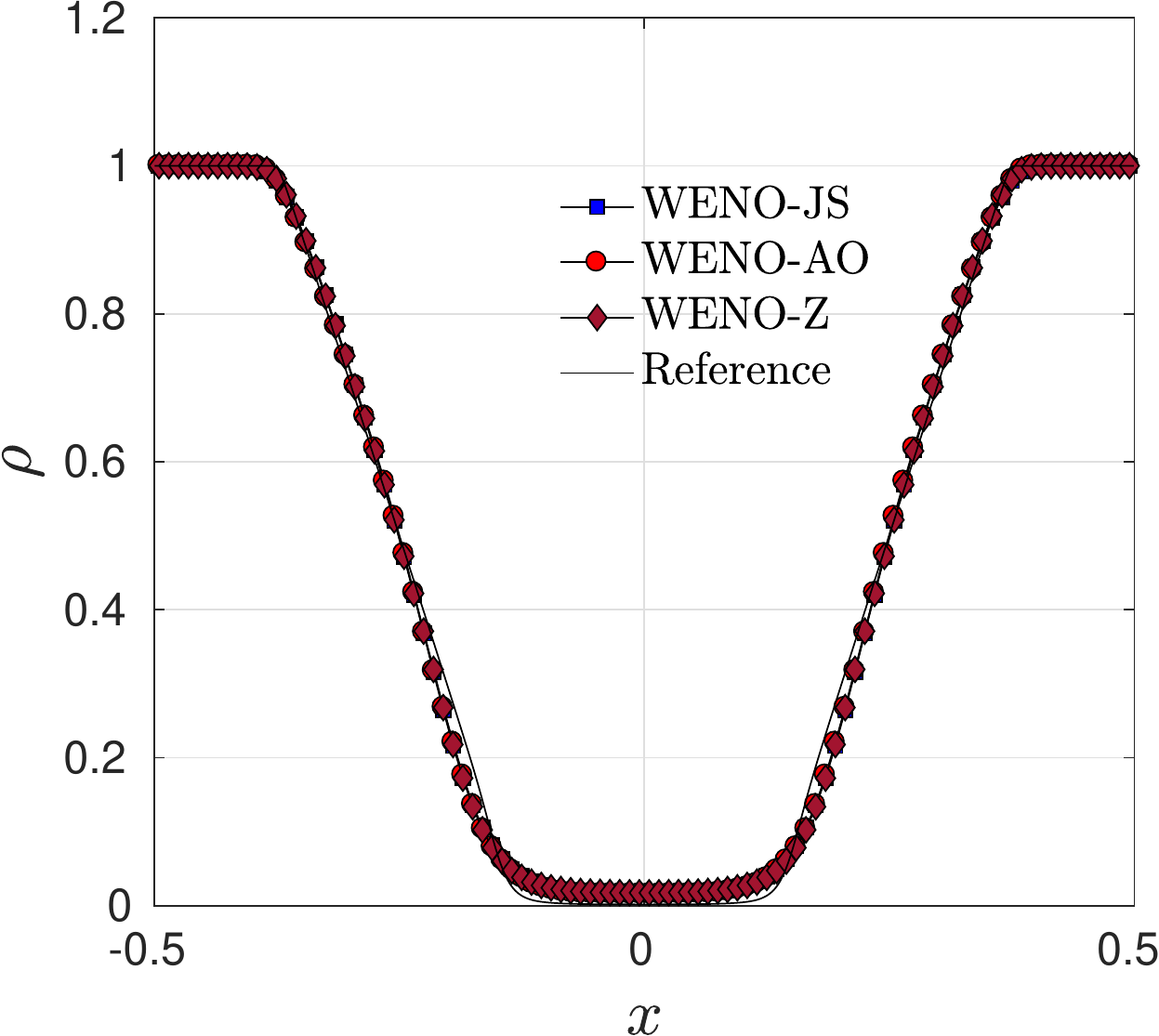} &
 \includegraphics[width=0.32\textwidth]{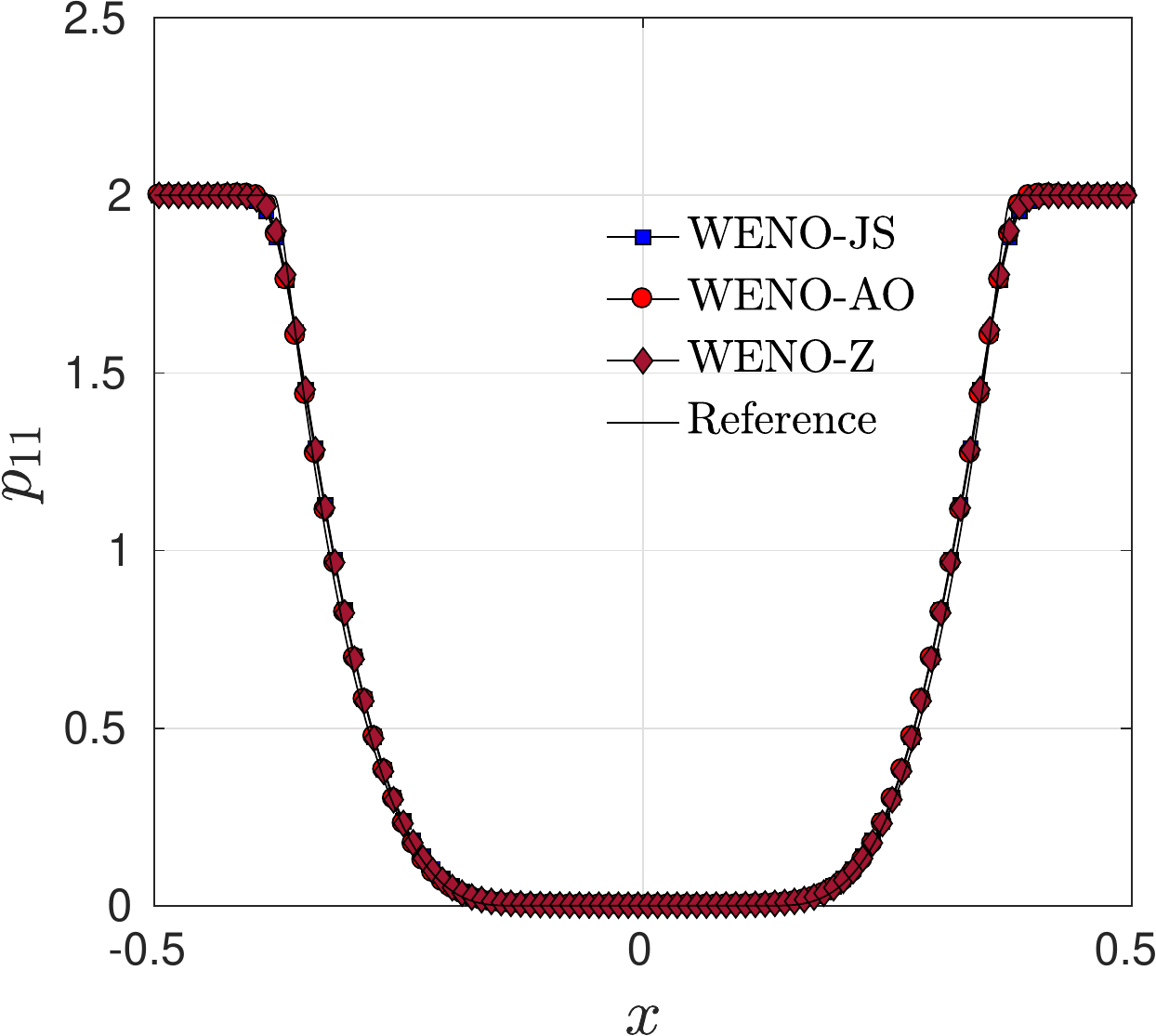} &
  \includegraphics[width=0.32\textwidth]{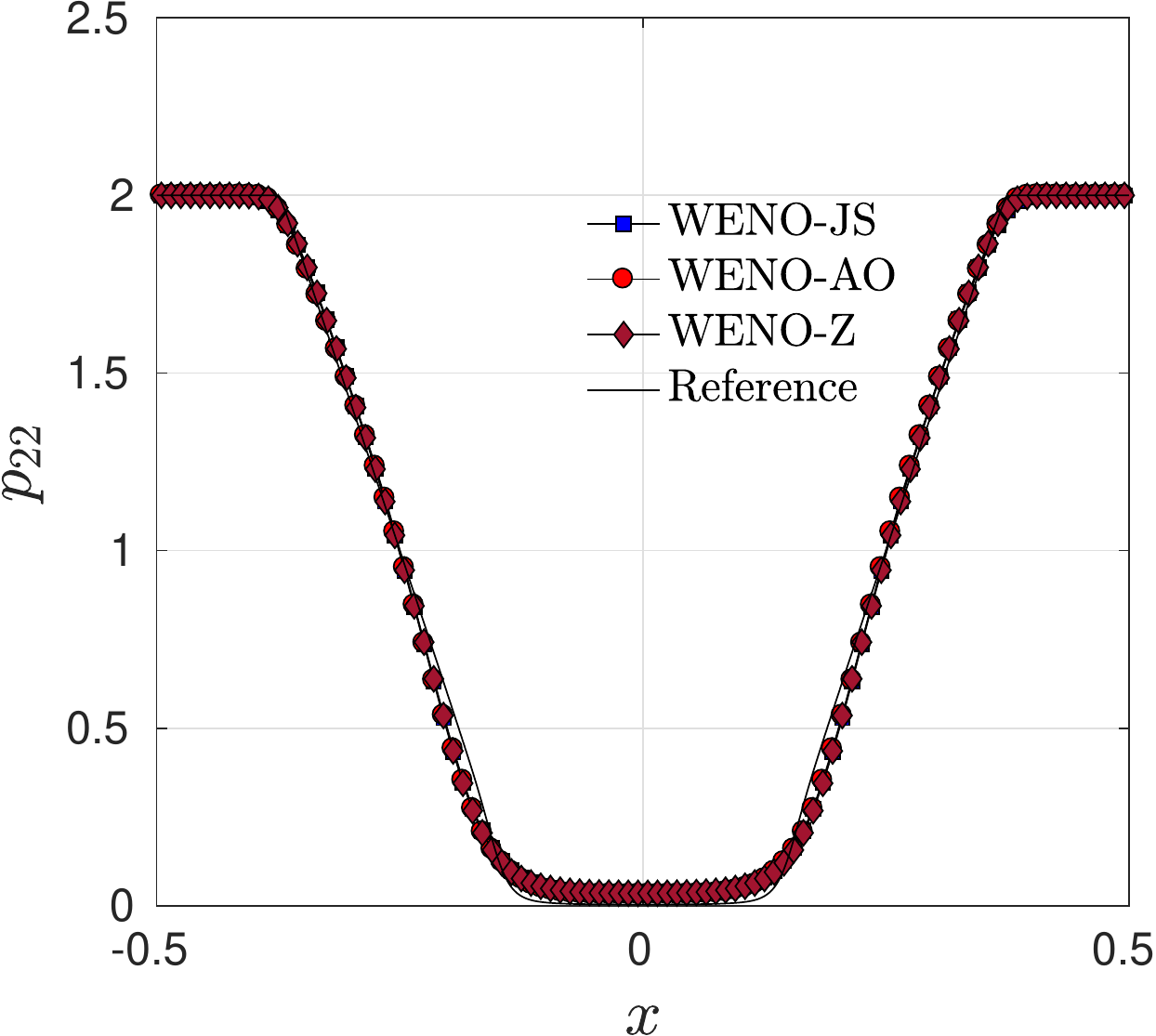} \\
 (a) $\rho$ & (b) $p_{11}$ &   (c) $p_{22}$  \\
 \includegraphics[width=0.32\textwidth]{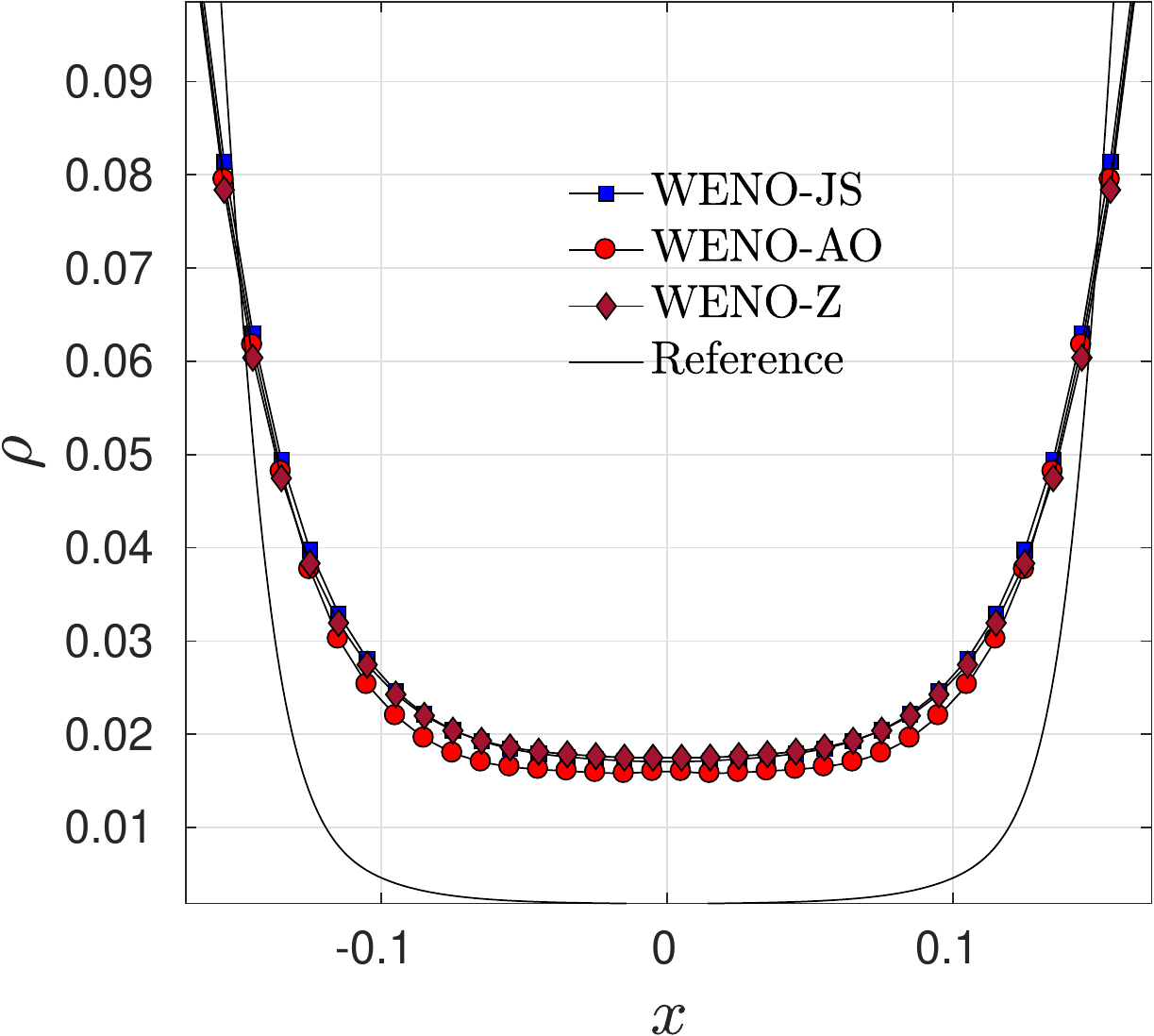} &
 \includegraphics[width=0.32\textwidth]{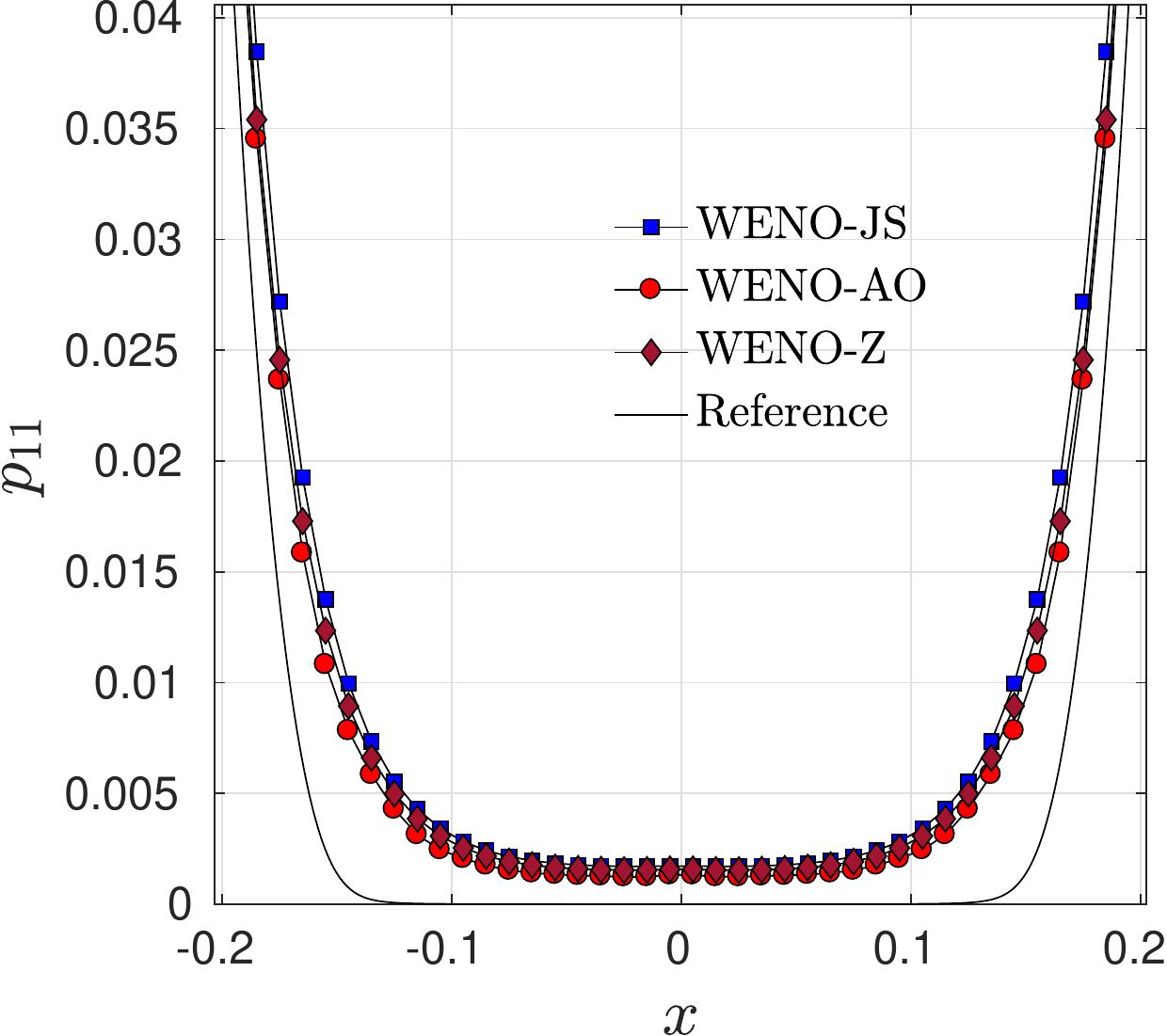} &
 \includegraphics[width=0.32\textwidth]{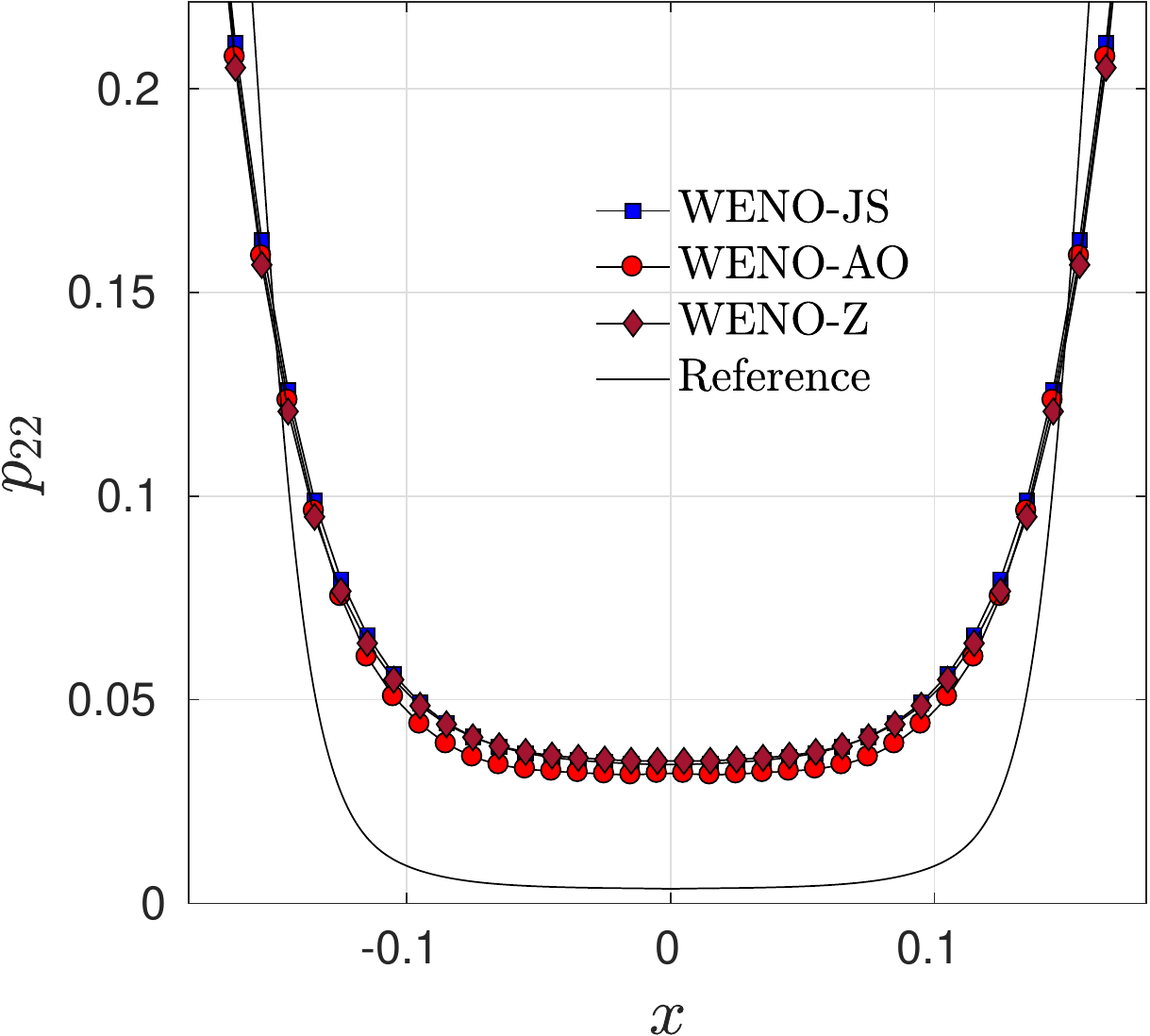}\\
(d) Zoom of $\rho$ &  (e) Zoom of $p_{11}$ & (f) Zoom of $p_{22}$
\end{tabular}
\end{center}
 \caption{Comparison of  primitive variables obtained using WENO-JS, WENO-Z and WENO-AO schemes with the exact solution for Example \ref{ex7}.}
 \label{fig:ex7}
\end{figure}

\noindent
The exact solution contains the two shock waves separated with contact discontinuity.
The numerical results are computed at the final time $T=0.125$ using 100 cells. In Figure \ref{fig:twoshock1} and \ref{fig:twoshock2}, we have compared the solutions obtained using WENO schemes 
with exact solutions. We observed that all the WENO schemes capture the shock efficiently without oscillations and the positivity limiter is not used during the simulation. We observed the spikes in  Figure \ref{fig:twoshock2}, which is the result of two opposing hypersonic flows at contact discontinuities. 

\begin{example}\label{ex:tworar}{\rm (Two-rarefaction wave)
This Riemann test contains two rarefaction waves separated with a contact discontinuity. Here,
 we consider the same domain and point of discontinuity as in the previous test. The initial data is given by
  \[
  \bold{V}=\begin{cases}
              (2, \ -0.5,  \ -0.5, \ 1.5,  \ 0.5, \ 1.5),& \mbox{if}~~x \leq 0\\
              (1, \ 1, \ 1, \ 1, \ 0, \ 1), &\mbox{if}~x>0
             \end{cases}
\]

 }
\end{example}
The numerical results are computed using WENO-JS, WENO-Z, and WENO-AO  schemes with 200 grid points up to the final time $T=0.15$ and comparison of the results is depicted in the Figure \ref{fig:tworarefaction1}. We can observe from the figure that all the proposed schemes resolve the solution accurately. Since this test case is not intended to test positivity limiter, we found that positivity limiter is not active in all the simulations. 

\begin{figure}
\begin{center}
\begin{tabular}{cc}
\includegraphics[width=0.48\textwidth]{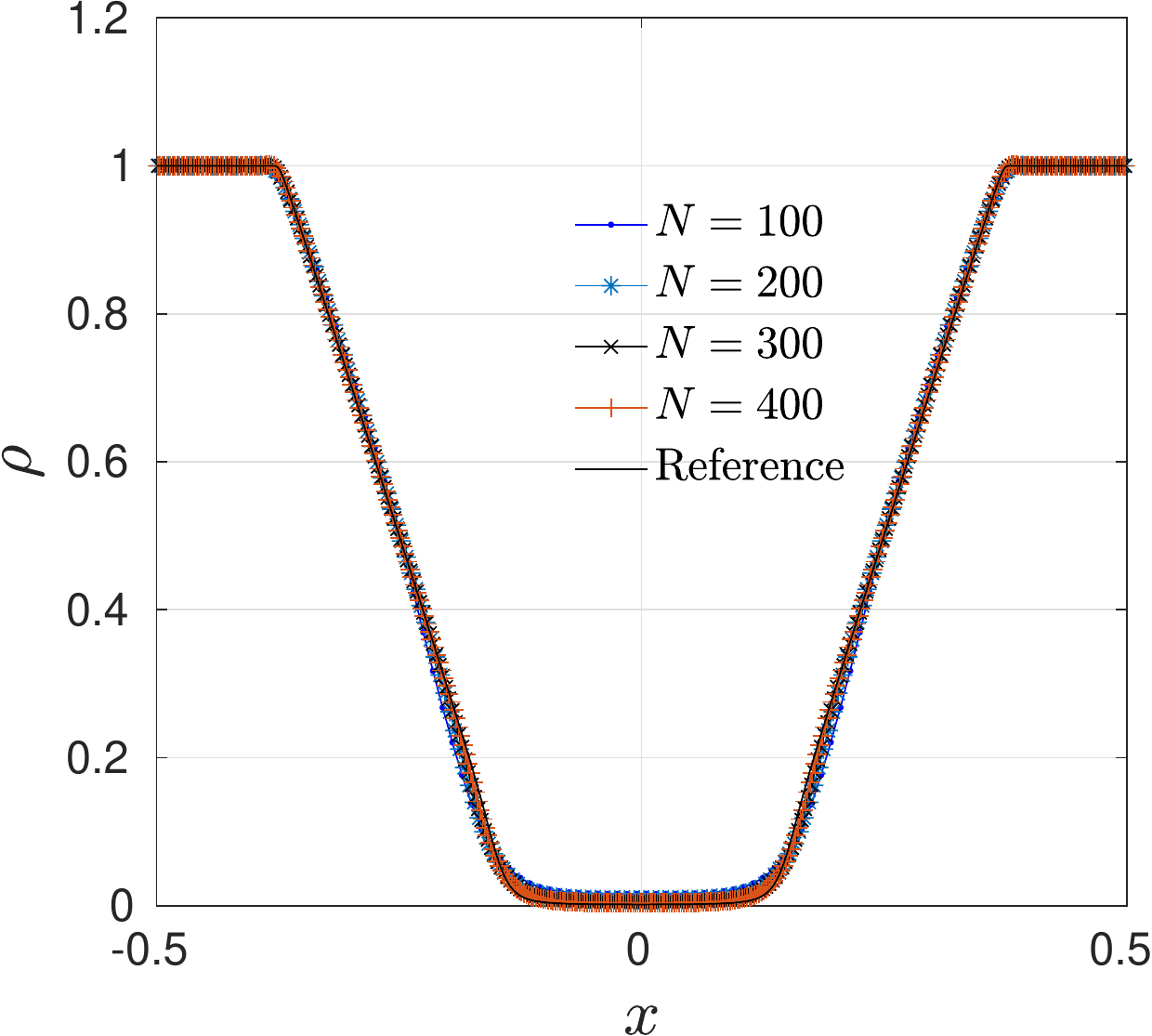} &
 \includegraphics[width=0.48\textwidth]{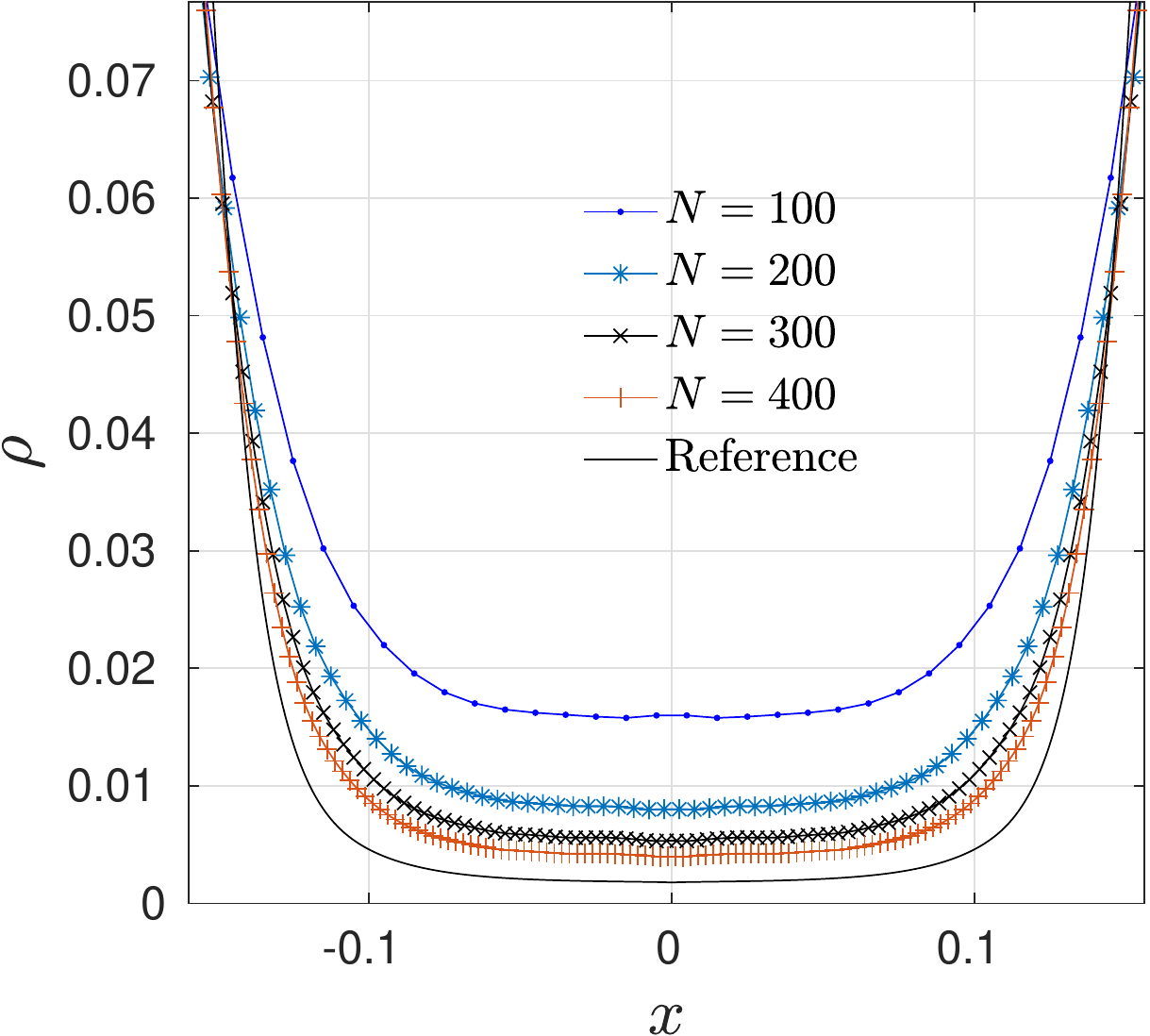} \\
 (a) Density & (b) Zoomed view of density
\end{tabular}
\end{center}
 \caption{ Numerical solution computed using WENO-AO scheme with varying the number of mesh points.}
 \label{fig:ex7ao}
\end{figure}

\begin{figure}
\begin{center}
\begin{tabular}{cc}
\includegraphics[width=0.48\textwidth]{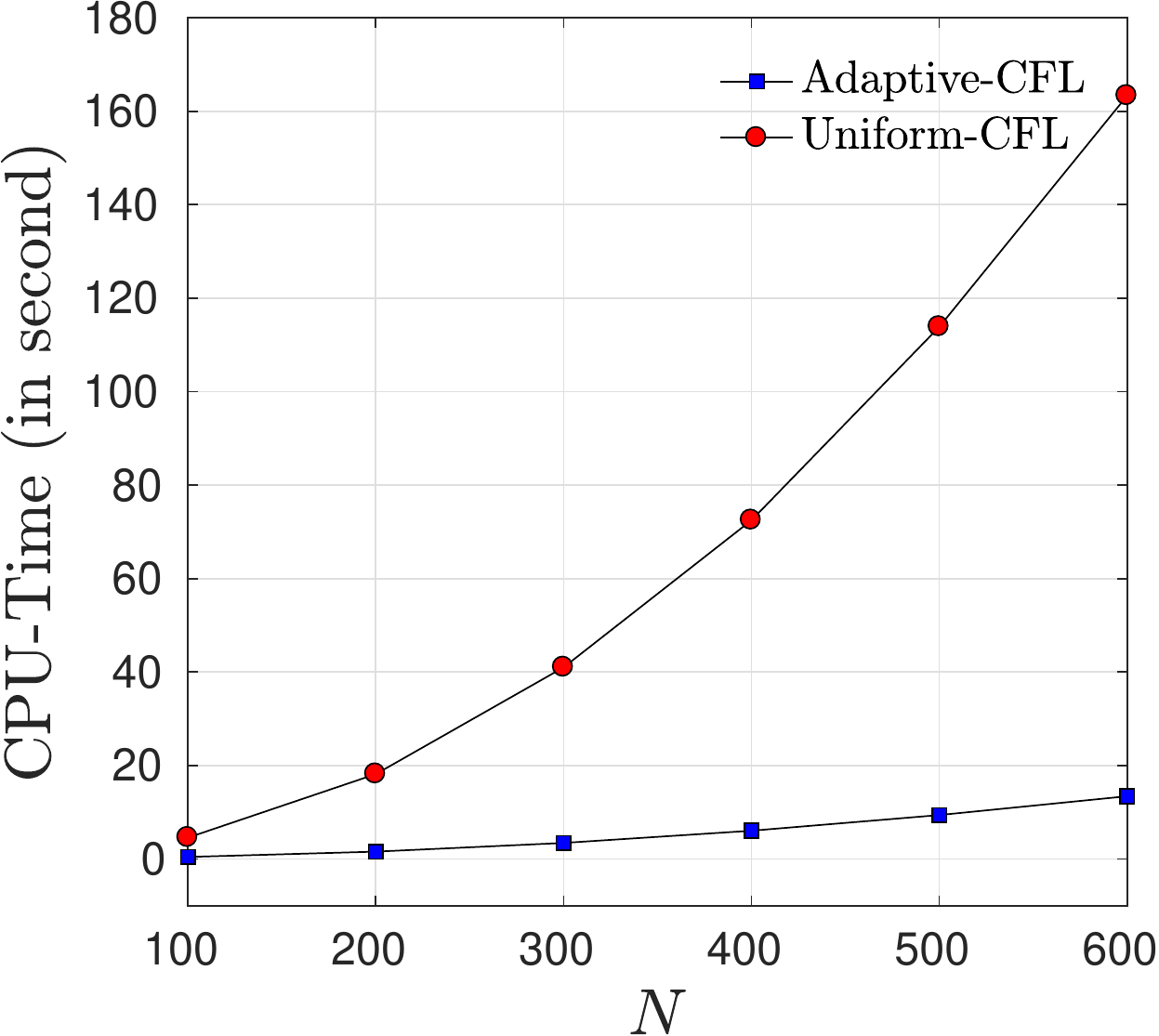} &
 \includegraphics[width=0.48\textwidth]{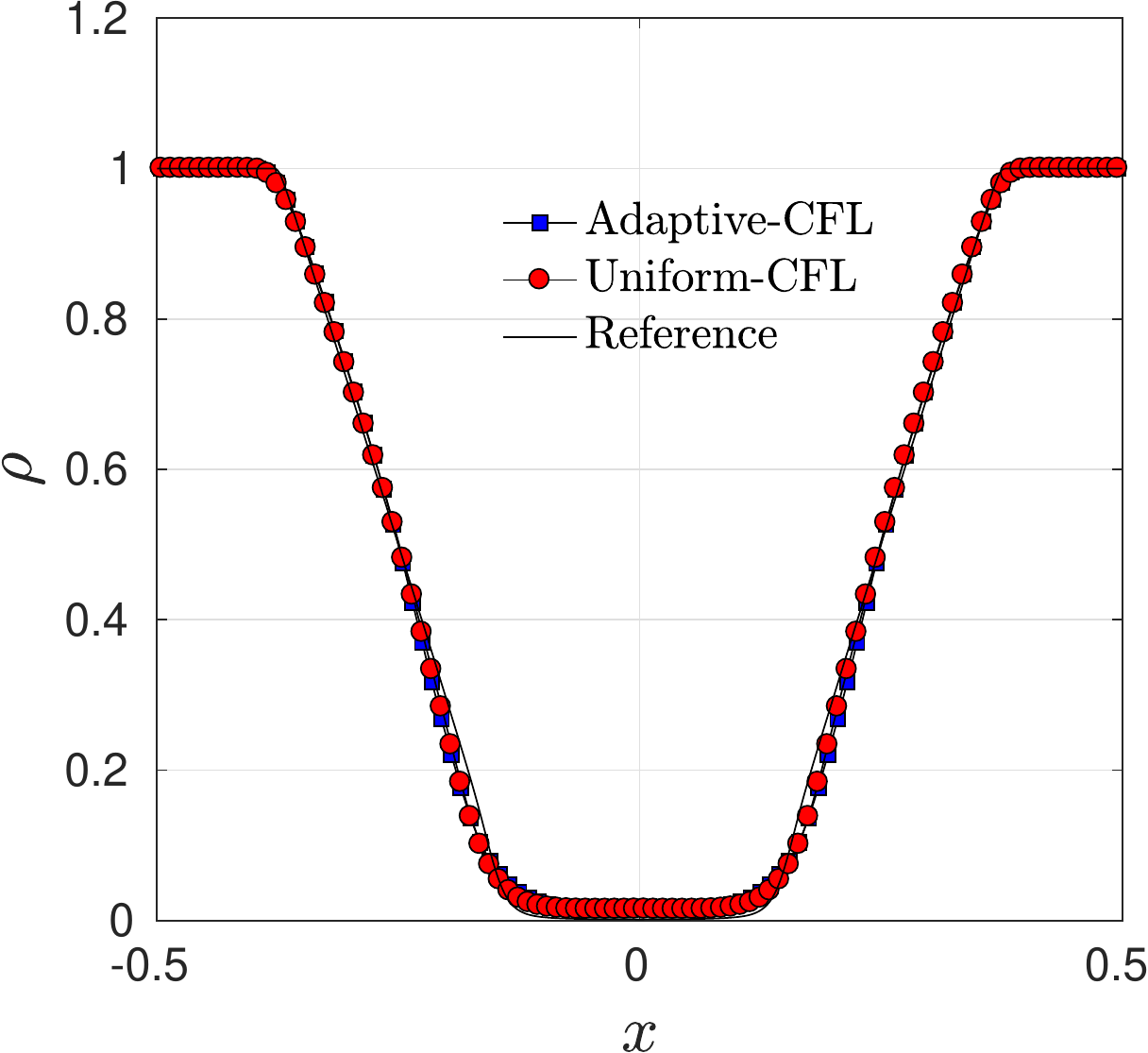} \\
 (a)  & (b) 
\end{tabular}
\end{center}
 \caption{(a) Comparison of CPU time in seconds for WENO-AO scheme with adaptive and uniform CFL for Example \ref{ex7} with varying  number of mesh points, (b) Comparison of solution obtained using adaptive and uniform CFL using $100$ mesh points.}
 \label{fig:ex7cpu}
\end{figure}

\begin{table}
\centering 
\begin{tabular}{|c| c|c| c|}
\hline 
 Grid & Adaptive-CFL  & Uniform-CFL & Speed-up factor\\
\hline
100  &  0.4        & 4.6         &11.5 \\ 
\hline 
200  &  1.6        & 18.2        &11.4 \\ 
\hline 
300  &  3.4        & 41.0        &12.0 \\ 
\hline 
400 &  6.0        & 72.5        &12.0 \\ 
\hline 
500  & 9.4        & 113.9       &12.1  \\ 
\hline 
600 &  13.4       & 163.3       &12.2 \\ 
\hline 
\end{tabular}
\caption{Comparison of CPU times taken in seconds  for Example \ref{ex7}, in case of adaptive and uniform CFL along with speed-up factor.}
\label{Table.ex7}
\end{table}

\begin{table}
\centering 

\begin{tabular}{|c| c|c| c|}
\hline 
Grid & Adaptive-CFL  & Uniform-CFL & Speed-up factor\\
\hline
$50\times  50$      &  3       & 4      &1.3\\ 
\hline 
$100\times  100$    &  20     & 34     &1.7 \\ 
\hline 
$150\times  150$    &  79      & 115    &1.4 \\ 
\hline 
$200\times  200$    &  189    & 272   & 1.4  \\ 
\hline 
$250\times  250$    & 308     & 535   &1.7 \\ 
\hline 
$300\times  300$    &  458     &  919   & 2.0 \\ 
\hline 
\end{tabular}
\caption{ Comparison of CPU time in seconds taken for Example \ref{ex8}, in case of adaptive and uniform CFL along with speed-up factor.}
\label{Table.ex8}
\end{table}

\begin{table}
\centering 
\begin{tabular}{|c| c|c| c|}
\hline 
 Grid               & Adaptive-CFL  & Uniform-CFL & Speed-up factor\\
\hline
$50\times  50$      &  5        & 9       &1.8  \\ 
\hline 
$100\times  100$    &  44      &  74     &1.7  \\ 
\hline 
$150\times  150$    &  108     & 253    &2.3 \\ 
\hline 
$200\times  200$    &  180      & 609     & 3.4\\ 
\hline 
$250\times  250$    & 325      & 1193    &3.7 \\ 
\hline 
$300\times  300$    &  531      & 2102    &3.9 \\ 
\hline 
\end{tabular}
\caption{Comparison of CPU time in seconds taken for Example \ref{ex9}, in case of adaptive and uniform CFL along with speed-up factor. }
\label{Table.ex9}
\end{table}

\begin{figure}
	\begin{center}
		\begin{tabular}{cc}
			\includegraphics[width=0.48\textwidth]{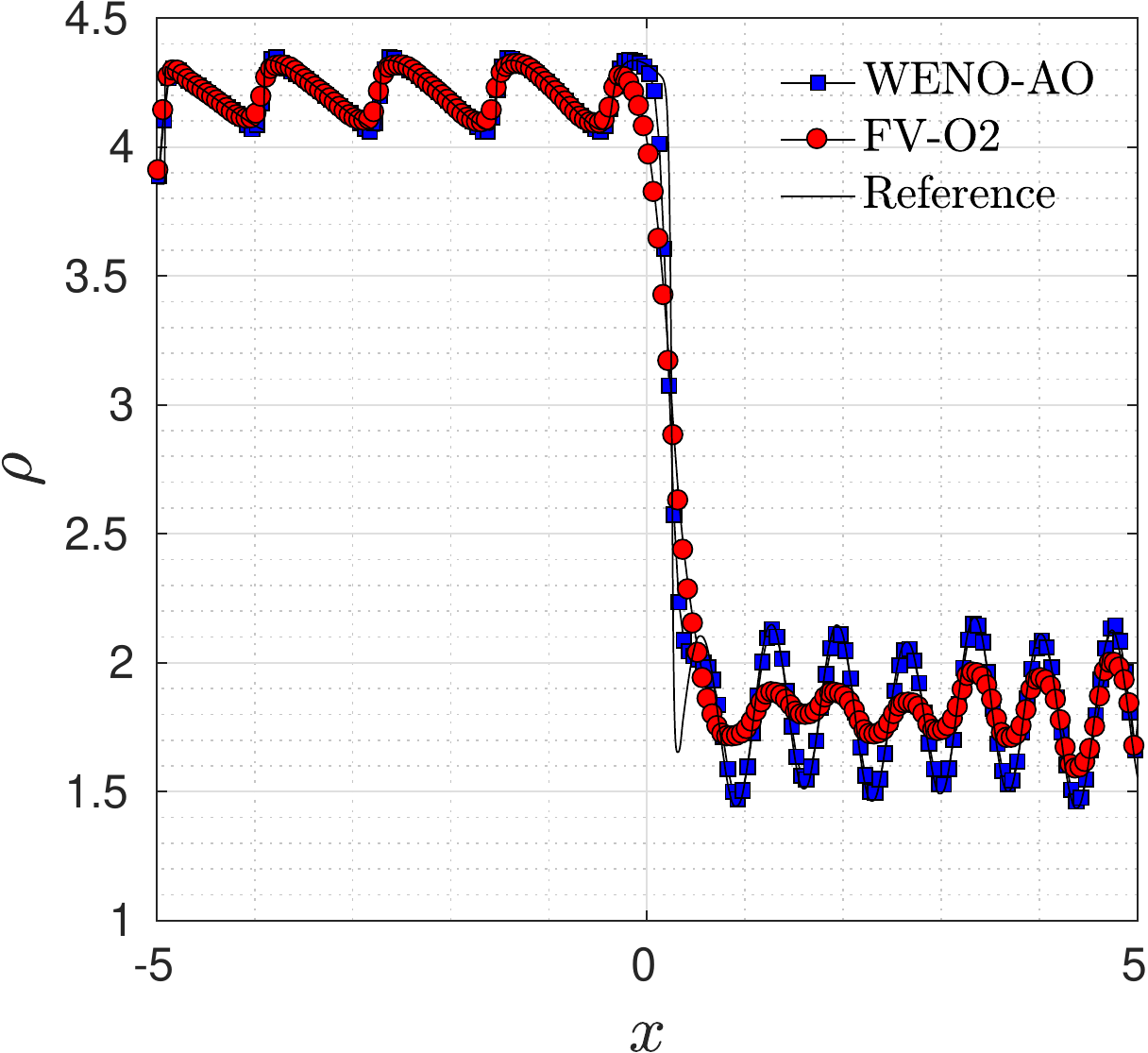} &
			\includegraphics[width=0.48\textwidth]{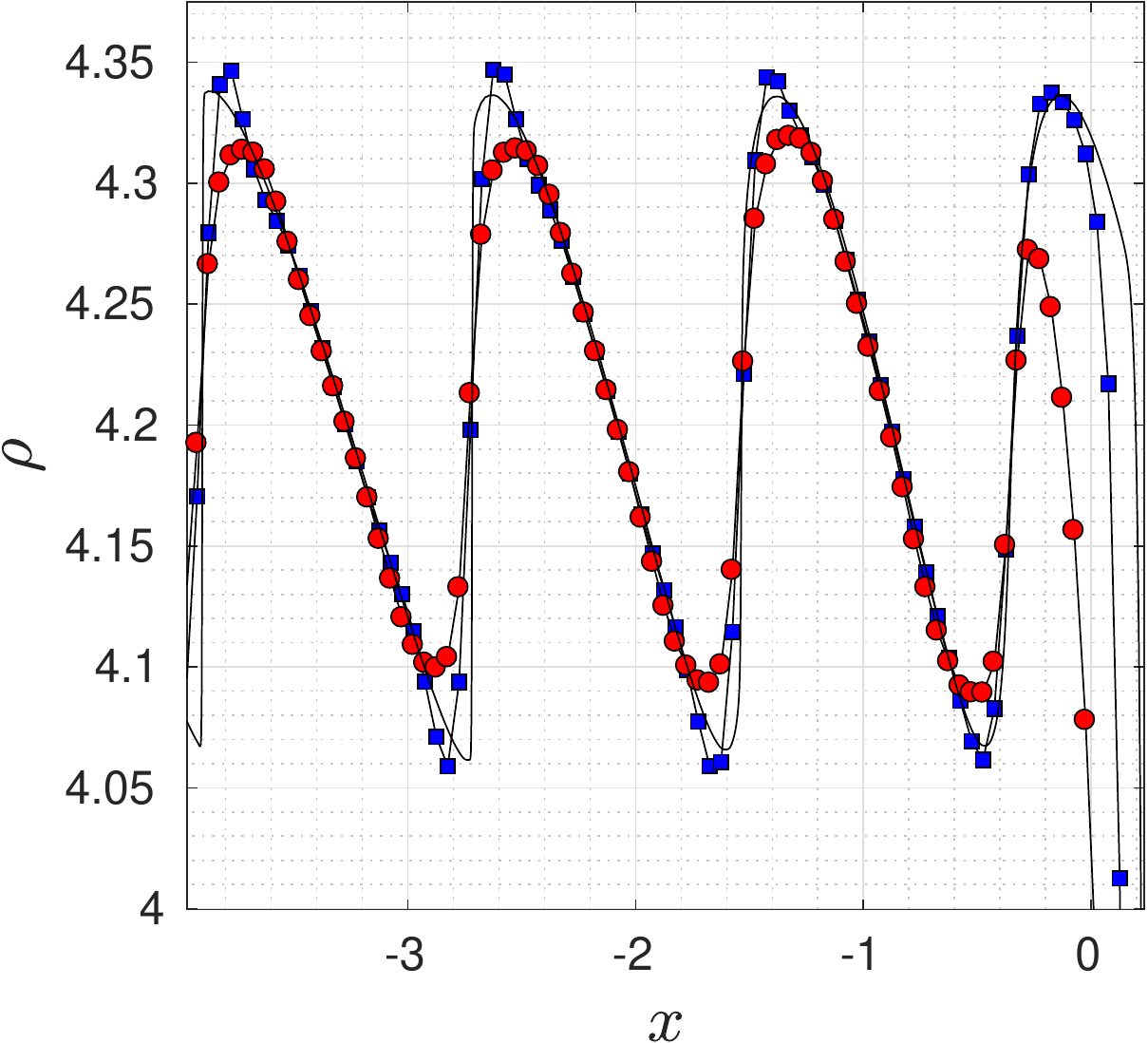} \\
			\includegraphics[width=0.48\textwidth]{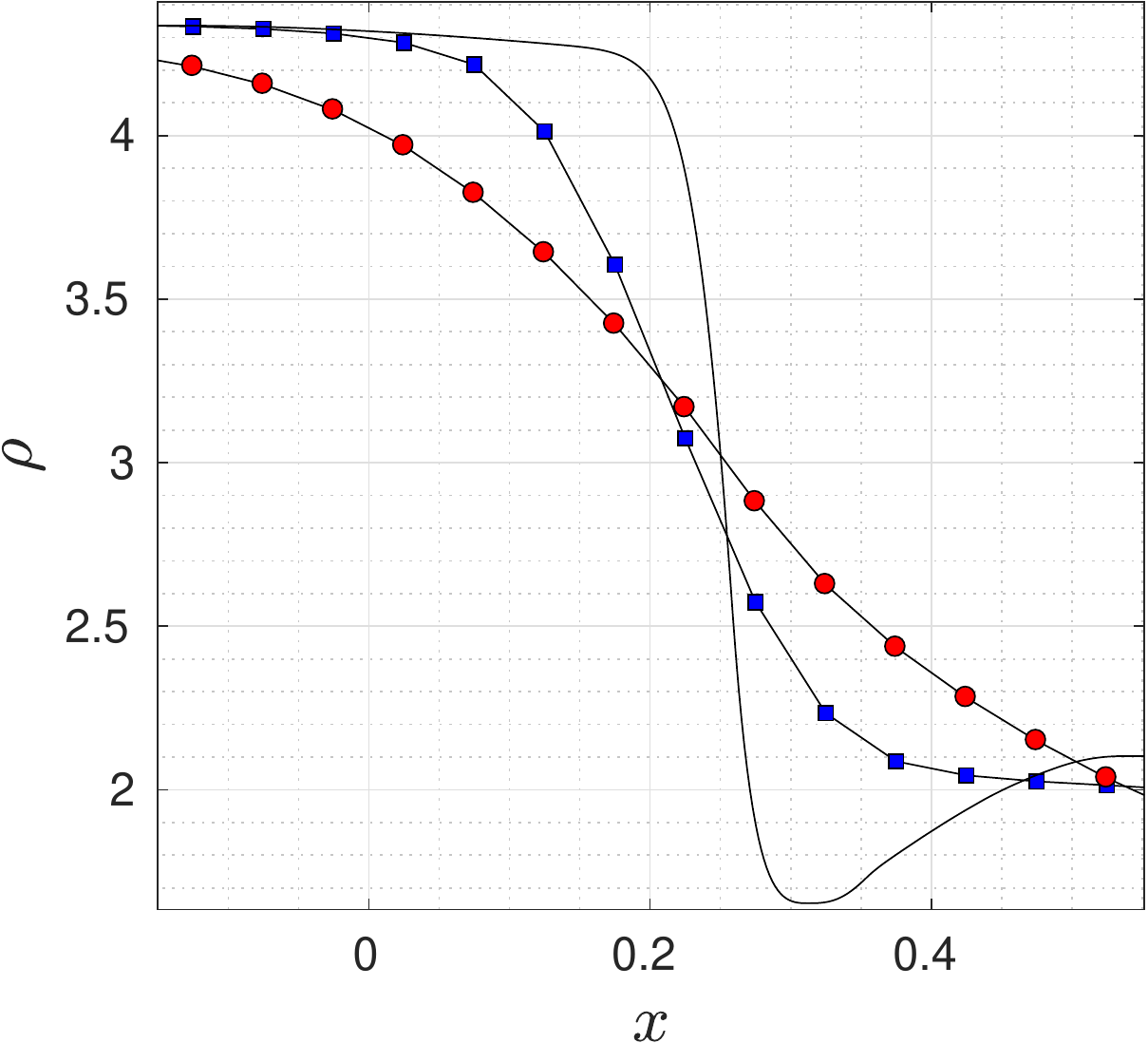} &
			\includegraphics[width=0.48\textwidth]{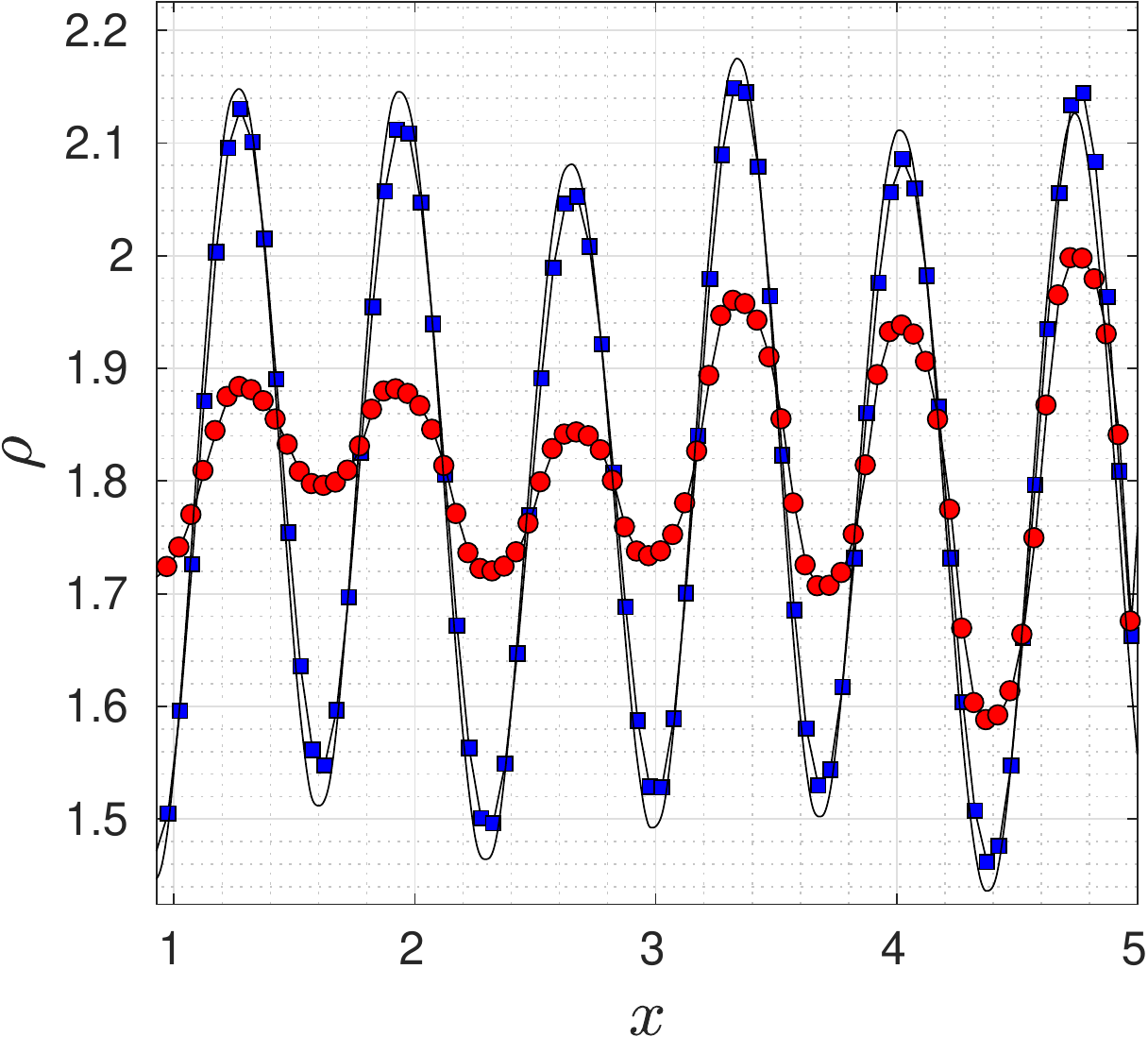}
		\end{tabular}
	\end{center}
	\caption{Comparison of density solutions for WENO-AO and second-order finite volume scheme (\cite{meenaKumarFVM}) for Example \ref{exShuOsher}.}
	\label{fig:ShuOsher}
\end{figure}

\begin{figure}
	\begin{center}
		\begin{tabular}{ccc}
			\includegraphics[width=0.32\textwidth]{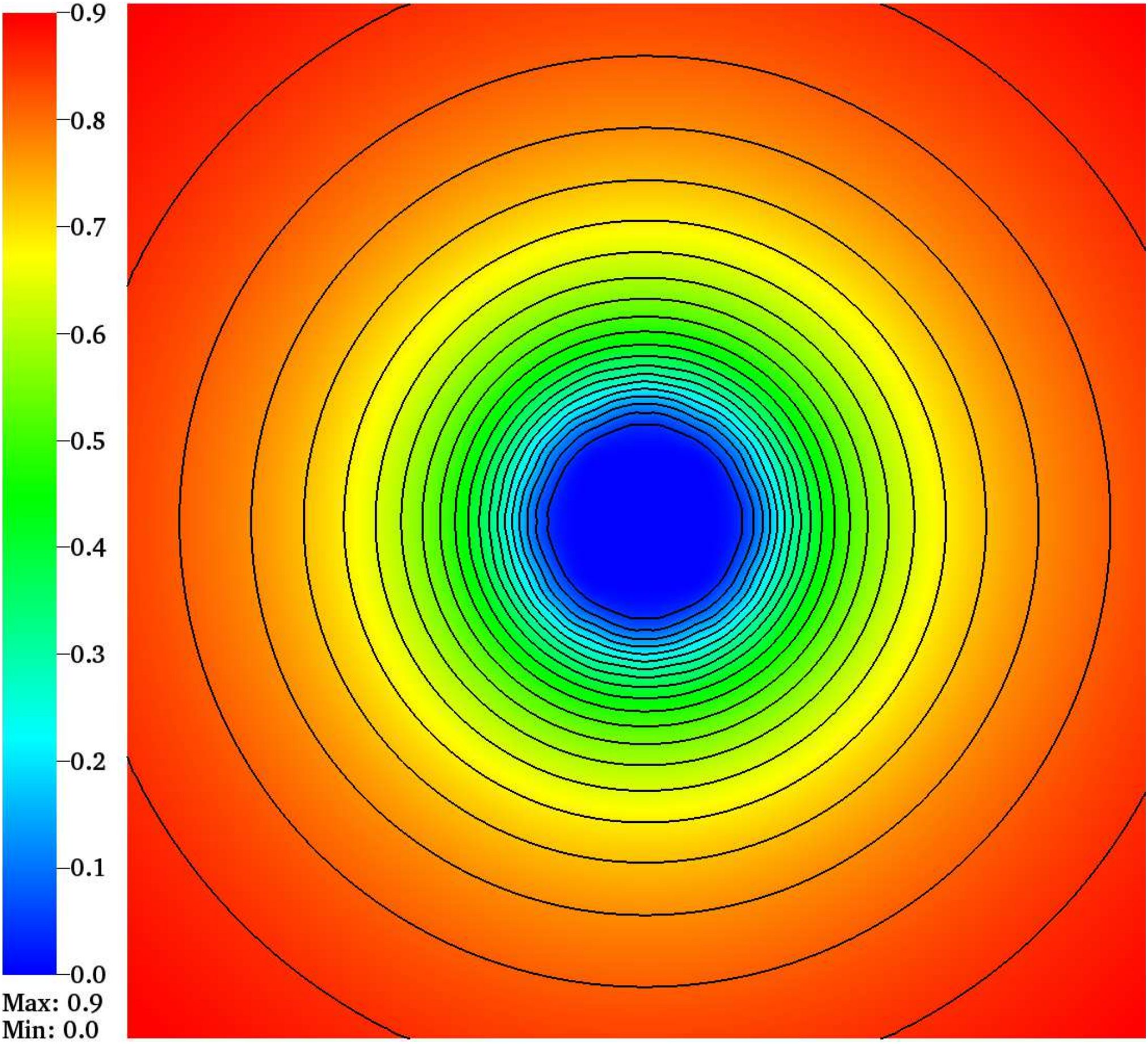} &
			\includegraphics[width=0.32\textwidth]{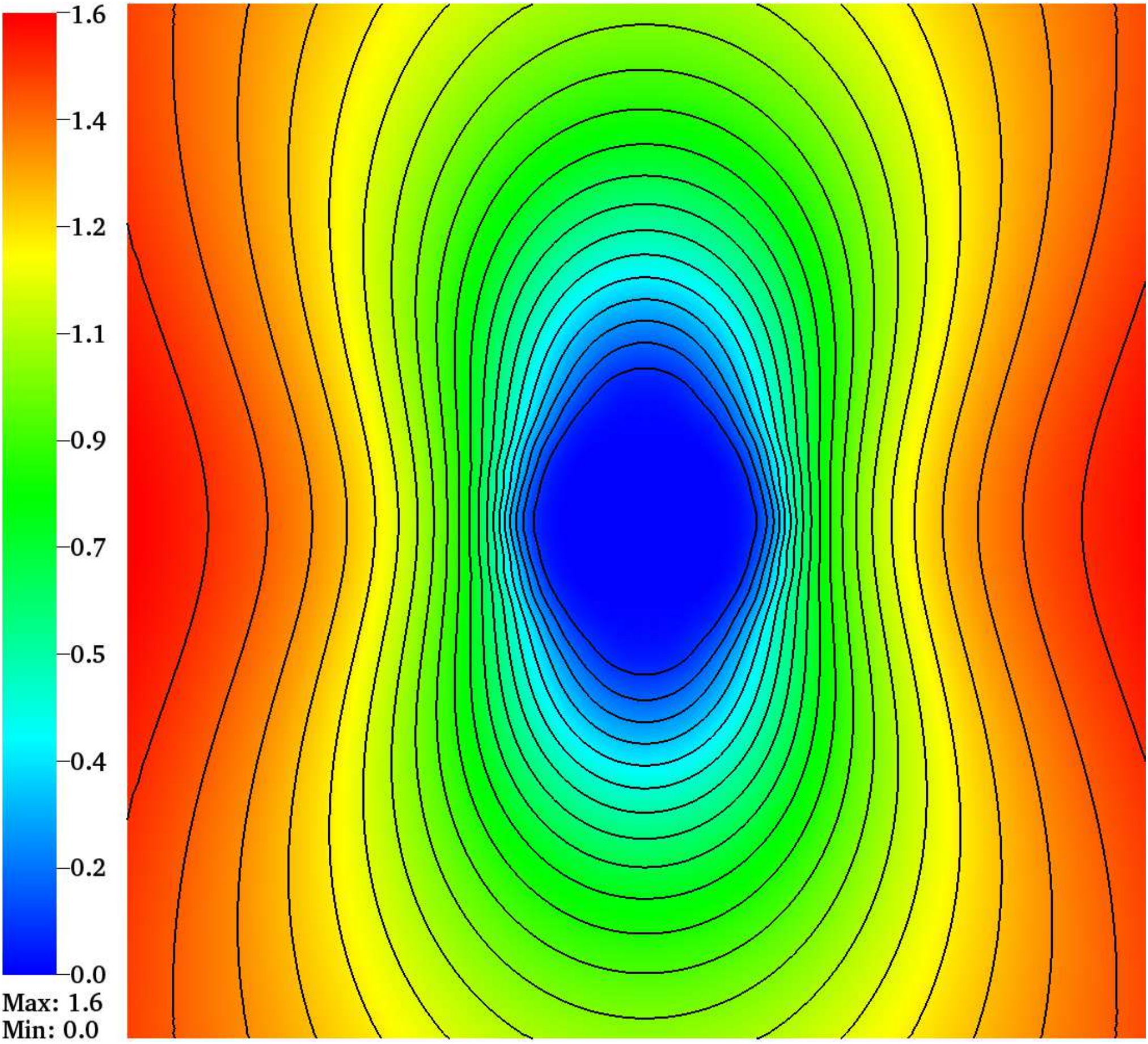} &
			\includegraphics[width=0.32\textwidth]{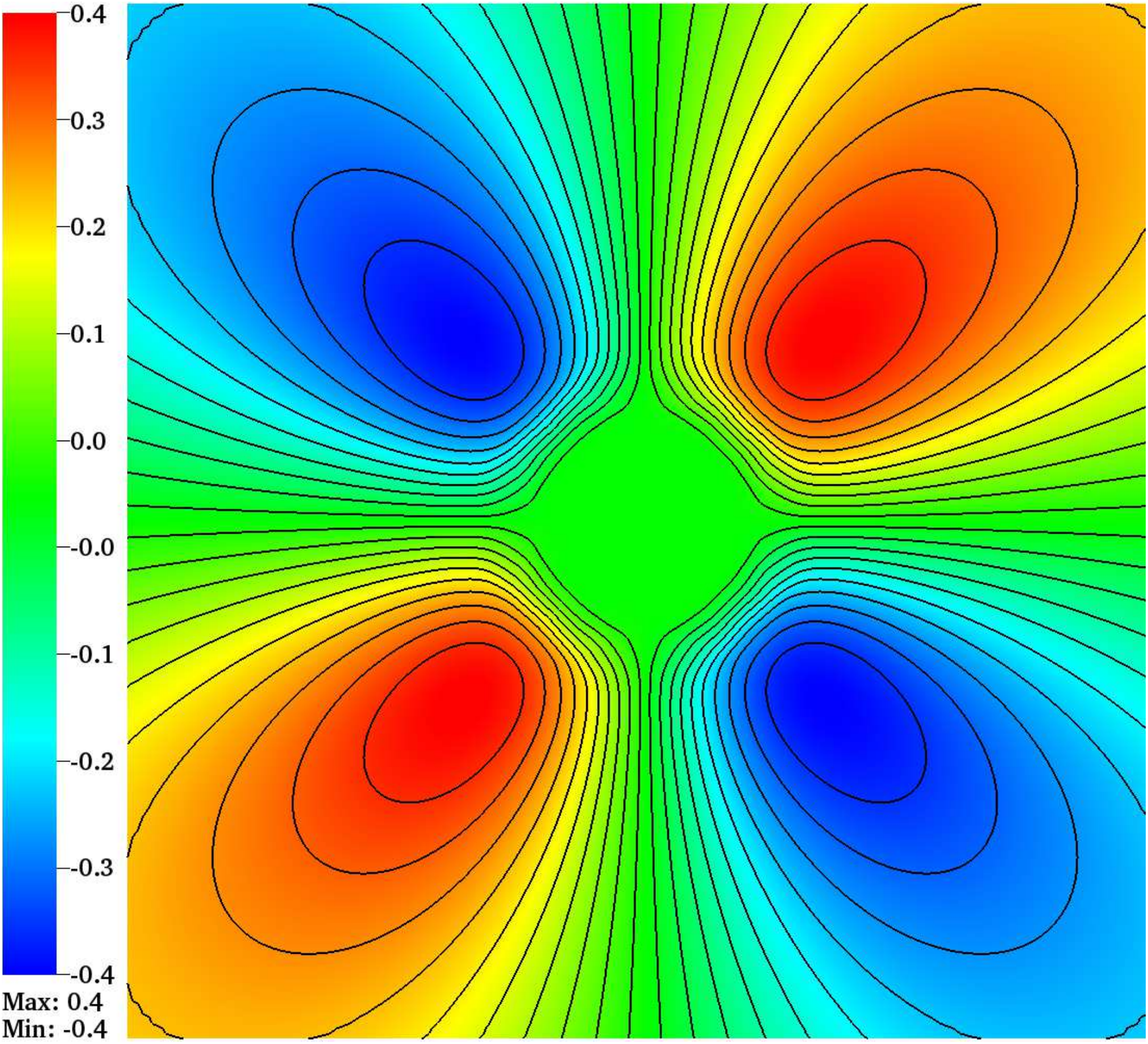} \\

			(a) $\rho$ & (b) $p_{11}$ & (c) $p_{12}$ \\
			\includegraphics[width=0.32\textwidth]{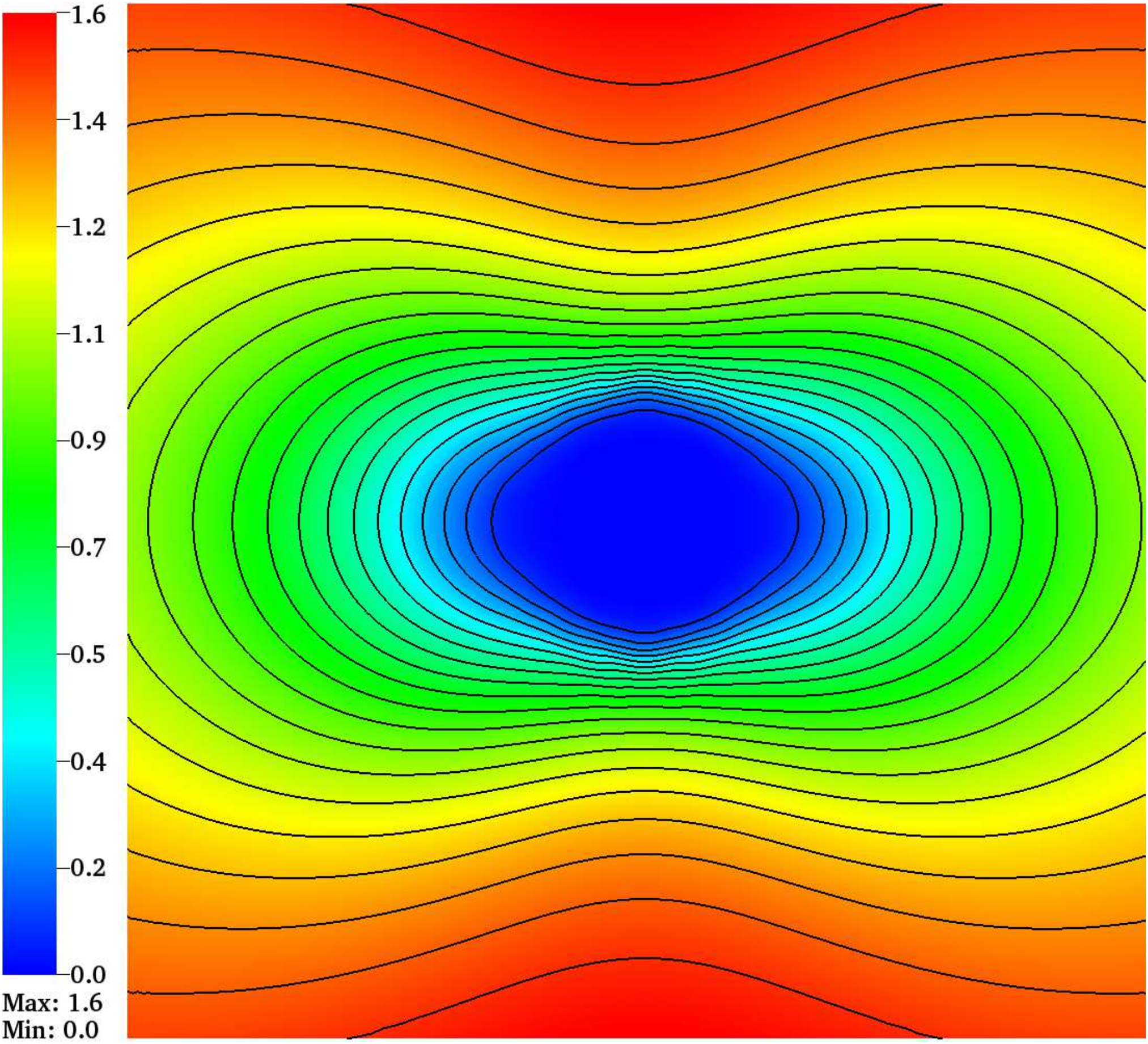} &
			\includegraphics[width=0.32\textwidth]{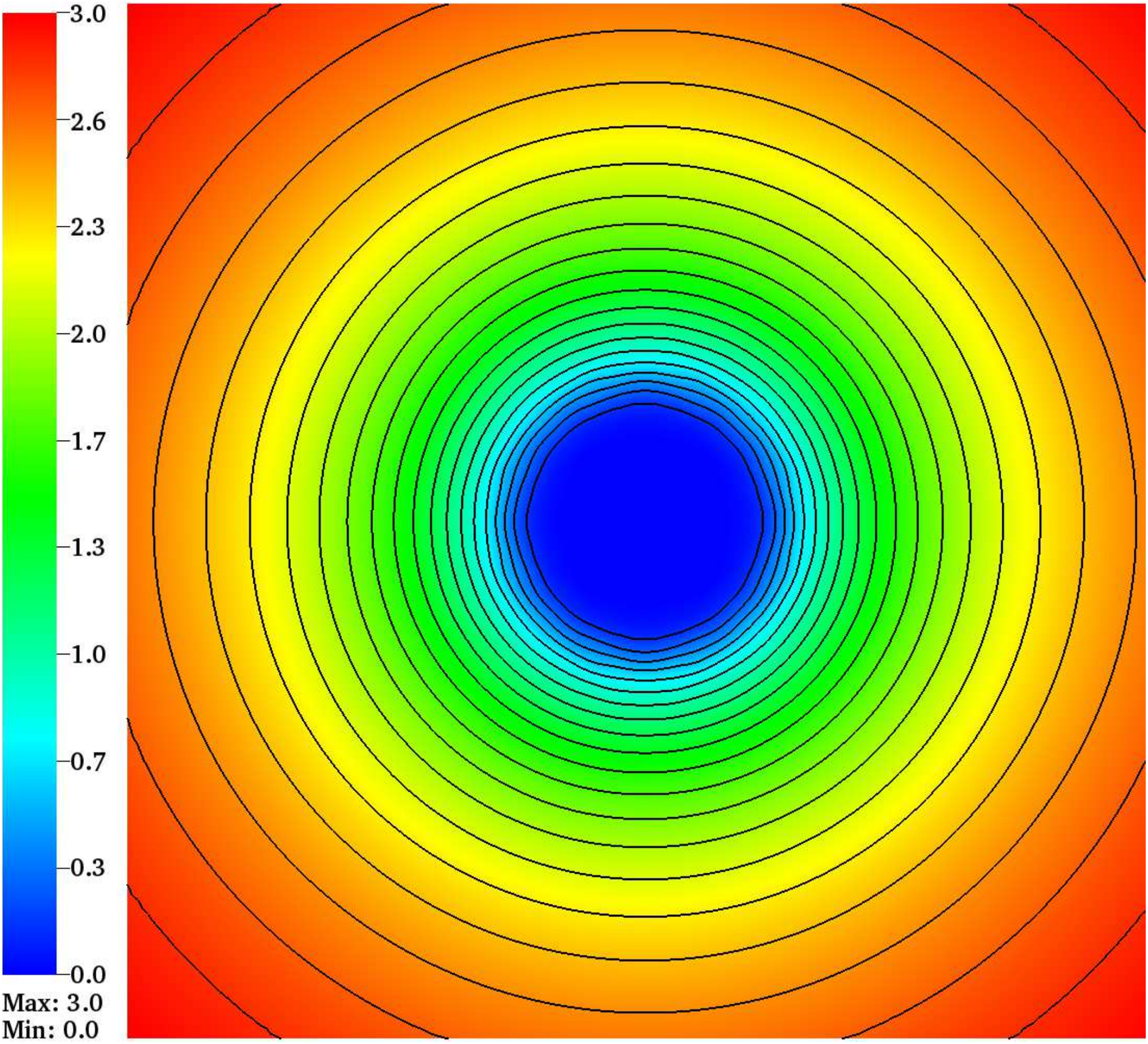} &
			\includegraphics[width=0.32\textwidth]{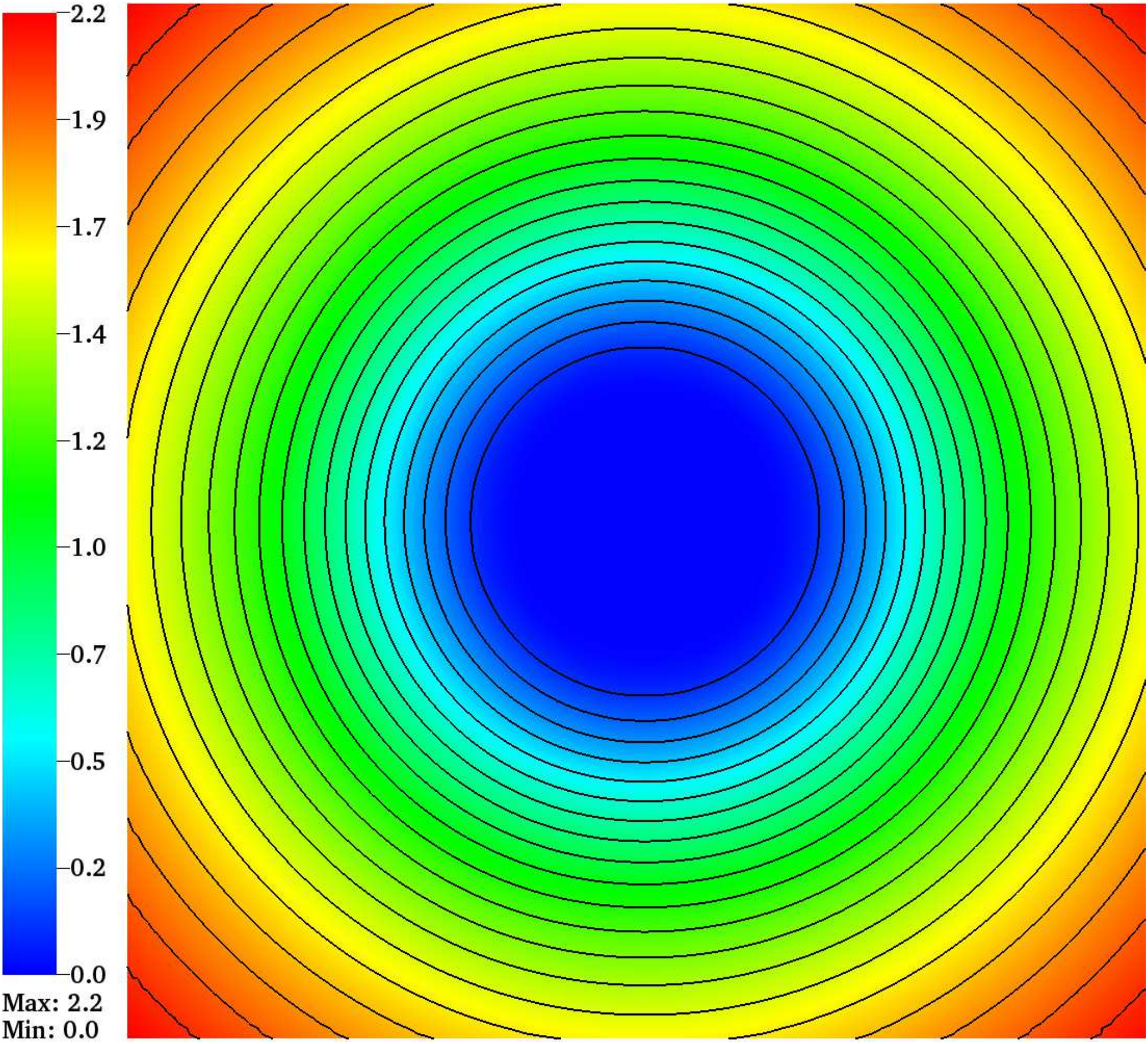}\\
			(d) $p_{22}$ & (e) $p_{11}+p_{22}$ & (f)  $p_{11}p_{22}-p_{12}^{2}$
		\end{tabular}
	\end{center}
	\caption{Plots of primitive variables with trace and determinant of pressure tensor with 20 contour lines for Example \ref{ex8} over the domain $[-2,2]\times [-2,2]$.}
	\label{fig:nv2d}
\end{figure}
\begin{example}\label{ex7}{\rm (1-D near vacuum test)
This Riemann problem \cite{ber_06a, ber-etal_15a} is used to test the positivity of the numerical schemes. The exact solution contains low density and pressure areas near the origin, which is difficult to compute using high order schemes without  a positivity limiter. The initial condition is given with a  discontinuity at the origin as follows
  \[
  \bold{V}=\begin{cases}
              (1, \ -5,  \ 0, \ 2,  \ 0, \ 2),& \mbox{if}~~x \leq 0\\
              (1, \ 5, \ 0, \ 2, \ 0, \ 2), &\mbox{if}~x>0
             \end{cases}
\]
 }
\end{example}
over the domain $[-0.5,0.5]$ with outflow boundary conditions. This  test case involves the propagation of two rarefaction waves away from the center leading to a low density and low pressure area near the origin.

 Numerical solutions computed using WENO schemes are presented at time $T=0.05$ using 100 mesh points. In Figure \ref{fig:ex7}, we have shown the numerical results obtained using considered WENO schemes for density and pressure components 
 $p_{11}$, $p_{22}$. Here, the reference solution is computed using WENO-JS scheme with 1000 mesh points. From Figure \ref{fig:ex7}, we can observe that the WENO-AO scheme result is closer to the reference solution as compared to the other considered WENO schemes. In Figure \ref{fig:ex7ao}, we have compared the result obtained using  WENO-AO scheme with increasing number of mesh points. The WENO-AO solution moves closer to the exact solution as we increase the number of mesh points. 
 Upto now, we have observed that WENO-AO performs better than the other two considered WENO schemes. In order to save space, henceforth we will show results obtained using WENO-AO scheme only. 
 
 The positivity of the solution is maintained by calling the positivity limiter during simulations. To maintain the positivity, we use a uniform CFL
 number equal to 1/12. Here, we want to compare the uniform CFL with adaptive  CFL. In the adaptive CFL, initially we set  CFL=0.95 and compute the solution. If solution fails to satisfy
 the positivity condition, we go back to previous level and reset the CFL equal to 1/12. This process is repeated until the final time is reached. In Table \ref{Table.ex7}, we have compared the CPU time (in second) taken by WENO-AO scheme using adaptive and uniform CFL. We can observe from Figure \ref{fig:ex7cpu} (a) and Table \ref{Table.ex7} that WENO-AO scheme with adaptive CFL is 10 times faster than WENO-AO scheme with uniform CFL. In Figure \ref{fig:ex7cpu} (b), we have compared the solutions obtained using adaptive and uniform CFL with 100 mesh points. The solution obtained using adaptive CFL is comparable with that obtained with uniform CFL condition.

\begin{example} \label{exShuOsher}{\rm (Shu-Osher problem)
This example is the modified version of Shu-Osher test \cite{shu-osh_89a} for the Ten-Moments equations. The primitive variables are initialized as follows
\[
\bold{V}=\begin{cases}
(3.857143, \ 2.699369,  \ 0, \ 10.33333,  \ 0, \ 10.33333),& \mbox{if } x \leq -4\\
(1+0.2 \sin(5x), \ 0, \ 0, \ 1, \ 0, \ 1), &\mbox{if } x > -4
\end{cases}
\]
Numerical simulations are performed over the domain $[-5,5]$ using 200 mesh points at time $T=1.8$. In Figure \ref{fig:ShuOsher}, we have compared the solution obtained using WENO-AO scheme with the second order finite volume scheme \cite{meenaKumarFVM} along with the reference solution. The reference solution is computed using finite volume scheme with 5000 mesh points. We can observe from the zoom versions of Figure \ref{fig:ShuOsher}, the WENO-AO scheme resolves the shock and high-frequency wave more accurately in comparison to second order finite volume scheme.	This again demonstrates the utility of high order scheme like WENO for computation of multiscale phenomena as compared to second order schemes.
}
\end{example}

\begin{example}\label{ex8}{\rm (2-D near vacuum test)
Here we consider a two-dimensional test case which contains low density and pressure areas \cite{mee-kum_17a}. The test is a generalization of one-dimensional case presented in Example \ref{ex7}. The simulations are performed over the  domain $[-2,2]\times[-2, 2]$ with outflow boundary conditions. The domain is filled with fluid at constant unit density with pressure $p_{11}=2,\;\;p_{12}=0$ and $p_{22}=2$. The velocity is taken to be $8(x/r,y/r)$, where $r = \sqrt{x^2+y^2}$ so the gas is flowing radially outward. The computational results are obtained for final time 0.05 using $200\time200$ cells. The results for density, pressure components and determinant of pressure tensor are shown in Figure \ref{fig:nv2d}. In Table \ref{Table.ex8} and Figure \ref{fig:ex8} (a), we have compared the numerical results obtained using adaptive CFL with that of uniform CFL. We observe that as we increase the number of mesh points, WENO-AO with adaptive CFL becomes faster than WENO-AO with uniform CFL. In Figure \ref{fig:ex8}, we have compared the solution obtained using adaptive CFL and uniform CFL  along the $y=0$ line. Both the solutions are comparable to each other.
 }
\end{example}
\begin{figure}
\begin{center}
\begin{tabular}{cc}
\includegraphics[width=0.50\textwidth]{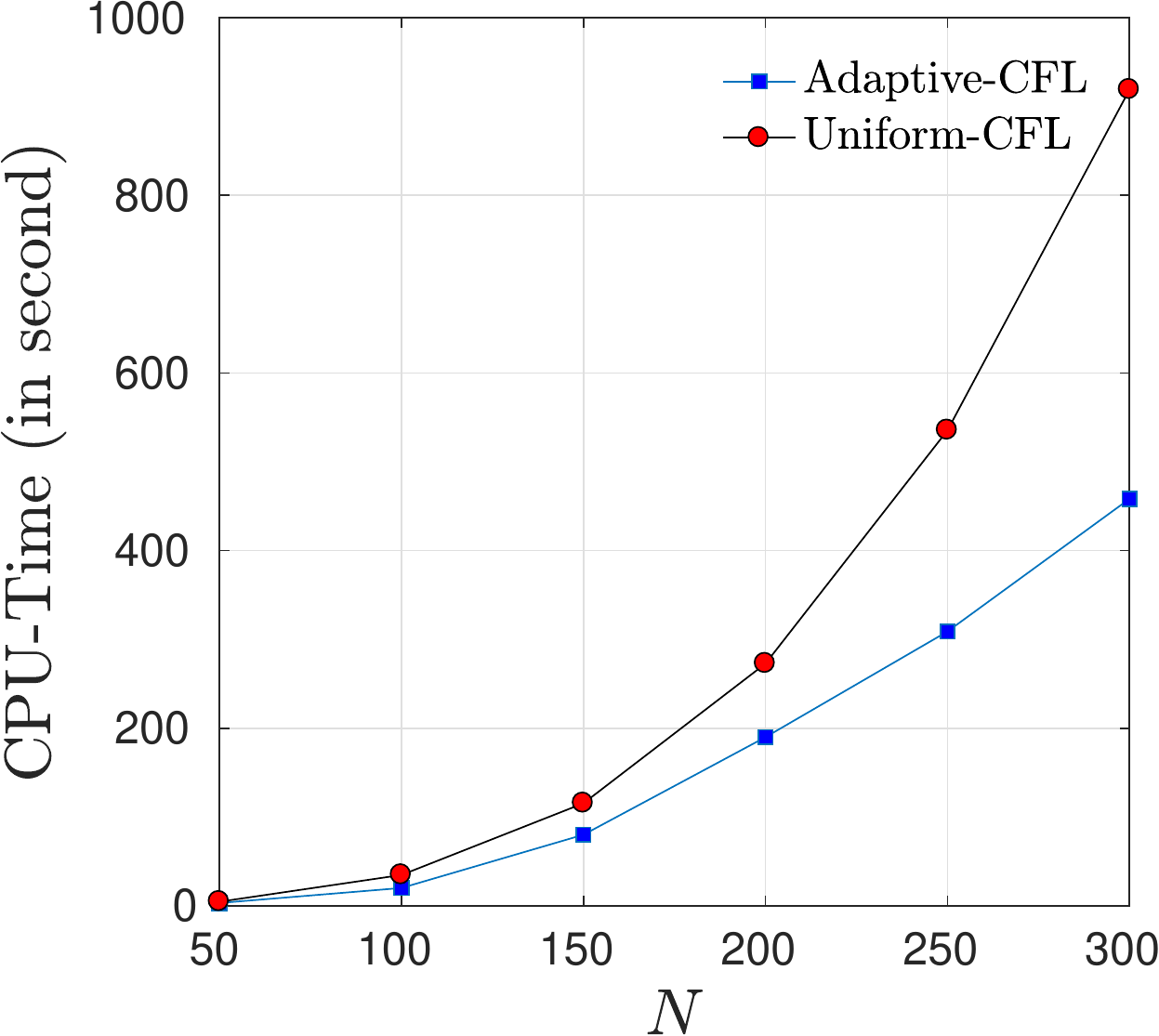} &
 \includegraphics[width=0.48\textwidth]{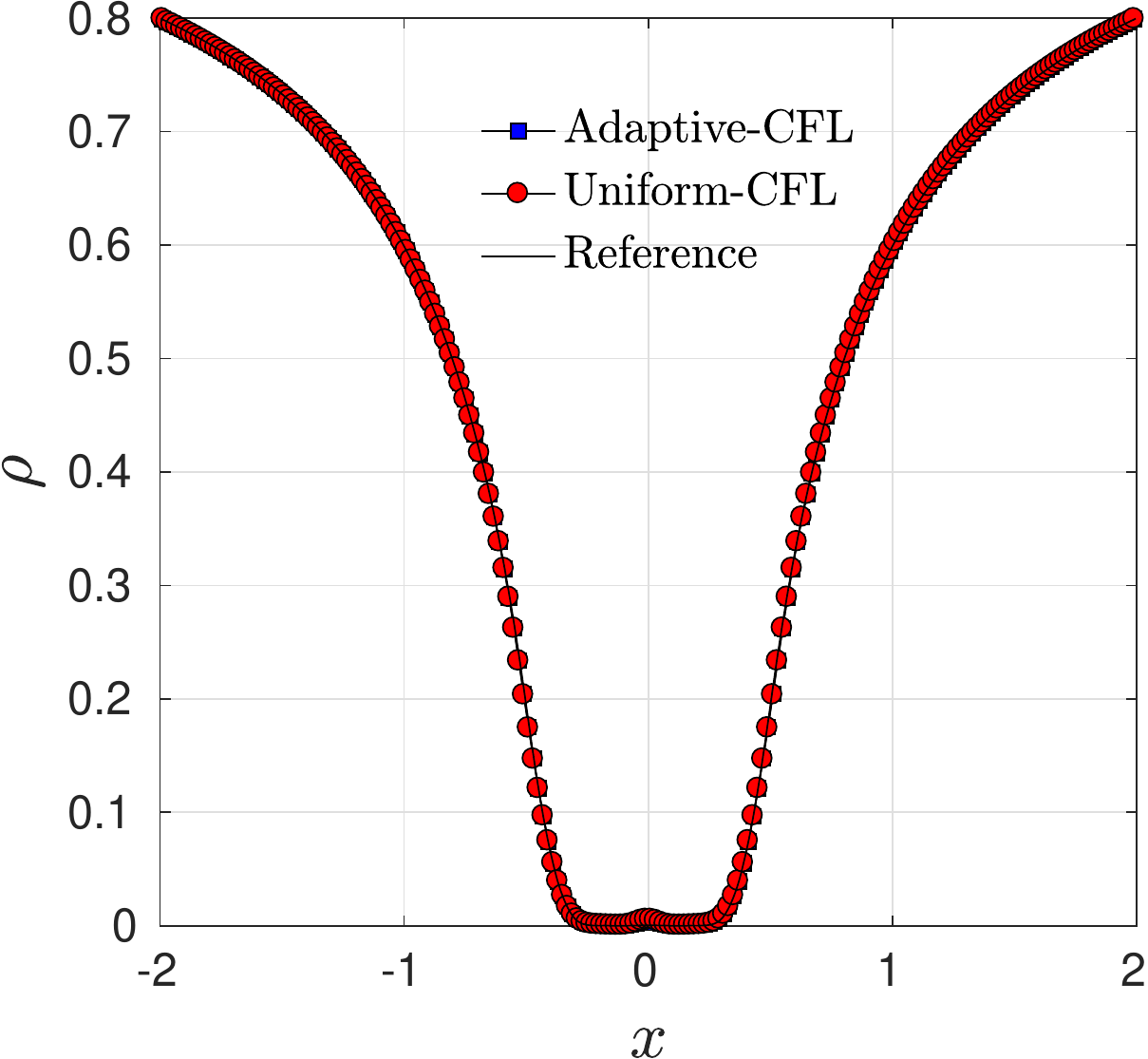} \\
 (a)  & (b) 
\end{tabular}
\end{center}
 \caption{(a) Comparison of CPU time in seconds of WENO-AO scheme with adaptive and uniform CFL for Example \ref{ex8} for various number of mesh points, (b) Comparison of density solution obtained using adaptive and uniform CFL for $100\times 100$ mesh points along the cross-section at $y=0$.}
 \label{fig:ex8}
\end{figure}
\begin{example}\label{ex9}{\rm (2-D discontinuous near vacuum test \cite{mee-kum_17a})
In this test case, we consider an initial condition with a discontinuity at unit radius in the domain $[-2,2]\times[-2,2]$ and the initial primitive variables are given by
 \[
  \bold{V}=\begin{cases}
              (1, \ 8\hat{r},  \ 8\hat{r}, \ 2, \ 0, \ 2), & \mbox{if } \sqrt{x^{2}+y^{2}}  < 1\\
              (1, \ 0, \ 0, \ 1, \ 0, \ 1), & \mbox{otherwise}
             \end{cases}
 \]
where $\hat{r} = (x/r,y/r)$, $r=\sqrt{x^2+y^2}$ is the unit vector in the radial direction and boundary values are obtained using outflow boundary conditions. The solutions are computed till final time $T=0.05$ using $200\time200$ mesh points. The numerical results are displayed in Figure \ref{fig:dnv2d}.  
We can observe that  WENO-AO scheme computes the solution in a non-oscillatory fashion while maintaining the postivity of the solution. The WENO-AO scheme with adaptive CFL is four times faster than WENO-AO scheme with uniform CFL as  observed from Table \ref{Table.ex9}. In Figure \ref{fig:ex9cpu} (b), we have compared the solution obtained with adaptive CFL and uniform CFL along the axis $y=0$ using $100\times 100$ mesh points. Both solutions are comprable with each other which shows that it is preferable to use adaptive CFL algorithm.
 }
\end{example}
\begin{figure}
	\begin{center}
		\begin{tabular}{ccc}
			\includegraphics[width=0.32\textwidth]{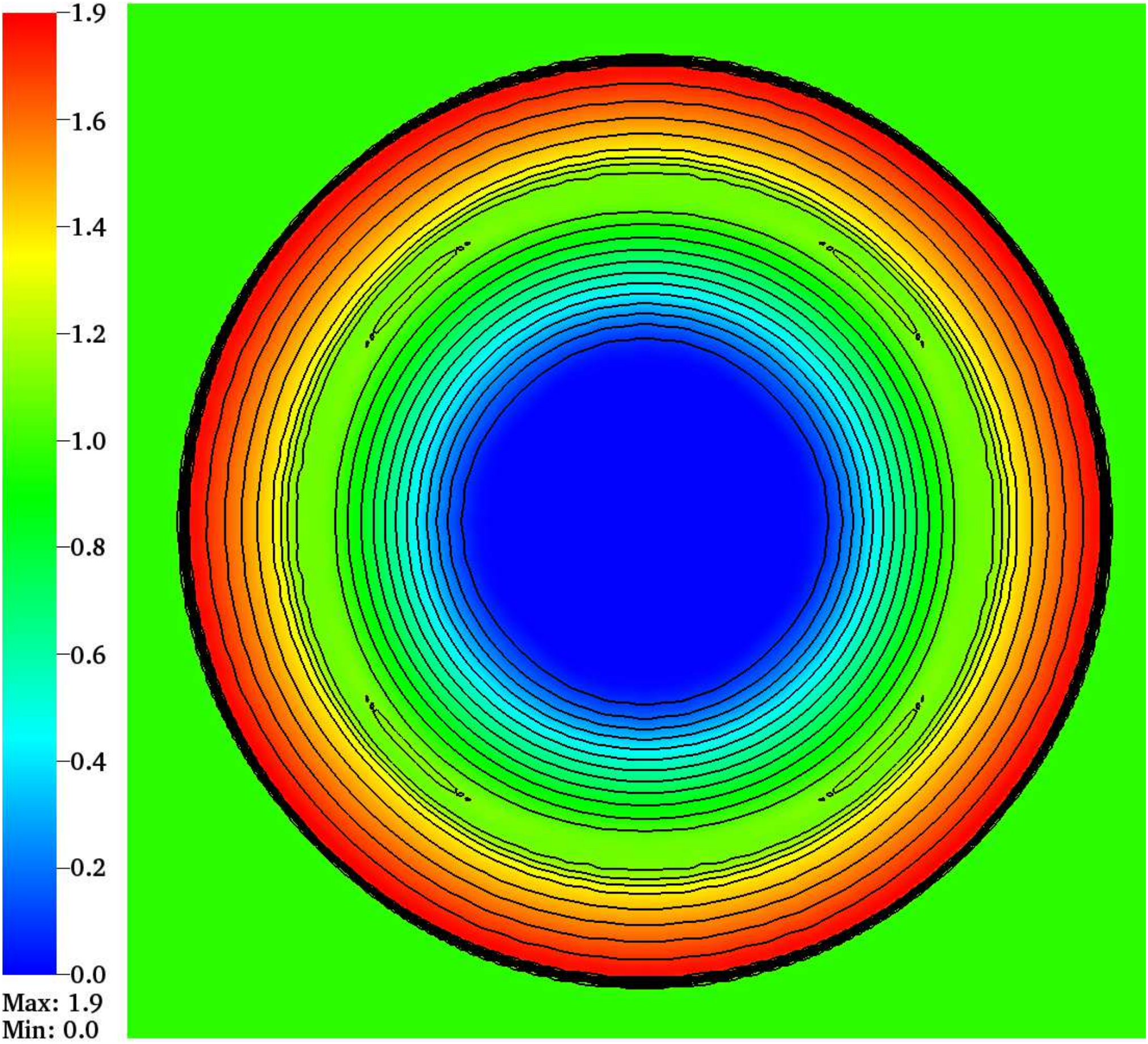} &
			\includegraphics[width=0.32\textwidth]{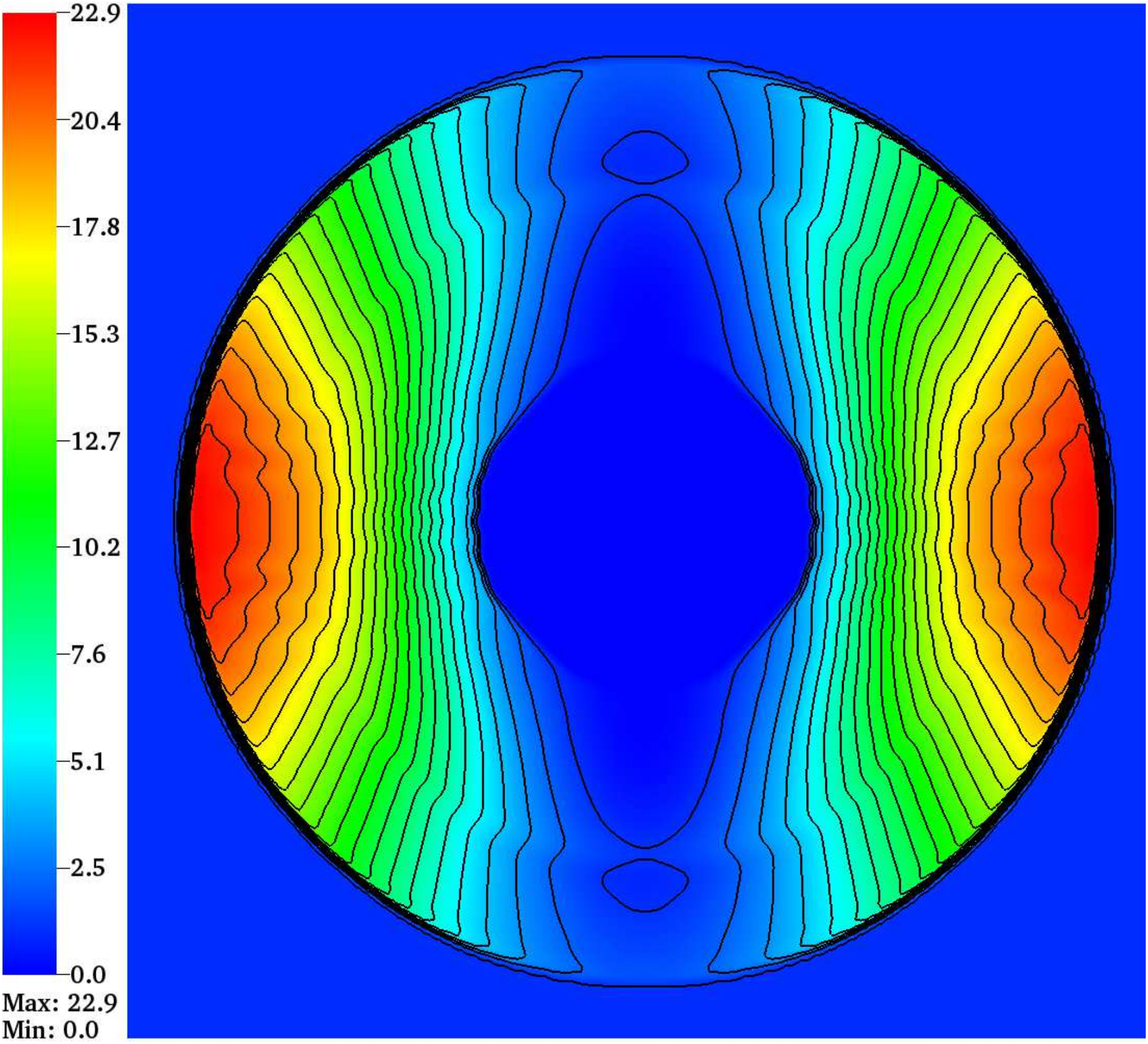} &
			\includegraphics[width=0.32\textwidth]{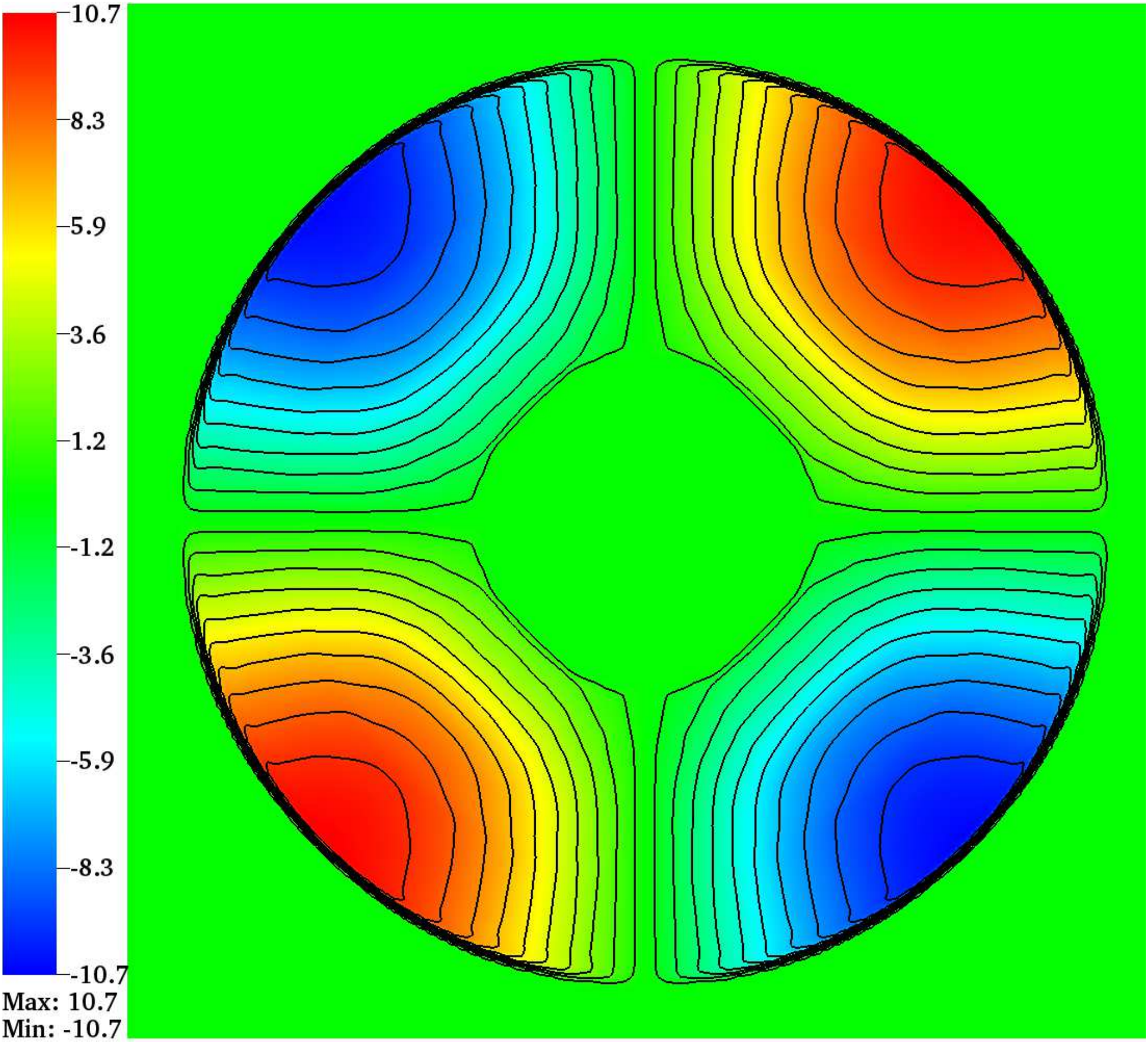} \\
			(a) $\rho$ & (b) $p_{11}$ & (c) $p_{12}$ \\
			\includegraphics[width=0.32\textwidth]{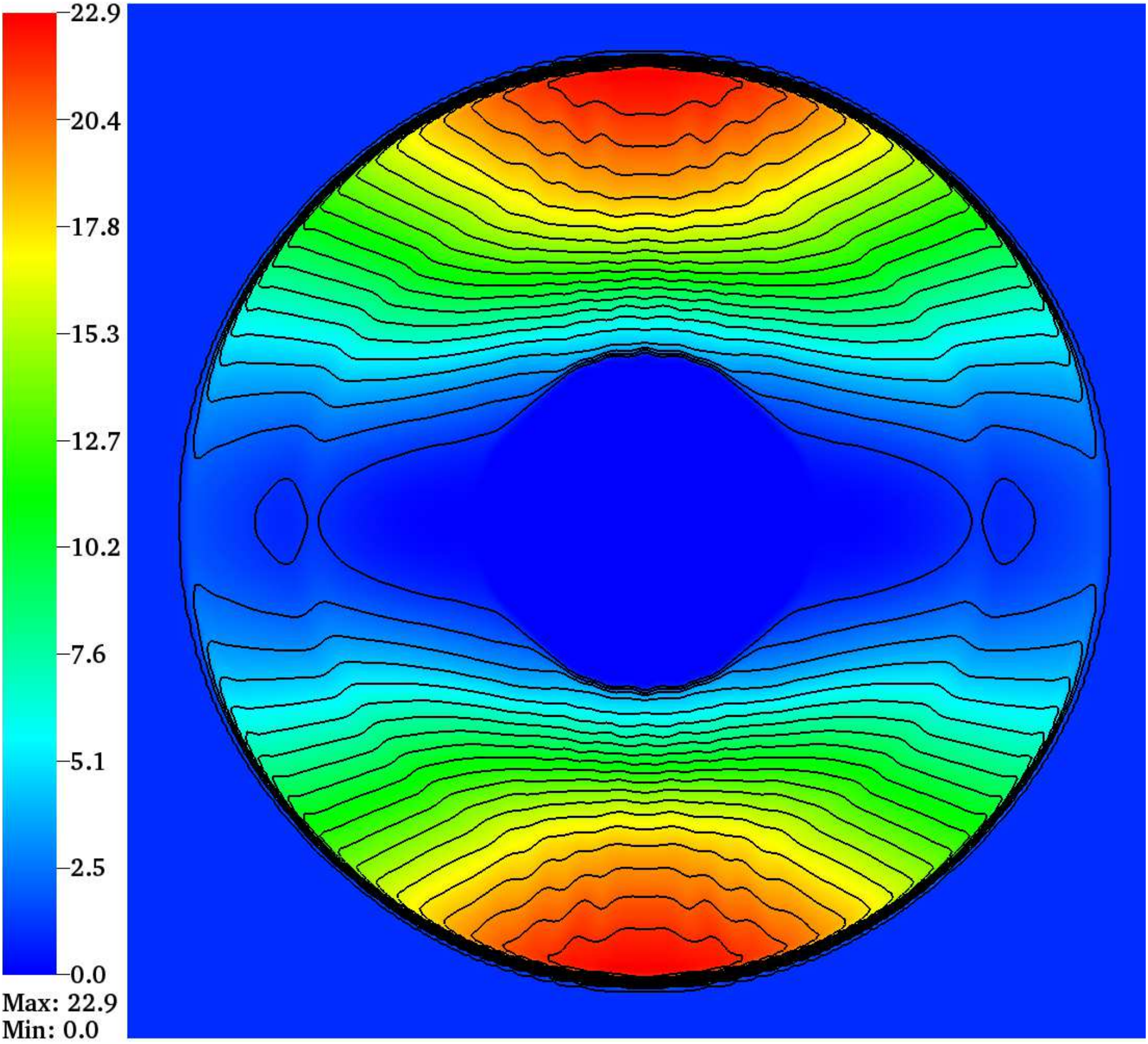} &
			\includegraphics[width=0.32\textwidth]{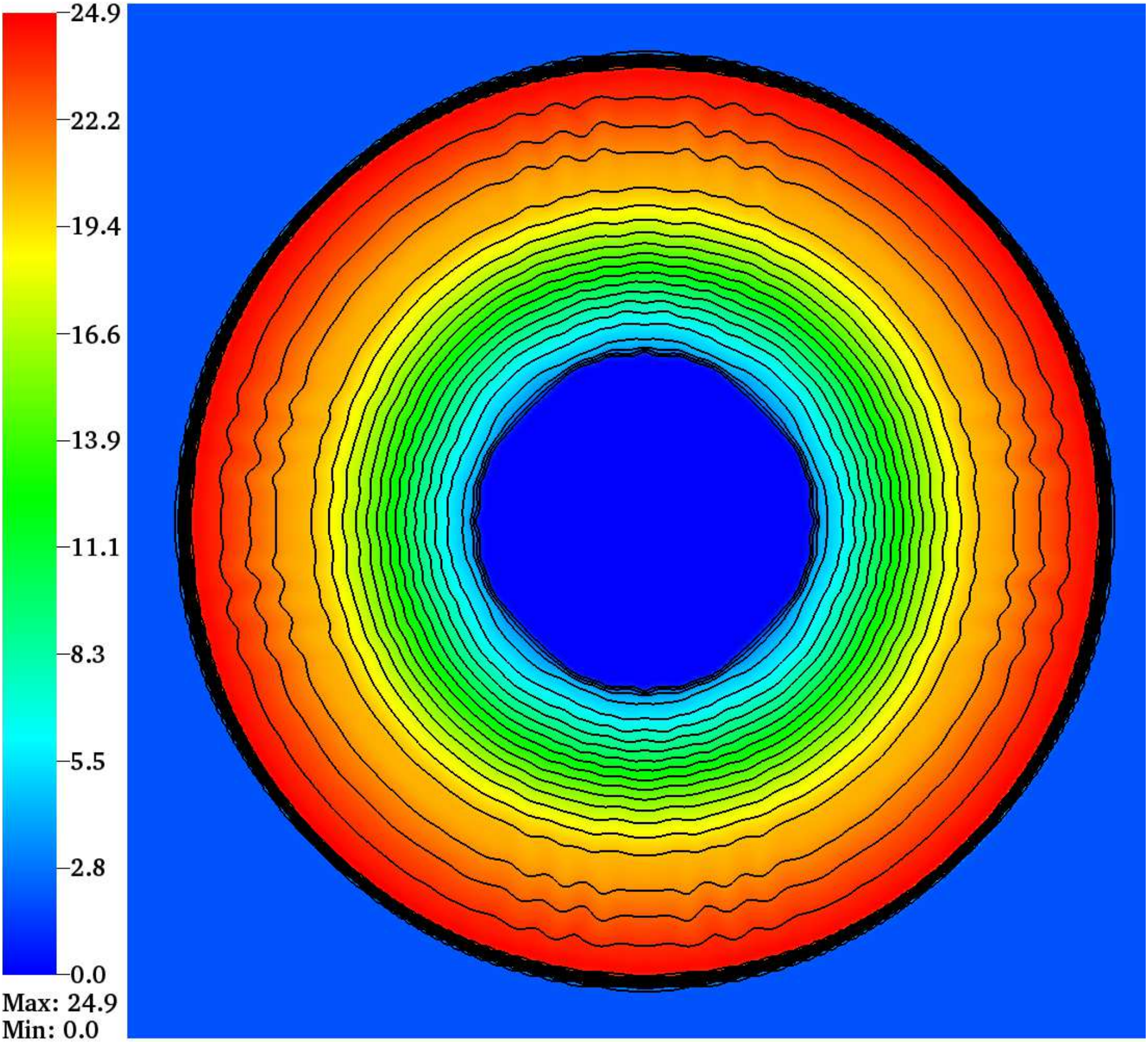} &
			\includegraphics[width=0.32\textwidth]{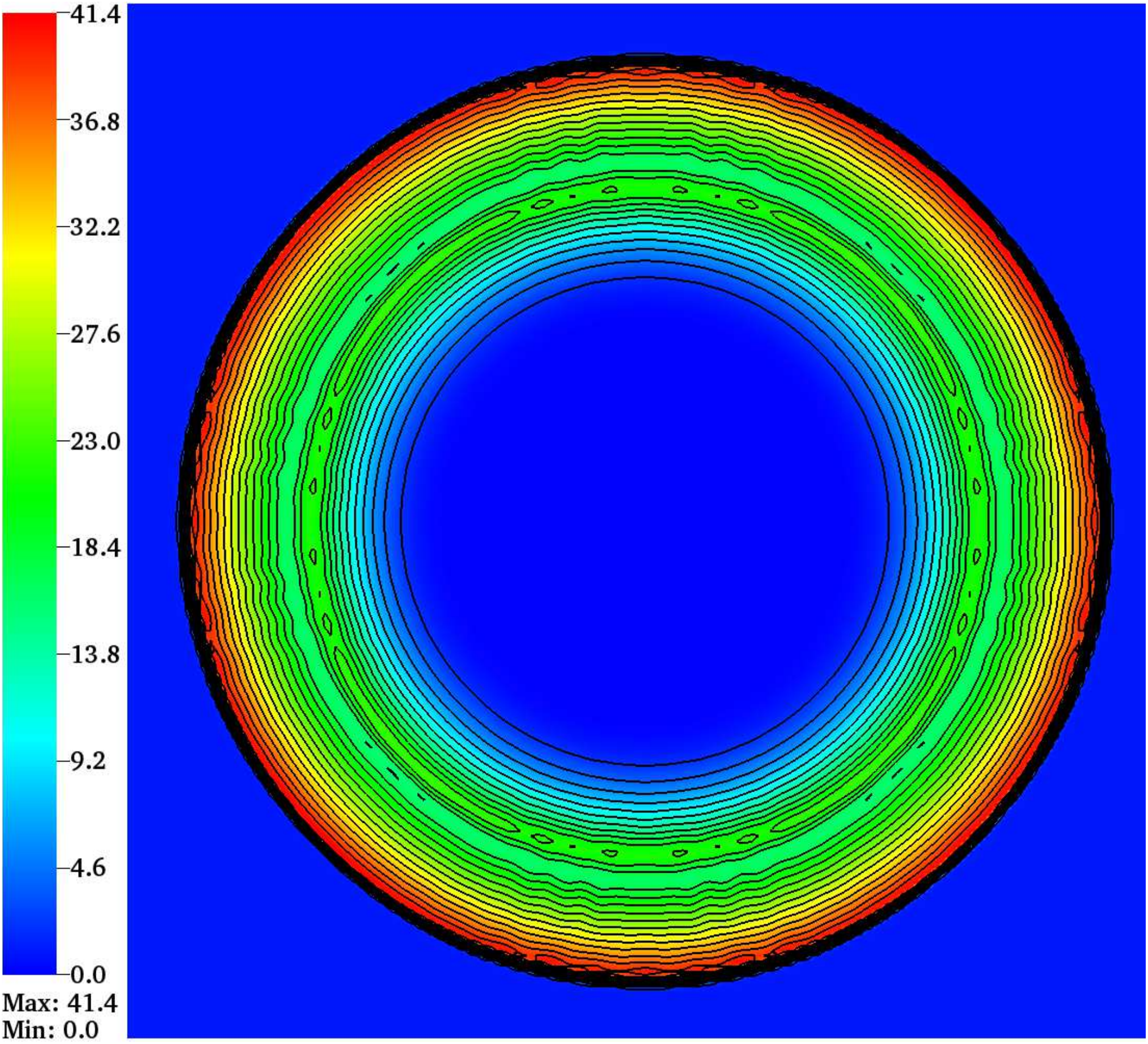}\\
			 (d) $p_{22}$ & (e) $p_{11}+p_{22}$ & (f)  $p_{11}p_{22}-p_{12}^{2}$
		\end{tabular}
	\end{center}
	\caption{Plots of primitive variables with trace and determinant of pressure tensor with 20 contour lines for Example \ref{ex9} over the domain $[-2,2]\times [-2,2]$.}
	\label{fig:dnv2d}
\end{figure}

\begin{figure}
\begin{center}
\begin{tabular}{cc}
\includegraphics[width=0.50\textwidth]{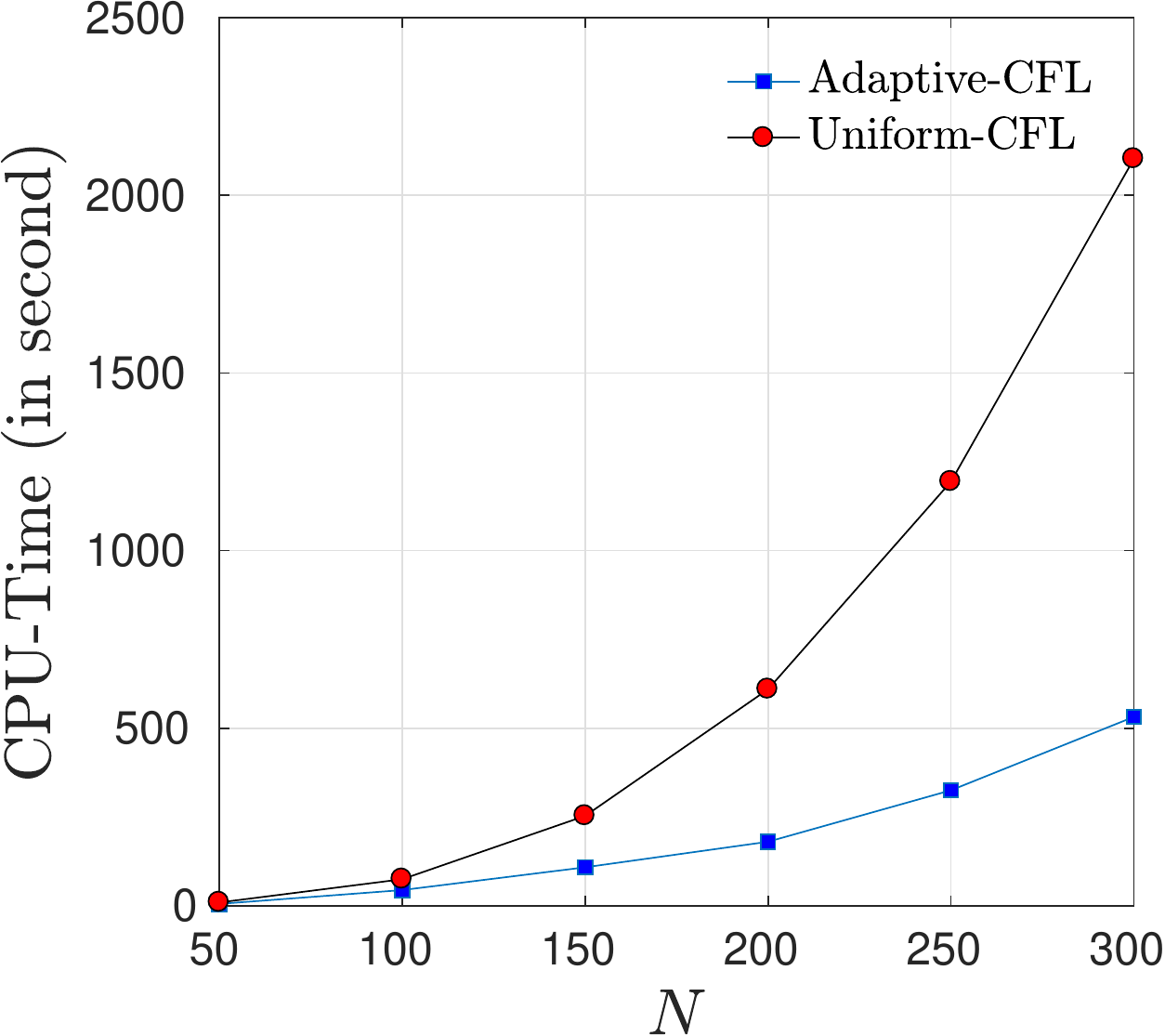} &
 \includegraphics[width=0.48\textwidth]{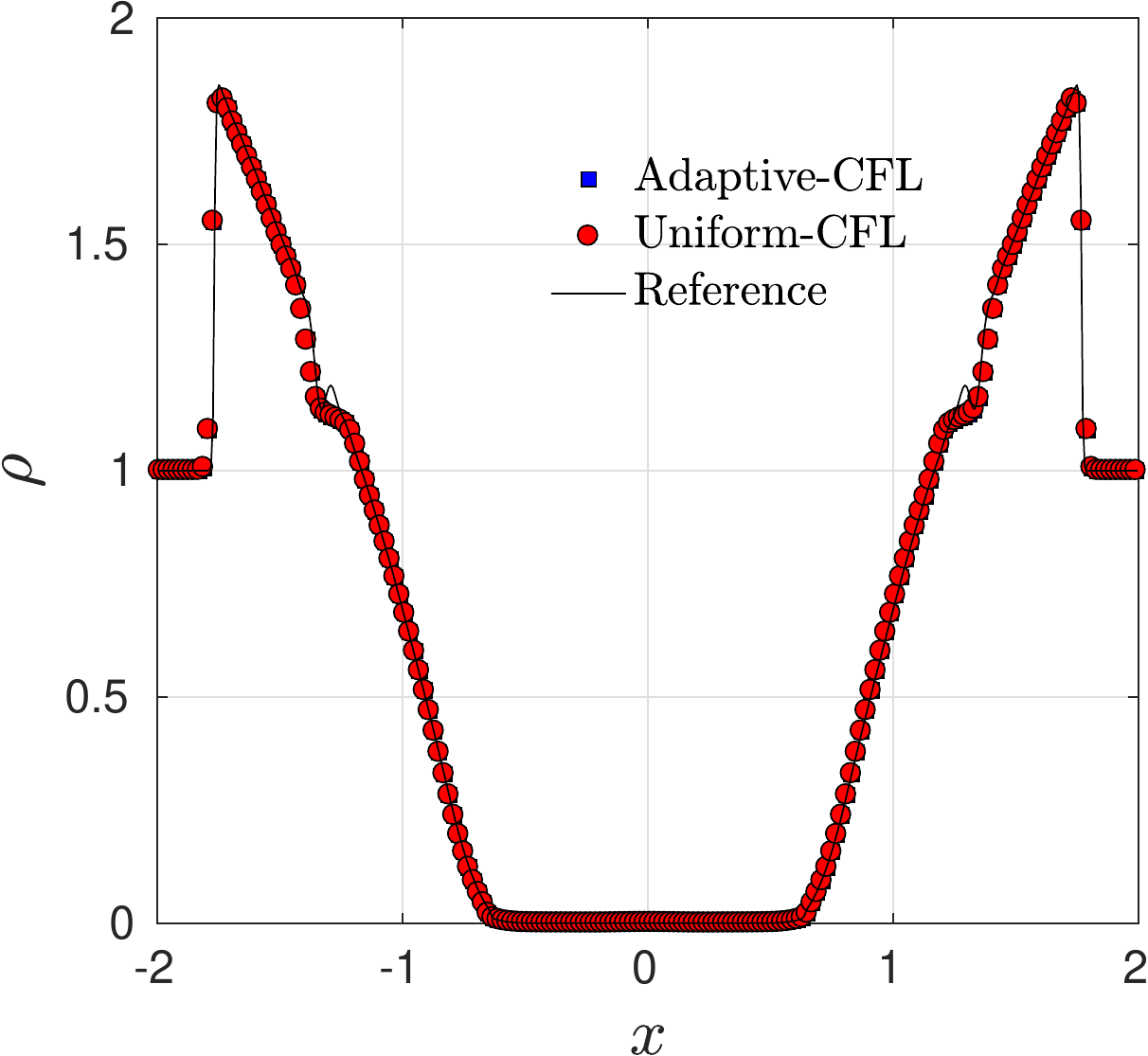} \\
 (a) & (b)
\end{tabular}
\end{center}
 \caption{(a) Comparison of CPU times of WENO-AO scheme with adaptive and uniform CFL for Example \ref{ex9} for varying number of mesh points, (b) Comparison of density solution obtained using adaptive and uniform CFL with $100\times 100$ mesh points along the cross-section at $y=0$.}
 \label{fig:ex9cpu}
\end{figure}

\begin{example}\label{ex:2drg}{\rm (Two-rarefaction waves with Gaussian source term)
In this problem, we analyze the effect of the source term on a two-dimensional  test case  containing the motion of two-rarefaction waves. Numerical experiments  are performed over the domain $[0,4]$ with initial conditions having  discontinuity at $x=2$. The initial profile of solution is  
 \[
  \bold{V}=\begin{cases}
              (1, \ -4,  \ 0, \ 9,  \ 7, \ 9),& \mbox{if } x \leq 2\\
              (1, \ 4, \ 0, \ 9, \ 7, \ 9), &\mbox{if } x>2
             \end{cases}
\]
and the Gaussian source term is given by,
\[
W(x,t) = 25\exp (-200 (x-2)^2)
\]
The numerical results are computed using  WENO-AO scheme with 500 grid points at time $T=0.1$. To see the effect of the source, we compute the solutions for the Riemann problem with source term (non-homogeneous) and without source term (homogeneous). In Figure \ref{fig:2ds} (a)-(b),  we have shown the density and pressure component $p_{11}$, respectively. The solution with and without source term  is predominantly in the low density regions in the middle of the domain. But the solution with source is seen to be much closer to the vacuum state. The solution without source term maintains its positivity without  calling the positivity limiter. On the other hand, the positivity of the solution with source term is preserved only by using the positivity limiter. The numerical results show good agreement with the published results \cite{mee-kum_17a, sen-kum_18a}.
 }
\end{example}

\begin{figure}
	\begin{center}
		\begin{tabular}{cc}
			\includegraphics[width=0.48\textwidth]{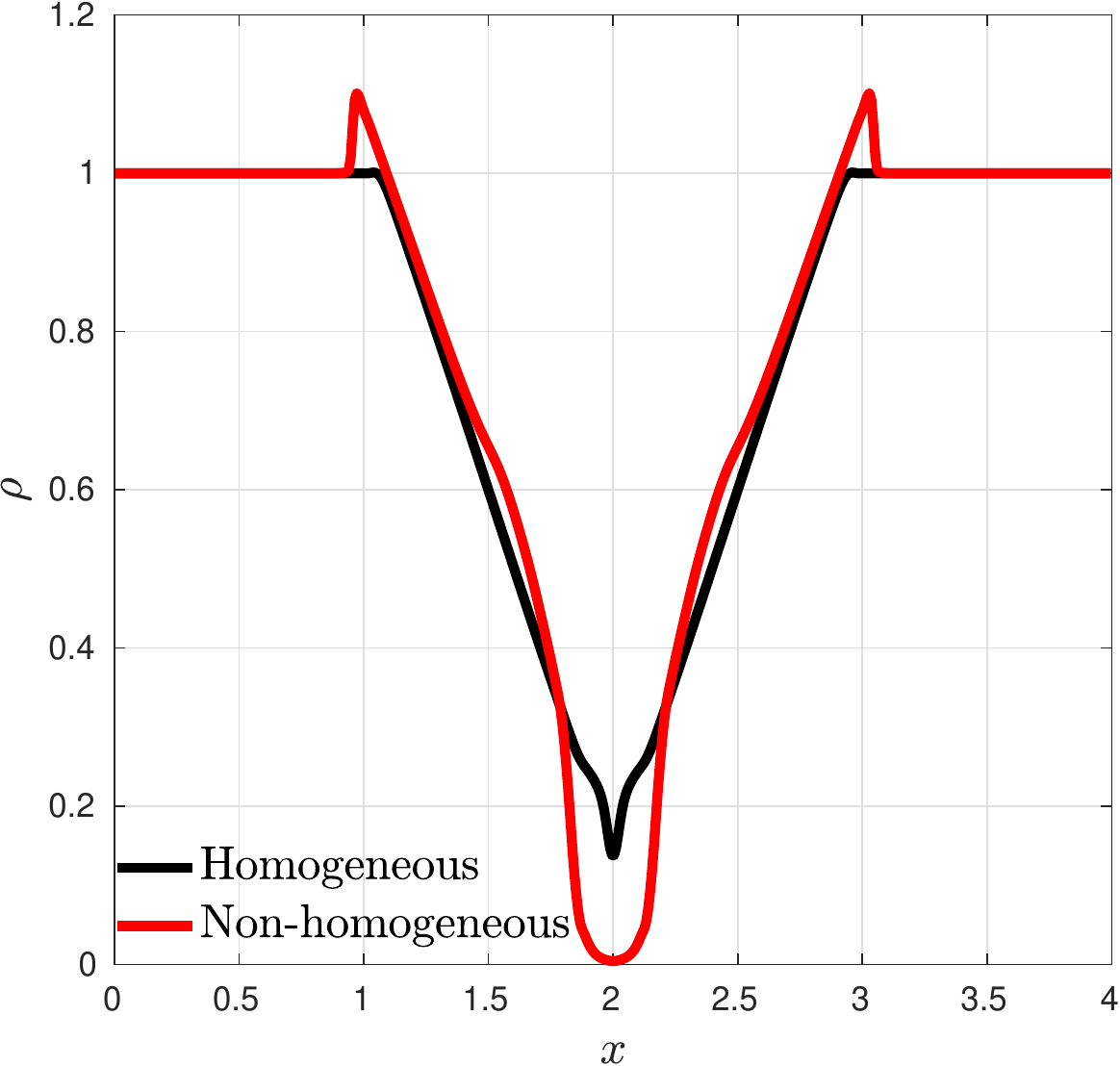} &
			\includegraphics[width=0.48\textwidth]{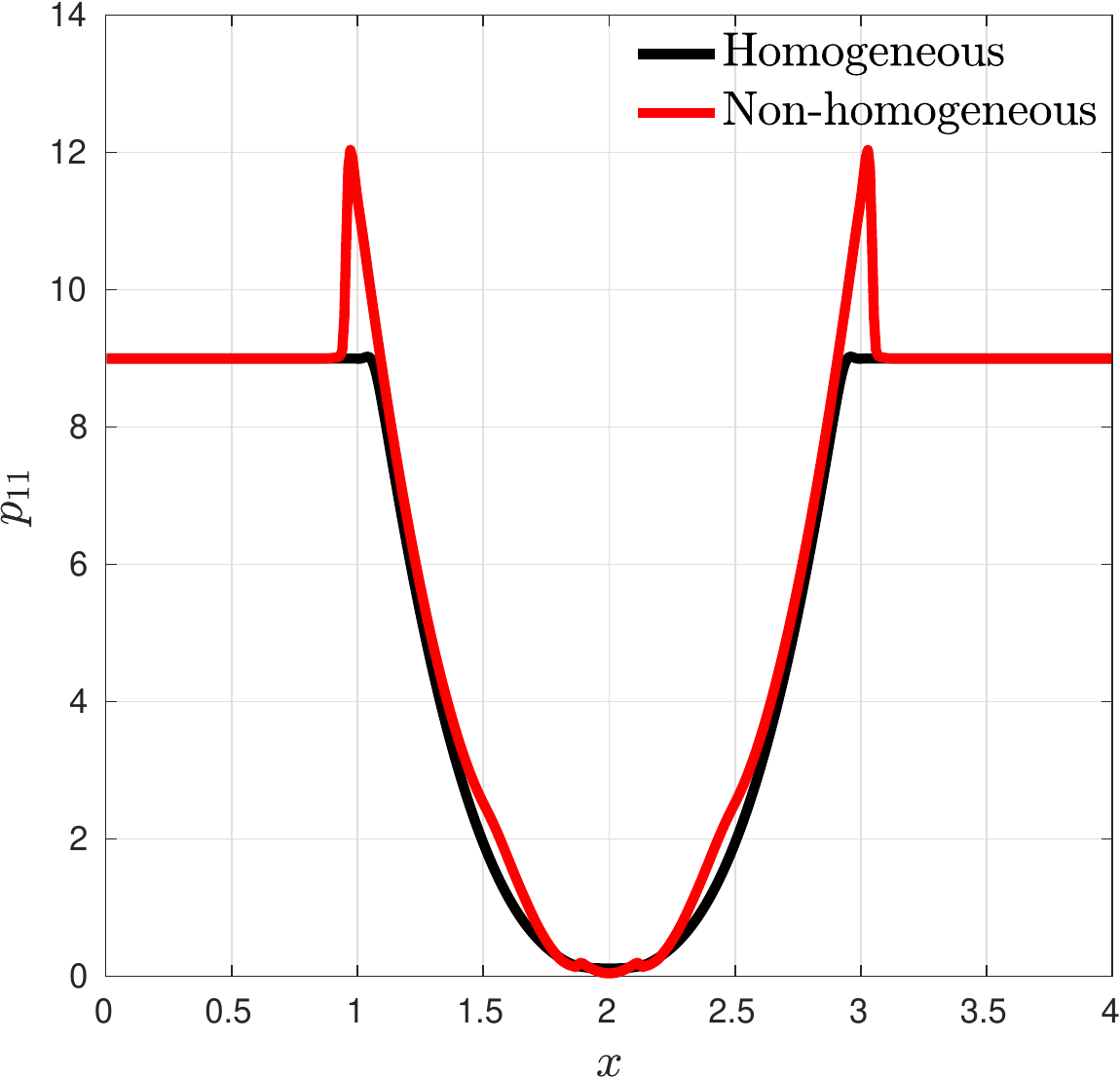} \\
			(a) Density & (b) Pressure component: $p_{11}$
		\end{tabular}
	\end{center}
	\caption{Numerical solution computed using WENO-AO scheme with 500 mesh points for Example \ref{ex:2drg}.}
	\label{fig:2ds}
\end{figure}

\begin{example}\label{ex:unp}{\rm (Uniform plasma state with two-dimensional Gaussian source \cite{mee-kum_17a})
This test is considered to see the two-dimensional effect of the source term \cite{mee-kum_17a, sen-kum_18a}. Here, we consider a uniform initial plasma state given by $\bold{V}=(0.1,0, 0,9, 7,9 )$ 
and a two-dimensional Gaussian source term given by,
\[
W(x,y,t) = 25\exp (-200  ((x-2)^2 + (y-2)^2))
\]
over the domain $[0,4]\times[0,4]$. Numerical results are computed at final time $T=0.1$ using $200\times200$ cells and outflow boundary conditions on all boundaries. We display the two-dimensional density  and pressure results in the Figures \ref{fig:s2up} (a)  and \ref{fig:s2up} (b), respectively. In Figure \ref{fig:s21d}, we have compared the one-dimensional cross-sections of density solution along the $y=-x+4$ line obtained using WENO-AO scheme  and the solution obtained using second order finite volume scheme~\cite{meenaKumarFVM} with different resolutions. Since we do not have an exact solution, the reference solution is obtained using the second order finite volume scheme~\cite{meenaKumarFVM} with $800\times 800$ grid points. From Figure~\ref{fig:s2up}-\ref{fig:s21d}, we can see that the solution consists of a low density area at the center of the domain. This is induced by the Gaussian source term, as a uniform state is preserved for a zero source term. The WENO-AO scheme obtains the solution without the use of positivity limiter and results are comparable with published results~\cite{mee-kum_17a, sen-kum_18a}.  The advantage of using a high order scheme is observed in Figure~\ref{fig:s21d} as the WENO-AO result on $200 \times 200$ mesh is close to the result of second order scheme on $800 \times 800$ mesh (Reference solution).
}
\end{example}

\begin{figure}
	\begin{center}
		\begin{tabular}{cc}
			\includegraphics[width=0.48\textwidth]{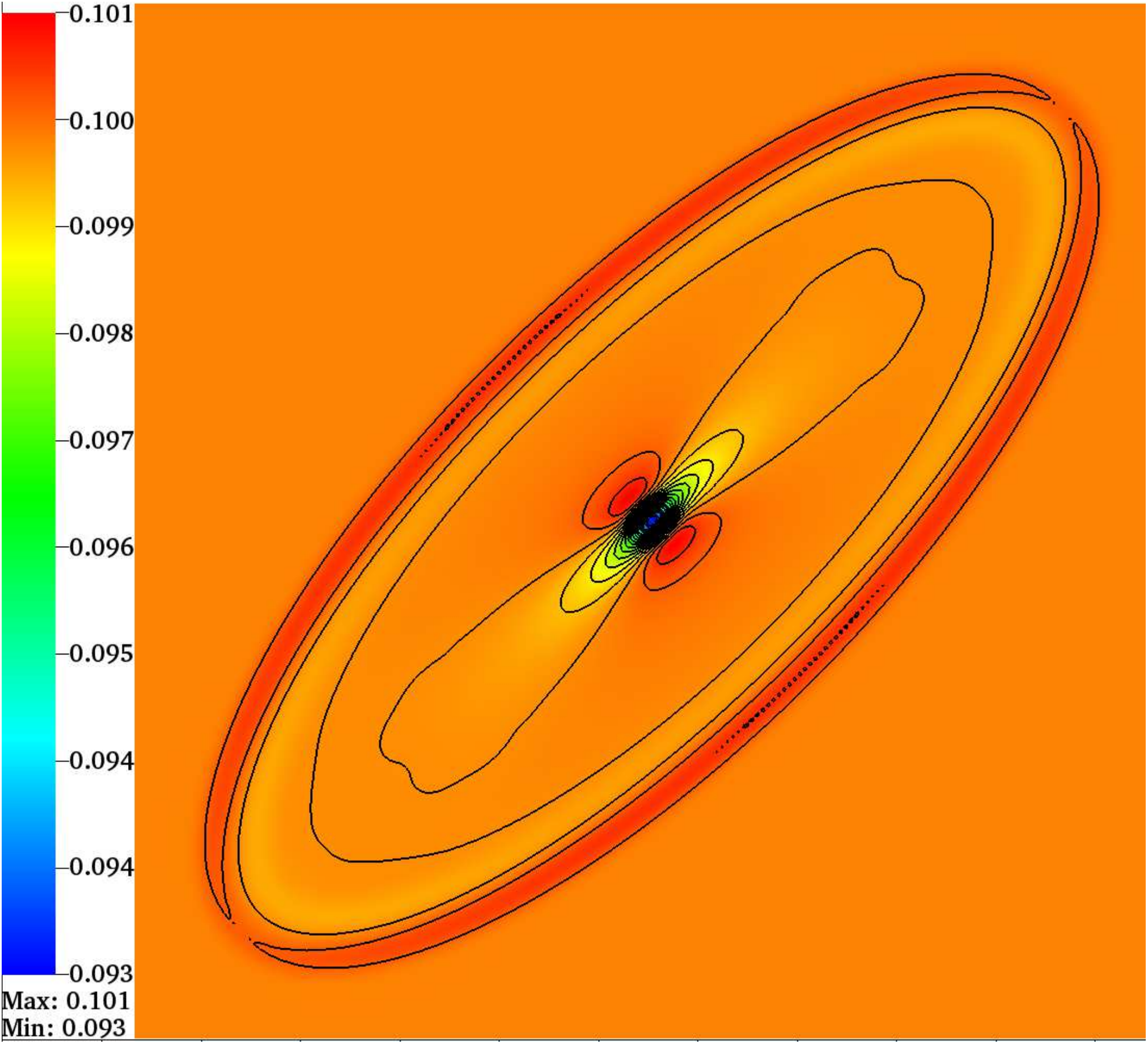} &
			\includegraphics[width=0.48\textwidth]{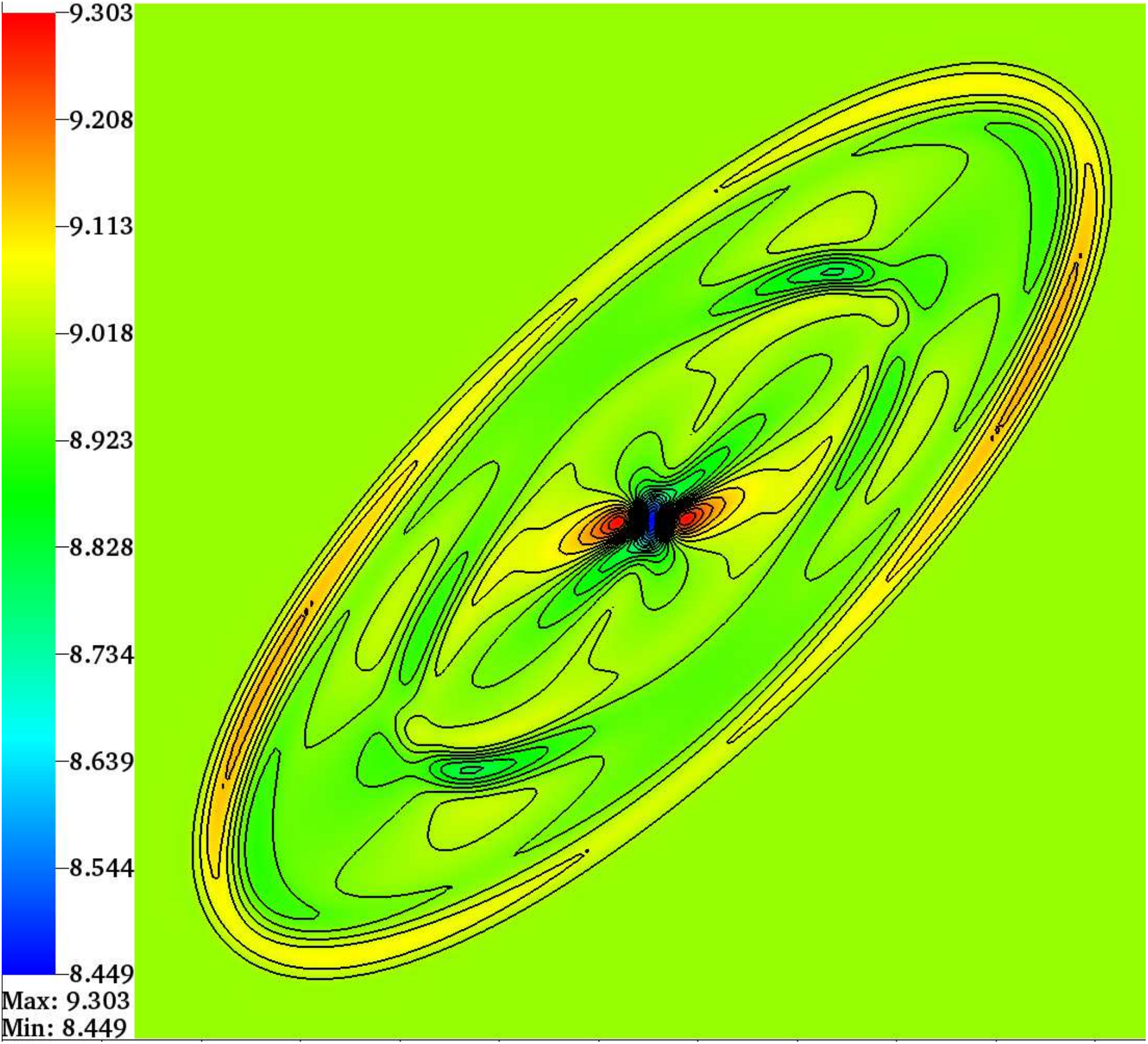} \\
			(a) Density & (b) Pressure component: $p_{11}$
		\end{tabular}
	\end{center}
	\caption{Plots of density and pressure component $p_{11}$ using 20 contour lines for Example \ref{ex:unp}.}
	\label{fig:s2up}
\end{figure}
\begin{figure}
	\begin{center}
		\begin{tabular}{cc}
			\includegraphics[width=0.48\textwidth]{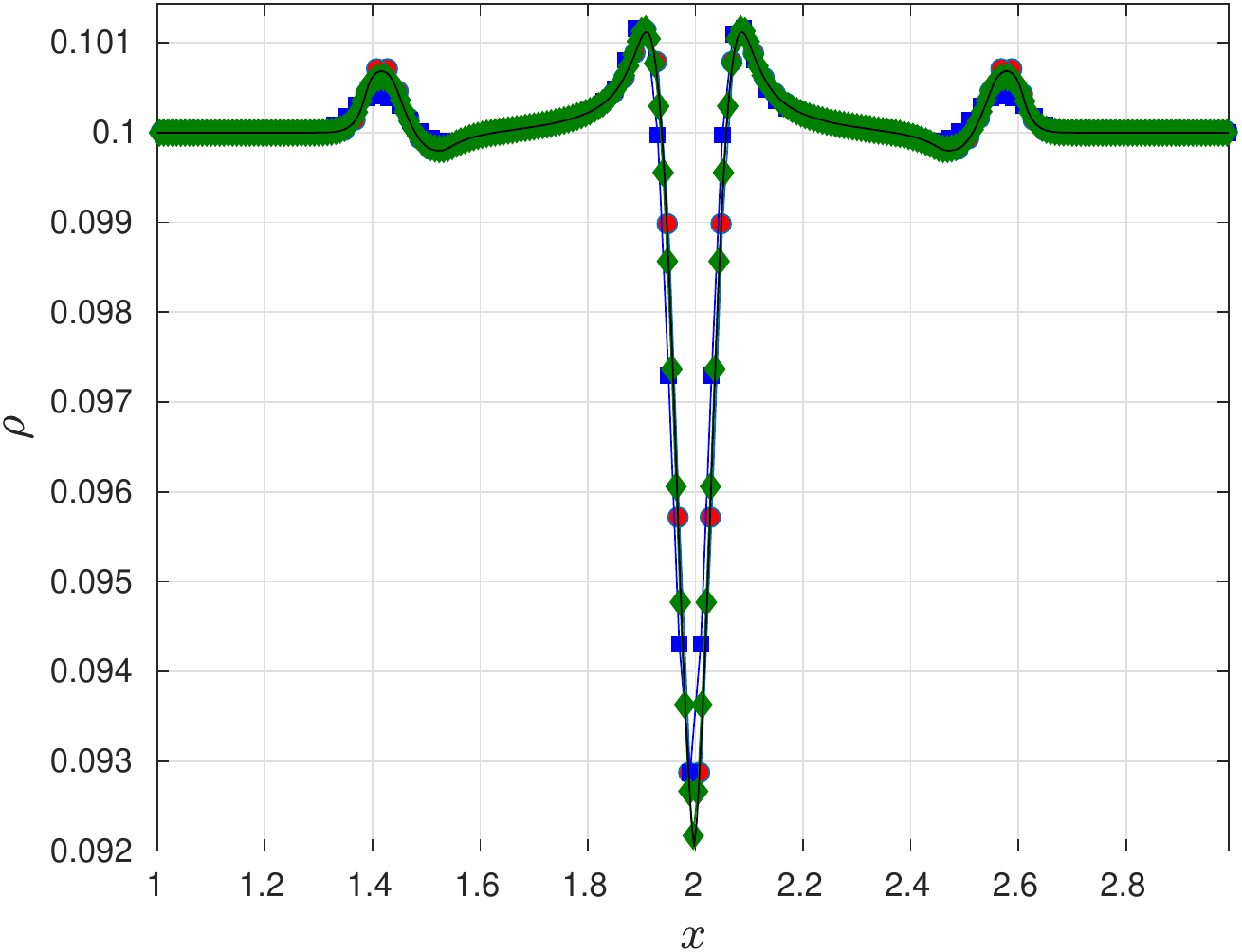} &
			\includegraphics[width=0.48\textwidth]{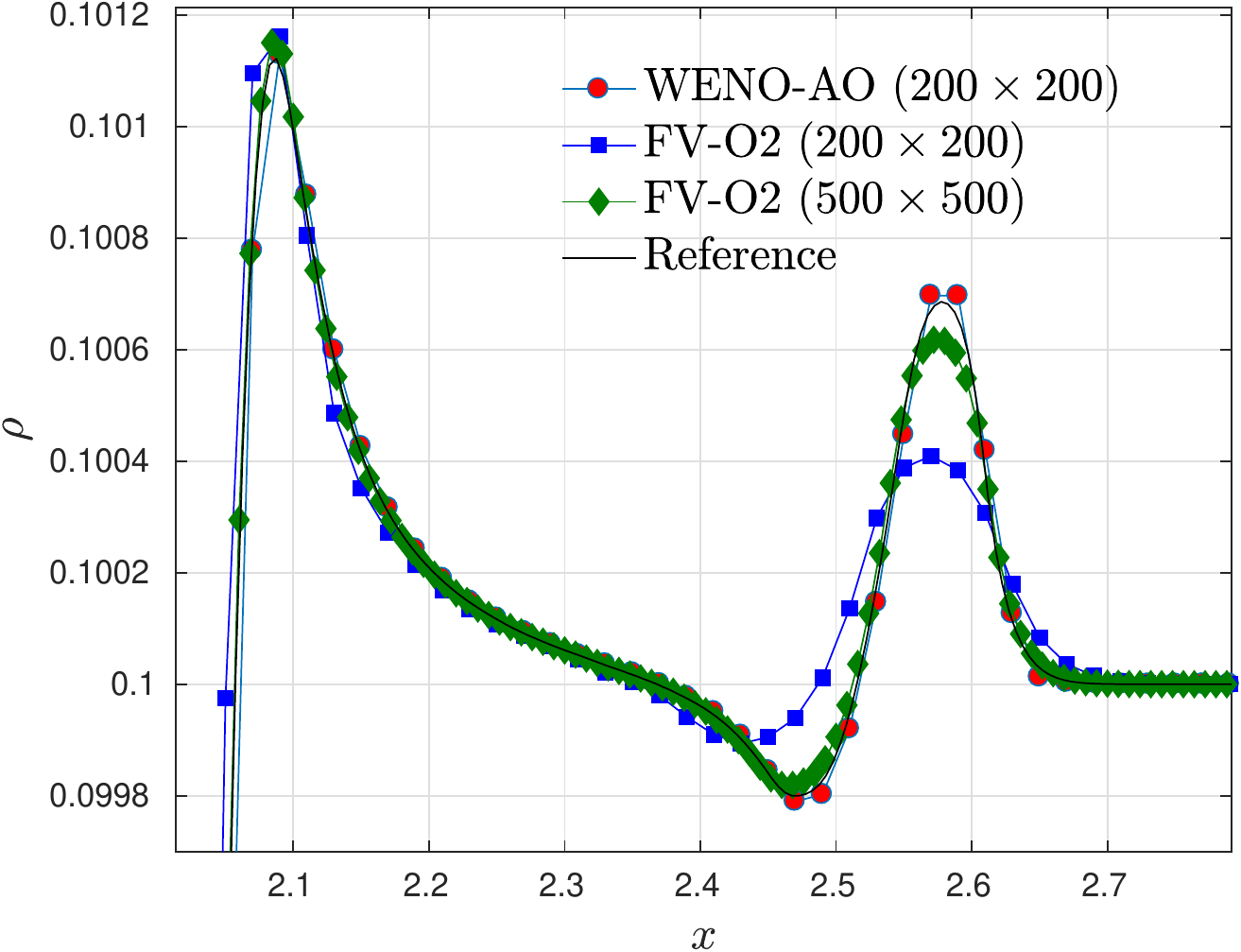} \\
			(a)  & (b) 
		\end{tabular}
	\end{center}
	\caption{One-dimensional cross-section plots of density variable along the line $x+y=4$ for Example \ref{ex:unp}.}
	\label{fig:s21d}
\end{figure}
\begin{figure}
\begin{center}
\begin{tabular}{cc}
\includegraphics[width=0.48\textwidth]{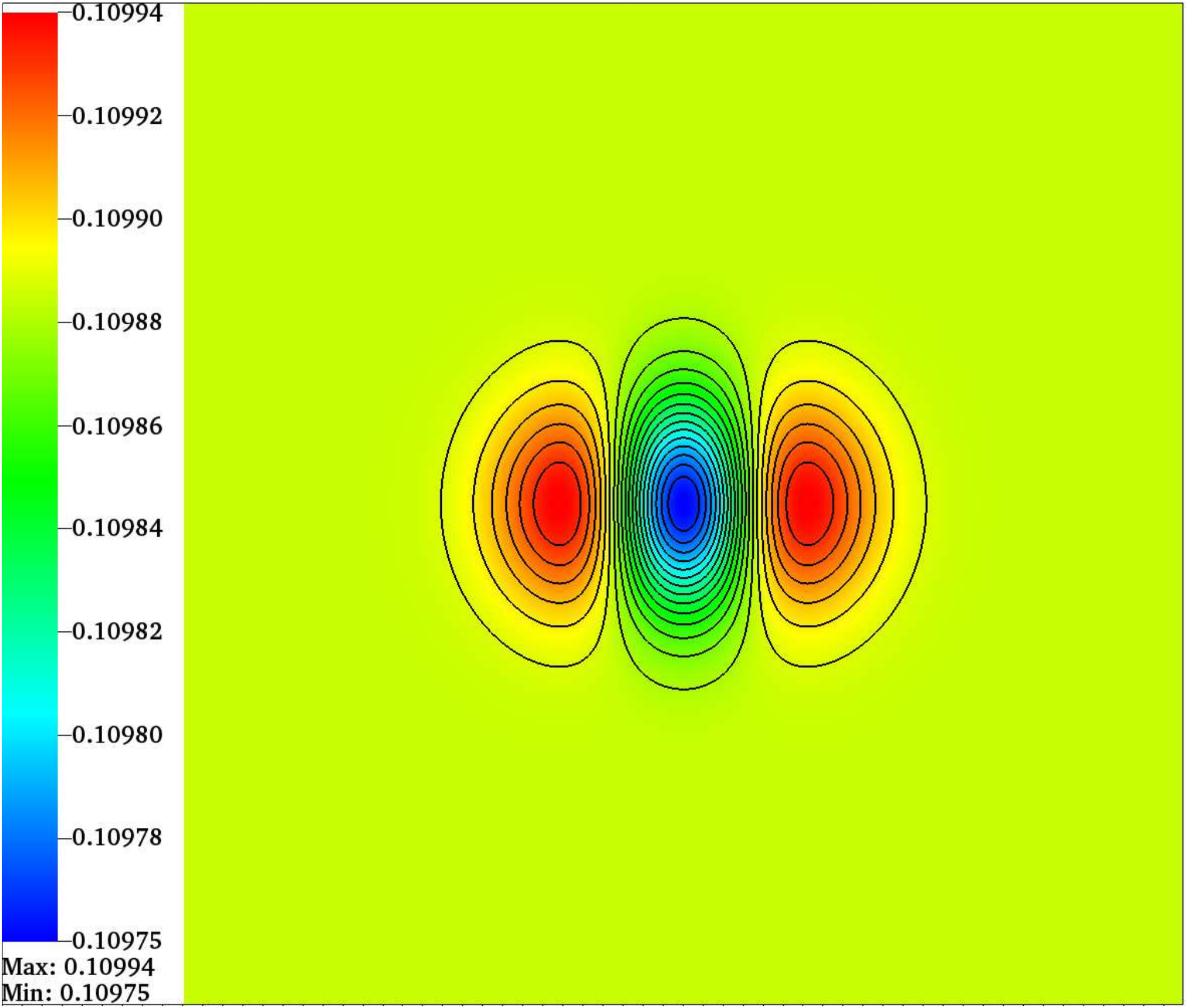} &
 \includegraphics[width=0.48\textwidth]{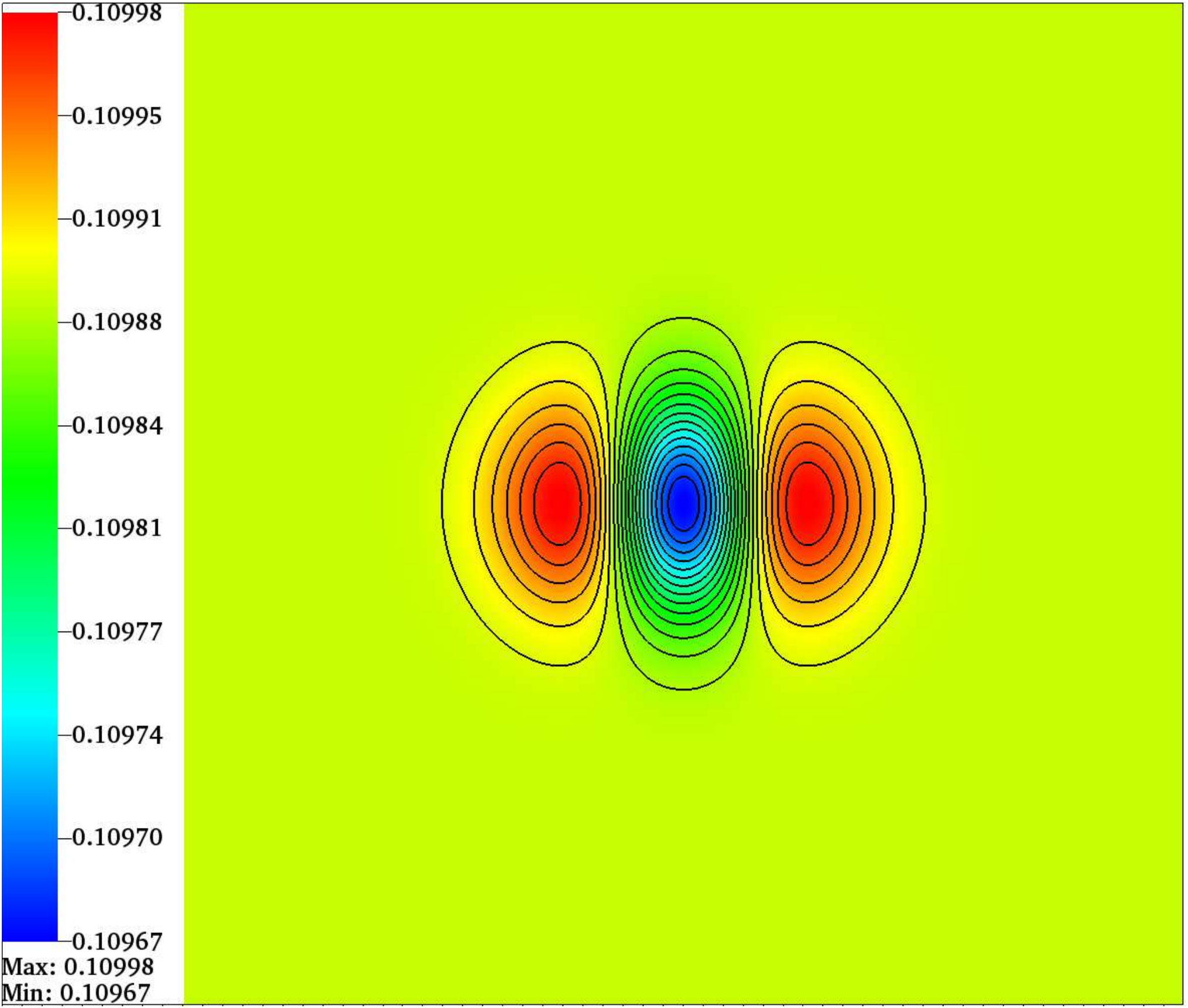} \\
 (a) Density $(v_T=0$) & (b) Density $(v_T=1$)
\end{tabular}
\end{center}
 \caption{Comparison of density for Example \ref{ex12} computed using WENO-AO scheme using $100\times 100$ mesh points. The number of contour lines taken to be 20. }
 \label{fig:ex12}
\end{figure}

\begin{figure}
\begin{center}
\begin{tabular}{cc}
\includegraphics[width=0.49\textwidth]{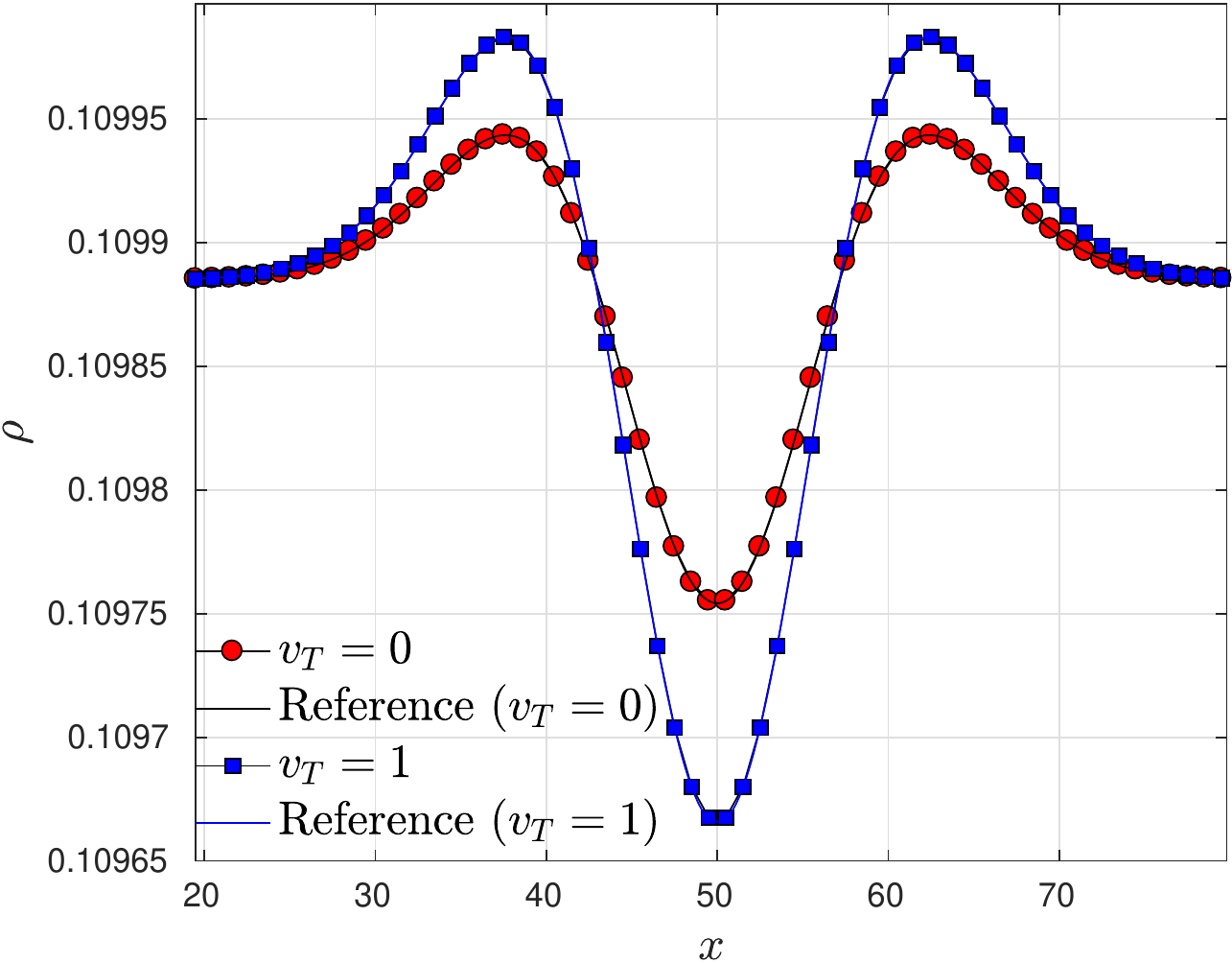} &
 \includegraphics[width=0.48\textwidth]{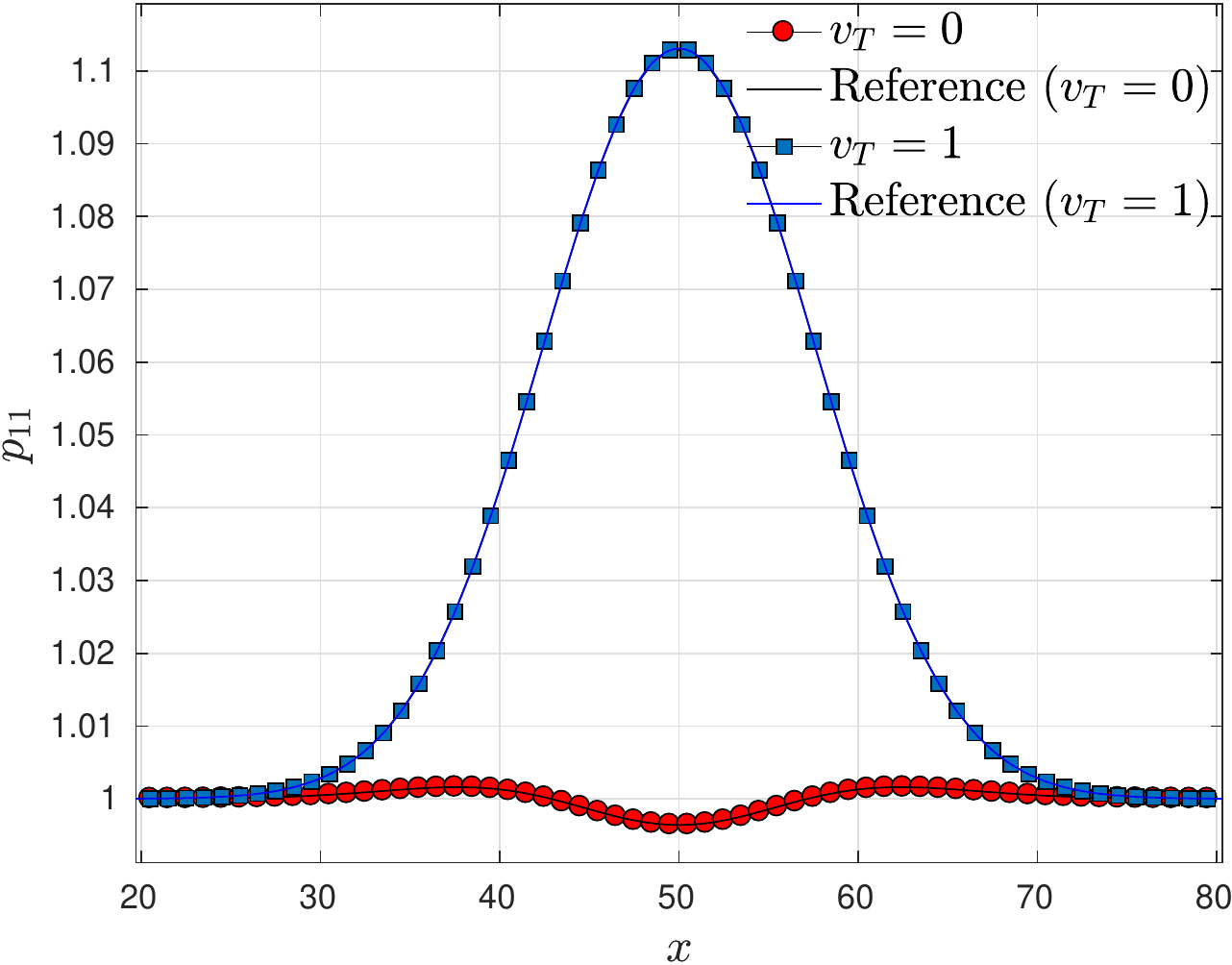} \\
 (a) $\rho$ & (b) $p_{11}$
\end{tabular}
\end{center}
 \caption{Comparison of density and pressure component $p_{11}$ for absorption coefficient $v_{T}=0$  and $v_{T}=1$ of Example \ref{ex12} along the line $y=50$ using $100\times100$ mesh points.}
 \label{fig:ex12cut}
\end{figure}

\begin{example}\label{ex12}{\rm (Realistic simulation in two-dimensions \cite{ber-etal_15a})
Consider a domain of $[0,100]\times[0,100]$ filled with plasma of density $0.109885$, initially at rest with pressure $p_{11}=p_{22}=1$ and $p_{12}=0$. This is excited with source term only in $x$-direction given by
\[
W(x,y)\equiv \exp\left( -\left( \frac{x-50}{10}\right) ^2-\left(\frac{y-50}{10} \right)^2\right)
\]
Here, we add the source term $v_T \rho W$ corresponding to the energy component. The term $v_T$ is known as {\em absorption coefficient}~\cite{ber-etal_15a, mee-kum_17a} and we will perform numerical experiments by choosing  different value of $v_T$. The outflow boundary conditions are considered on all boundaries. The numerical solutions are computed upto final time 0.5 for $v_T=0$ and $v_T=1$. In both cases, the positivity limiter is not active during simulation, since the solution remains away from the boundaries of the admissible set. In Figures \ref{fig:ex12} (a) and  \ref{fig:ex12} (b), we have shown the density variable 
for $v_T=0$ and $v_T=1$, respectively, using $100\times 100$ mesh points.  In Figure~\ref{fig:ex12cut}, the numerical results along the line $y=50$ are compared with the reference solution which is computed using finite volume scheme~\cite{meenaKumarFVM} with $500\times 500$ mesh points. In Figures \ref{fig:ex12cut} (a) and \ref{fig:ex12cut} (b), we have compared the density and pressure $p_{11}$ solutions respectively, for $v_{T}=0$ and $v_{T}=1$. We see that, as we increase the value of absorption coefficient $v_T$, the density decreases and pressure $p_{11}$ increases at the center of the domain, and these results are consistent with the results presented in \cite{ber-etal_15a, mee-kum_17a, mee-kum_17b, sen-kum_18a}.
}
\end{example}

\section{Summary and conclusions}
The Ten-Moment equations form a hyperbolic system of PDEs which can develop shocks and other discontinuous features. In this work, we have developed a positivity preserving finite difference WENO scheme for the Ten-Moment equations with source term. The key idea is to first establish that the scheme with Lax-Friedrich flux splitting is positive under a CFL condition and then use this to make the high order scheme positive by scaling the numerical flux of the high order scheme. The proof of the positivity property is independent of WENO reconstruction which makes this a very general approach. To demonstrate the idea, we consider three variants of WENO schemes, namely WENO-JS \cite{jia-shu_96a}, WENO-Z \cite{bor-etal_08d}, and WENO-AO \cite{bal-etal_16a}.  In addition, we solved the source ODE exactly and proved that the proposed solution of the source ODE is positivity preserving without any restriction on the time step. The discretization of the complete system with high order accuracy in time is achieved by an integrating factor strong stability preserving time integration scheme. The resultant scheme is shown to be positivity preserving for any given potential $W(x,y,t)$ in the source term. The numerical experiments show that design order of accuracy is achieved even in the presence of time dependent source terms.  The non-oscillatory nature is demonstrated on many test cases containing discontinuous solutions and we also show that the high order nature of the scheme gives better solutions compared to standard second order finite volume schemes. For some test cases, the positivity limiter is shown to lead to stable computations which otherwise break down due to loss of positivity. The positivity limiter induces a smaller CFL number and to improve the efficiency of the algorithm, we have implemented an adaptive CFL strategy. We have analyzed that the scheme with adaptive CFL is about 10-12 times faster than that with uniform CFL for 1-D test cases and almost 2-4 times faster in 2-D cases, while solution obtained using adaptive CFL is comparable with that of uniform CFL. In comparison of  WENO-AO scheme with WENO-JS, WENO-Z, and  finite volume~\cite{meenaKumarFVM} schemes, we conclude that the WENO-AO scheme resolves the shocks and high frequency waves more accurately. 

\section*{Acknowledgments} 
Rakesh Kumar would like to acknowledge funding support from the  National Post-doctoral Fellowship (PDF/2018/002621)  administered by SERB-DST,  India.

\bibliographystyle{IMANUM-BIB}

\bibliography{bib_weno}

\end{document}